\documentclass[11pt,english]{article}

\usepackage[margin = 2.1cm]{geometry}

\usepackage{amsthm}
\usepackage{amsmath}
\usepackage{amssymb}
\usepackage{mathtools}
\usepackage{graphicx}
\usepackage[pdftex,colorlinks,backref=page,citecolor=blue,bookmarks=false]{hyperref}
\usepackage{thm-restate}
\usepackage{enumitem}
\usepackage{xcolor}
\usepackage{subcaption}

\usepackage{caption}
\usepackage[square,sort,comma,numbers]{natbib} 
\usepackage{setspace}
\usepackage{cleveref}
\theoremstyle{plain}

\newtheorem*{theorem*}{Theorem}
\newtheorem{theorem}{Theorem}[section]
\crefname{theorem}{Theorem}{Theorems}
\Crefname{theorem}{Theorem}{Theorems}

\newtheorem*{lemma*}{Lemma}
\newtheorem{lemma}[theorem]{Lemma}
\crefname{lemma}{Lemma}{Lemmas}
\Crefname{lemma}{Lemma}{Lemmas}

\newtheorem*{claim*}{Claim}
\newtheorem{claim}[theorem]{Claim}
\crefname{claim}{Claim}{Claims}
\Crefname{claim}{Claim}{Claims}

\crefname{proposition}{Proposition}{Propositions}
\Crefname{proposition}{Proposition}{Propositions}

\crefname{corollary}{Corollary}{Corollaries}
\Crefname{corollary}{Corollary}{Corollaries}

\newtheorem{conjecture}[theorem]{Conjecture}
\crefname{conjecture}{Conjecture}{Conjectures}
\Crefname{conjecture}{Conjecture}{Conjectures}

\newtheorem{question}[theorem]{Question}
\crefname{question}{Question}{Questions}
\Crefname{question}{Question}{Questions}

\newtheorem{observation}[theorem]{Observation}
\crefname{observation}{Observation}{Observations}
\Crefname{observation}{Observation}{Observations}

\crefname{example}{Example}{Examples}
\Crefname{example}{Example}{Examples}

\newtheorem{algorithm}[theorem]{Algorithm}
\crefname{algorithm}{Algorithm}{Algorithm}
\Crefname{algorithm}{Algorithm}{Algorithm}

\theoremstyle{definition}
\newtheorem{problem}[theorem]{Problem}
\crefname{problem}{Problem}{Problems}
\Crefname{problem}{Problem}{Problems}

\newtheorem{definition}[theorem]{Definition}
\crefname{definition}{Definition}{Definitions}
\Crefname{definition}{Definition}{Definitions}

\theoremstyle{remark}

\crefname{remark}{Remark}{Remarks}
\Crefname{remark}{Remark}{Remarks}

\usepackage{xpatch}
\xpatchcmd{\proof}{\itshape}{\normalfont\proofnamefont}{}{}
\newcommand{\proofnamefont}{}
\renewcommand{\proofnamefont}{\bfseries}

\usepackage{framed}

\newcommand{\remove}[1]{}

\newcommand{\ceil}[1]{
    \lceil #1 \rceil
}

\newcommand{\eps}{\varepsilon}

\def \cP {\mathcal{P}}

\def \cH {\mathcal{H}}
\def \cI {\mathcal{I}}
\def \cJ {\mathcal{J}}

\def \cC {\mathcal{C}}
\renewcommand{\Pr}{\mathbb{P}}
\newcommand{\Ex}{\mathbb{E}}

\DeclareMathOperator{\cand}{cand}
\DeclareMathOperator{\dist}{dist}

\DeclareMathOperator{\red}{red}
\DeclareMathOperator{\blue}{blue}

\DeclareMathOperator{\diam}{diam}
\DeclareMathOperator{\exc}{exc}
\def \Fexc {F_{\exc}}
\def \Fexcs {F_{\exc}^*}
\def \Fexcss {F_{\exc}^{**}}
\DeclareMathOperator{\la}{la}

\title{Decomposing cubic graphs into isomorphic linear forests}

\author{
    Gal Kronenberg\thanks{Mathematical Institute, University of Oxford, Andrew Wiles Building, Radcliffe Observatory Quarter, Woodstock Road, Oxford, United Kingdom.   E-mail: \texttt{kronenberg}@\texttt{maths.ox.ac.uk}. Research supported by the European Union’s Horizon
    2020 research and innovation programme under the Marie Sk\l odowska Curie grant agreement No. 101030925. }
   \and
	    Shoham Letzter\thanks{
		Department of Mathematics, 
		University College London, 
		Gower Street, London WC1E~6BT, UK. 
		Email: \texttt{s.letzter}@\texttt{ucl.ac.uk}. 
		Research supported by the Royal Society.
    }
    \and
    Alexey Pokrovskiy\thanks{
		Department of Mathematics, 
		University College London, 
		Gower Street, London WC1E~6BT, UK. 
		Email: \texttt{a.pokrovskiy}@\texttt{ucl.ac.uk}.     }
   \and
   Liana Yepremyan\thanks{
   Department of Mathematics, Emory University, 
   Atlanta, Georgia, 30322, US. Email: \texttt{lyeprem}@\texttt{emory.edu}.
   Research supported by NSF Award 2247013: Forbidden and Colored Subgraphs
   }
}

\begin{document}

\date{}
\maketitle

\begin{abstract}

\setlength{\parskip}{\medskipamount}
\setlength{\parindent}{0pt}
\noindent

A fundamental problem in graph decompositions seeks to partition the edge set of a given graph into few similar and simple subgraphs, under certain divisibility conditions.  In 1987 Wormald conjectured that the edges of every cubic graph on $4n$ vertices can be partitioned into two isomorphic linear forests. We prove this conjecture for large connected cubic graphs. Our proof uses a wide range of probabilistic tools in conjunction with intricate structural analysis, and introduces a variety of  local recolouring techniques.
\end{abstract}

\section{Introduction} \label{sec:intro}

Many problems in graph theory seek to decompose the edges of a given graph into simpler pieces. A fundamental example seeks for a decomposition of the edges into matchings, that is, a proper edge-colouring of a graph. 
According to a well-known result of Vizing~\cite{vizing1964estimate} from 1964, the \emph{chromatic index} of a graph $G$, denoted $\chi'(G)$ and defined to be the minimum number of matchings needed to decompose the edges of a simple graph $G$, is either $\Delta(G)$ or $\Delta(G)+1$.

In this paper we are interested in decompositions of graphs into few \textit{isomorphic} pieces. That is, we would like to decompose a graph into subgraphs that are not only simple, but also isomorphic, which clearly adds an additional layer of difficulty.  There is a lot of literature on isomorphic decompositions where the pieces of the graph are predetermined and the host graph is well structured. Famous examples include decompositions of complete graphs, random graphs, or quasirandom graphs into Hamilton cycles, perfect matchings, or specific trees (see~\cite{keevash2014existence, glock2023existence} and the references therein). 

Note that if we require the isomorphic pieces to be matchings then this is easy to do. Indeed, one can use Vizing's theorem to properly edge-colour any graph with maximum degree $\Delta$ using $\Delta+1$ colours. Then, noting that the union of two matchings can be decomposed into two matchings whose size differs by at most 1, one can sequentially rebalance pairs of colour classes whose sizes differ by at most 2, until no such pairs remain. Assuming that the number of edges in the graph is divisible by $\Delta + 1$, this results in a decomposition into $\Delta+1$ equally-sized matchings.
In other words, every $\Delta$-regular graph whose number of edges is divisible by $\Delta + 1$ can be decomposed into at most $\Delta+1$ isomorphic matchings.

One can ask if we can find decompositions into a smaller number of pieces that are still relatively simple. Natural candidates of such pieces are trees or forests. In this direction, for 3-regular graphs, in 1987  Wormald conjectured \cite{wormald1987problem} the following, asking for an \textit{isomorphic decomposition} into two linear forests, where a \emph{linear forest} is a forest whose components are paths.

\begin{conjecture}[Wormald]\label{WormaldConj}
    The edges of any cubic graph, whose number of vertices is divisible by $4$, can be $2$-edge-coloured such that the two colour classes are isomorphic linear forests.
\end{conjecture}

We essentially settle this conjecture as stated below.

\begin{restatable}{theorem}{thmMain} \label{thm:main}
    Let $G$ be a connected cubic graph on $n$ vertices, where $n$ is large and divisible by $4$. Then there is a red-blue colouring of the edges of $G$ whose colour classes span isomorphic linear forests.
\end{restatable}

Note that the host graphs in our setting will be arbitrary regular graphs, so not particularly well-structured, and the isomorphic pieces are not predetermined, thus we will require different approaches as the methods commonly used in the quasi-random settings will not be applicable. 
We further note that our proof can be modified to relax the requirement of $G$ being connected to $G$ having at least one large connected component, meaning a component of size at least a certain (large) universal constant. This strengthening is described in the concluding section, \Cref{sec:conc}.   In addition, our proof can be modified to show that if the number of vertices is not divisible by 4 then we can guarantee
that the two linear forests are isomorphic up to removing or adding an edge. Specifically, we can guarantee that the
only difference in component structure is that the red graph has two extra edge components, and the
blue graph has one extra component which is a path of length 3. 

Prior to our work, Wormald's Conjecture was known to be true only for some very specific cubic graphs. It was proved for Jaeger graphs in work of Bermond, Fouquet, Habib and P\'eroche \cite{fouquet2007linear} and Wormald~\cite{wormald1987problem}, and for some further classes of cubic graphs by Fouquet, Thuillier, Vanherpe and Wojda \cite{fouquet2009isomorphic}. 
On a related note, in 1984, Bermond, Fouquet, Habib, and P\'eroche \cite{bermond1984linear} conjectured that not only can every cubic graph be decomposed into two linear forests as proven by the validity of the linear arboricity conjecture for cubic graphs, but it can be done in such a way that every path in each of the two linear forests has length at most $5$. In 1996 Jackson and Wormald \cite{jackson1996linear} proved this conjecture with the constant $18$ instead of $5$. Later this was improved by Aldred and Wormald \cite{aldred1998more} to 9, and the conjecture was finally resolved in 1999 by Thomassen  \cite{thomassen1999two}.

\begin{theorem}[Thomassen] \label{thm:thomassen}
    Any cubic graph can be $2$-edge-coloured such that every monochromatic component is a path of length at most five.
\end{theorem}

Thomassen's beautiful result is an important building block for our proof of our main result.

It is worth mentioning that, although our proof of Conjecture~\ref{WormaldConj} is quite long and involved, we can prove the following approximate version of Wormald’s conjecture with a fairly short and neat proof (note that here we do not require divisibility or connectivity).

\begin{theorem} \label{thm:main-approx}
    Let $G$ be a cubic graph on $n$ vertices, where $n$ is large. Then $G$ can be red-blue coloured so that all monochromatic components are paths of length $O(\log n)$, and the numbers of red and blue components which are paths of length $t$ differ by at most $n^{2/3}$, for every $t \ge 1$.
\end{theorem}

This approximate version is proved in \Cref{sec:approx} (it is a special case of \Cref{lem:MainCol}), and the rest of the paper (\Cref{sec:Exact,sec:partial-colouring,sec:gadgets,sec:GeodesicNoCommonNghbs,sec:GeodesicWithCommon}) is dedicated to proving \Cref{thm:main} using \Cref{lem:MainCol}.

It is natural to ask what happens for higher degrees. To state a reasonable question in this direction, we feel an interlude to a related conjecture, namely the linear arboricity conjecture, is necessary, where we no longer demand the forests to be isomorphic. Perhaps the most famous work in this direction is the Nash-Williams theorem \cite{nash1964decomposition} from 1964 regarding the \emph{arboricity} of a graph $G$, namely the minimum number of forests needed in order to decompose the edges of $G$. As an interesting interpolation between decompositions into matchings and into forests, in 1970 Harary \cite{harary1970covering} suggested to study the minimum number of linear forests needed to decompose the edges of a given graph $G$. This parameter is called the \emph{linear arboricity} of $G$ and is denoted $\la(G)$. Clearly, $\la(G)\leq \chi'(G)$ for every graph $G$, as a matching is also a linear forest. A well-known conjecture by Akiyama, Exoo, and Harary from  1980 \cite{akiyama1980covering} states that by using linear forests instead of matchings, one can reduce the number of pieces needed by a factor of almost 2, namely, that $\la(G)\leq \ceil{\frac {\Delta(G)+1}{2}}$.  This conjecture is known in the literature as the \textit{linear arboricity conjecture}. It is easy to check that the upper bound is tight. Also, since every graph of maximum degree $\Delta$ can be embedded in a $\Delta$-regular graph, an equivalent statement of the conjecture, which is more commonly studied, is the following.

\begin{conjecture}[The linear arboricity conjecture; Akiyama, Exoo, and Harary \cite{akiyama1980covering}] Every $\Delta$-regular graph $G$ satisfies $\la(G)= \lceil\frac{\Delta+1}2\rceil$.
\end{conjecture}

The case of cubic graphs (that is, $\Delta=3$), where the conjecture predicts $\la(G)=2$, was verified already in the paper of Akiyama, Exoo, and Harary \cite{akiyama1980covering}, and later a shorter proof was presented by Akiyama and Chv\'atal \cite{akiyama1981short} in 1981. (It is worth mentioning that we use  Akiyama's and Chv\'atal's switching techniques as a basis and build more complicated switching blocks inspired by them.) In 1988 Alon~\cite{alon1988linear} showed that the linear arboricity conjecture holds asymptotically, meaning that if $G$ has maximum degree at most $\Delta$ then $\la(G)= \Delta/2+o(\Delta)$; his $o(\Delta)$ term was of order $O\left(\Delta \frac{\log{\log{\Delta}}}{\log{\Delta}}\right)$.  In the same paper Alon also showed that the conjecture holds for graphs with high girth, that is, when the girth of the graph is $\Omega(\Delta)$. Subsequently, using Alon's approach, the linear arboricity conjecture was proved to hold for random regular graphs by Reed and McDiarmid \cite{mcdiarmid1990linear} in 1990. Alon and Spencer~\cite{alon2016probabilistic} in 1992 improved Alon's asymptotic error term to $O(\Delta^{2/3}(\log{\Delta})^{1/3})$. After almost thirty years,  this error term was very recently improved; at first to $O(\Delta^{2/3-\varepsilon})$, for some absolute constant $\varepsilon>0$, by Ferber, Fox and Jain~\cite{ferber2020towards}, and finally to $O(\sqrt{\Delta} (\log{\Delta})^4)$ by Lang and Postle \cite{lang2020improved}. All these approximate results make use of classical probabilistic tools, specifically, R\"{o}dl's nibble and Lov\'asz's local lemma. Apart from approximate results and random regular graphs, the linear arboricity conjecture  was verified for $d\in \{3,4\}$ \cite{akiyama1981short,akiyama1980covering,akiyama1981covering}, $d\in \{5, 6\}$ \cite{enomoto1981linear,peroche1982partition,tomasta1982note}, $d =8$ \cite{enomoto1984linear}, $d=10$ \cite{guldan1986linear}, as well as for various other families of graphs such as complete graphs, complete bipartite graphs, trees, planar graphs, some $k$-degenerate graphs and binomial random graphs~\cite{akiyama1980covering,akiyama1981covering,enomoto1984linear,guldan1986linear,wu1999linear,wu2008linear,ChenHaoYu, glock2016optimal}.  However, in its full generality the conjecture remains  open.

 We are finally ready to pose the following generalization of Wormald's conjecture to higher degrees which also generalizes the linear arboricity conjecture.

\begin{conjecture}\label{conj:wormaldgeneralized}
    Let $d \ge 3$ be an integer. Every large connected $d$-regular graph, whose number of edges is divisible by $\ceil{(d+1)/2}$ can be decomposed into $\ceil{(d+1)/2}$ isomorphic linear forests.
\end{conjecture}

Theorem~\ref{thm:main} essentially settles this conjecture for the $\Delta=3$, yielding the first tight result for isomorphic (non-trivial) decompositions for general $\Delta$-regular graphs.  
We believe that the tools we developed in this paper could provide a new avenue for progress towards the linear arboricity conjecture as well.

\section{Notation}
 Given a graph $G$ and an edge-colouring $\chi$ of $G$, will denote by $b_{G,\chi}(P_t)$ and $r_{G,\chi}(P_t)$ the number of blue and red paths of length $t$ in $G$. For two paths $P,Q$ in a graph $G$, we denote by $d_G(P,Q)$  the length of the shortest path between them. When $G$ or $\chi$ are clear from the context, we will skip the corresponding subscript. In a graph $G$, a \emph{geodesic}  is defined to be the shortest path between some two vertices.  Given an edge-coloured graph $G$, we call a vertex \emph{monochromatic} if all edges incident to it have the same colour.  In this paper $\log n$ is the natural logarithm. In a graph $G$, we say a subgraph $H\subseteq G$ \emph{touches} an edge $e$ if $e$ has exactly one endpoint in $H$.

\section{Proof overview}

We now give a detailed sketch of our proof. 
The main idea is to first colour a small part of the graph in a very structured way, so that it can later be used to make small fixes to the full colouring, and then colour the rest of the graph in a random way, while guaranteeing that the  monochromatic components are (not too long) paths. Using the randomness, we show that the two colour classes are almost isomorphic. We then use the pre-coloured graph to fix the imbalance and make the colour classes isomorphic, thus completing the proof.   The starting point of the proof of \Cref{thm:main} is \Cref{thm:main-approx}, and as such, we now briefly sketch the proof of the latter, approximate theorem.
One natural way to split a cubic graph into almost isomorphic parts is to colour each edge either red or blue uniformly at random and independently. However, the two colour classes will have many vertices of degree $3$, and will thus be far from being a linear forest. 
To get a colouring which is both balanced with high probability, and whose colour classes are ``close'' to being linear forests, instead of using the random colouring described above, we use a semi-random colouring scheme whose colour classes have maximum degree at most 2. Before describing how we do so, we point out that this is not the end of the road; we still need to eliminate monochromatic cycles and long monochromatic paths. Doing so requires some clever tricks, which achieve these goals using local steps, without harming the nice properties we have achieved, and without introducing too many dependencies.

For the semi-random colouring, we start with a 2-colouring of $G$ where each monochromatic component is a path, and recolour each path randomly and independently, reminiscently of Kempe changes.
For technical reasons (related to the concentration inequality which we apply at the end), we need every monochromatic component in the initial colouring to be a \emph{short} path. 
This can be achieved through Theorem~\ref{thm:thomassen}.

We will, in fact, prove a more general statement than the one in \Cref{thm:main-approx} (see \Cref{lem:MainCol}) that is applicable also for cubic graph with a given partial colouring with nice properties. This will allow us to pre-colour small parts of the graph in advance and obtain an ``almost balanced'' colouring which preserves large parts of the pre-colouring.
The pre-coloured subgraph will contain many ``gadgets'', which are small subgraphs with two colourings with the property that, by replacing one colouring by the other, the difference between the number of red and blue components which are paths of a given length $\ell$ decreases by 1, and the numbers of longer monochromatic components does not change. Since \Cref{lem:MainCol} guarantees that many gadgets for each relevant $\ell$ survive, we can use them straightforwardly to correct the imbalance between red and blue component counts.\footnote{Actually, since we can only guarantee gadgets of length $o(\log n)$, and the components resulting from \Cref{lem:MainCol} can be longer, we need a separate argument to balance intermediate lengths.}

Finding gadgets that interact well with the rest of the graph is a pretty subtle process, as we need to consider not only the various different forms the structure of the gadget itself can take, but we also need to consider its neighbourhood, e.g.\ to make sure that its colouring can be extended to a colouring of the whole graph where monochromatic components are paths. This is particularly challenging when the graph contains many short cycles and to overcome this we need sophisticated ways of tracking a certain colouring process around a geodesic (see \Cref{sec:GeodesicWithCommon}).

 

Next we give a more detailed sketch of the proof. The rest of the paper will be structured as follows. In \Cref{sec:approx} we will state and prove \Cref{lem:MainCol} about the existence of an almost-balanced colouring. In section \Cref{sec:Exact} we will state \Cref{lem:partialColMainLemma} about the existence of a good partial colouring, and in \Cref{sec:finish} will show how to use \Cref{lem:MainCol} and \Cref{lem:partialColMainLemma} in order to prove \Cref{thm:main}. In \Cref{sec:gadgets,sec:GeodesicNoCommonNghbs,sec:GeodesicWithCommon,sec:partial-colouring} we will prove \Cref{lem:partialColMainLemma}.

\subsection{The approximate result} \label{subsec:approx}
While this is not the first step in the process, we now describe an approximate solution of Wormald's conjecture, and later explain how to obtain an appropriate partial pre-colouring. For the purpose of this explanation, our task is to red-blue colour a (large) given cubic graph such that the colour classes are ``almost" isomorphic, that is, the difference between the number of red and blue components isomorphic to a path of length $t$ is small, for all $t$. For this, we wish to colour the graph randomly, while maintaining certain properties such as the monochromatic components being paths.  

Our random colouring will consist of three steps.
For the first step, we use Thomassen's result (\Cref{thm:thomassen}) about the existence of a $2$-colouring where each monochromatic component is a path of length at most $5$\footnote{The exact constant here is not important. We could even make do with a polylogarithmic bound on the lengths, and, additionally, we can allow for even cycles, but not odd ones.}; we denote the two colours here by purple and green.
The first random step colours each purple or green component by one of the two possible alternating red-blue colourings, chosen uniformly at random and independently.
Notice that this random red-blue colouring of $G$ has no \emph{monochromatic} vertices meaning a vertex who is incident to three edges of the same colour. Moreover, the symmetry between the colours and the bound on the lengths of purple and green paths would allow us to show that the colours are, in some sense, close to being isomorphic. However, there is nothing preventing the appearance of cycles, and we could not rule out the existence of very long monochromatic paths. This is the more technical part for applying concentration inequalities and doing the final re-balancing.

This brings us to the second random step, which will be broken into two parts, and whose purpose is to eliminate monochromatic cycles. 
Here, we first do something very intuitive: we simply flip the colour of one edge $e_C$ of each monochromatic cycle $C$, choosing the edge $e_C$ uniformly at random and independently. Unsurprisingly, while this breaks all monochromatic cycles that existed before the first step, new monochromatic cycles can appear. Luckily, a small fix saves us and eliminates all monochromatic cycles. The fix essentially consists of re-swapping the colour of $e_C$ for some of the originally monochromatic cycles $C$, and swapping the colour of one of the two neighbouring edges of $e_C$ in $C$, while  again making choices randomly and independently.

The next and final random step is designed to break ``long'' paths. After this process monochromatic paths have length of order $O(\log{n})$. Here the idea is less intuitive. We let each monochromatic path $P$ choose one of the possibly four \emph{boundary edges}, that is the edges of the opposite colour that touch an end of $P$ uniformly at random and independently. Then, for each edge $e$ that was chosen by two paths, we flip the colour of $e$ with probability $1/2$, independently. This somewhat strange process has several benefits: first, with high probability, it swaps an edge of each monochromatic path of length at least $1000 \log n$; second, it creates no monochromatic cycles; and, third, it does not allow more than two monochromatic paths to join up (more precisely, monochromatic paths in the new colouring have at most one edge whose colour was swapped). 

Finally, we analyse the resulting random colouring, and show that its colour classes are almost isomorphic. The proof of this is conceptually simple: observing that the distribution of the red and blue graphs is identical, thanks to all steps being performed simultaneously for red and blue, all we need to do is to prove that the number of red components isomorphic to $P_t$ is concentrated, for every $t \ge 1$. We accomplish this goal via McDiarmid's inequality (\Cref{thm:mcdiarmid}), using the independence of the various random decisions, as well as the fact that each decision has a small impact on the resulting graph. 

\subsection{From approximate to exact result}

Going from the approximate result to the exact result we wish to balance the number of red and blue components isomorphic to $P_t$, for every $t$. For this, the main idea is to pre-colour a small part of the graph, thereby creating many \emph{gadgets} that can later be used for balancing the number of paths in each length. This is done before finding the approximate partition.

We define a blue \emph{$\ell$-gadget} in a cubic graph $G$ to be a subgraph $H \subseteq G$ together with two red-blue edge-colourings $\chi, \chi'$ such that for every red-blue colouring of $G$ that extends $(H, \chi)$ and whose monochromatic components are paths, the monochromatic component counts change as follows if we switch the colouring of $H$ from $\chi$ to $\chi'$: $b_{\chi',G}(P_{\ell})= b_{\chi,G}(P_{\ell})-1$; $b_{\chi,G}(P_t)=b_{\chi',G}(P_t)$ with $t > \ell$; $b_{\chi,G}(P_t)$ with $t < \ell$ changes only slightly in $\chi'$; and $r(P_t)$ with any $t$ does not change. Such a gadget (and its counterpart with roles of colours reversed) will be used to equalise $r(P_t)$ and $b(P_t)$ iteratively, and we balance the path counts from the longest to the shortest, ensuring that the process terminates. Due to a simple counting argument, we only need to run the process until $\ell=3$. 

The most difficult and lengthy part of this paper is to find gadgets. 
We will not describe here the concrete structure of the gadgets we shall find (this can be found in \Cref{sec:gadgets}), but let us just say that these gadgets consist of a long path along with some pending edges and paths.
As such, it makes sense to look for gadgets in the neighbourhood of a ``long'' geodesic.
This is indeed what we do in \Cref{sec:GeodesicNoCommonNghbs,sec:GeodesicWithCommon}. We start off with a geodesic $P$, whose length is sufficiently larger than $\ell$, and then colour it and its neighbourhood appropriately. The difficulty is that, in addition to obtaining the gadget structure, we also need to make sure that other desirable properties are maintained, an obvious one being the non-existence of monochromatic vertices. For this, we define a colouring algorithm that colours the geodesic and some ball around it to guarantee this property. The fact that we started from a geodesic (rather than any arbitrary path) will allow us to claim that our colouring algorithm ends after a small number of steps while maintaining the desirable properties. The main obstruction to obtaining these desired property are short cycles. Indeed, our proof when no two vertices in the geodesic $P$ have a common neighbour outside of $P$ is much simpler (we treat it separately in \Cref{sec:GeodesicNoCommonNghbs}), and this property holds automatically if the girth of $G$ is at least $5$. If the girth is required to be at least $7$, we get an even simpler proof; indeed, with this girth assumption only the first case in \Cref{sec:GeodesicNoCommonNghbs} can occur.
For the rest of this overview, let us assume that we know how to find a gadget (with some desirable properties) in a small ball around a geodesic of length $c \cdot \ell$, for some constant $c$.

There are three things to keep in mind now. First, we need to find many well-separated geodesics of length $c \cdot\ell$, for all $\ell=3$ up to $\ell=O(\sqrt{\log{n}})$. Second, our strategy for the approximate result should be adapted so as to allow the underlying graph being partially pre-coloured because of the gadgets. And, third, we need many of the gadgets to survive at the end of the random process that we run to prove the approximate result.

For the first part, it is easy enough to show in any $n$-vertex cubic graph there are $n^{1 + o(1)}$ gadgets of length $L$ that are far enough from each other, for any $L = o(\log n)$. This already presents a small hurdle: our strategy for the approximate result yields colourings whose monochromatic components are paths of length at most $1000 \log n$, a bound which is too large for our gadgets to be applicable for balancing. Luckily, we can balance paths of length at least $c \sqrt{\log n}$, for some constant $c$, using a separate argument, not involving gadgets. Suppose there are $x$ many more blue paths of length $t$ than red ones, we let $\mathcal{I}$ to be the collection of these $x$ blue $P_t$-paths and break each one of them into segments of length $\Theta(\sqrt{\log n})$; denote the collection of segments by $\cI$. Now, using the Lov\'asz Local Lemma, we show that there is a choice of an edge $e_I$ from each segment $I \in \cI$, such that if the colours of these edges are flipped, then we create no cycles and no new monochromatic paths of length at least $c' \sqrt{\log n}$, for some appropriate constant $c'$. 

Regarding the second point, luckily the strategy we described in \Cref{subsec:approx} is amenable to the underlying graph being partially pre-coloured, as long as the following hold:
\begin{itemize}\item vertices incident to two pre-coloured edges see edges of both colours,
\item the coloured components in the pre-colouring are quite small,
\item and a version of Thomassen's result holds for the graph of uncoloured edges. 
\end{itemize}

The first property is the main property that we address when building gadgets. The second one holds naturally since we only colour edges in small neighbourhoods of geodesics of length $O(\sqrt{\log{n}})$. The third requires an additional argument, which can be found in \Cref{sec:partial-colouring}. With these properties at hand, the only modification of the process described in \Cref{subsec:approx} is in the first random step. Now we have some short purple and green paths, but also various small red-blue coloured components. We colour the purple and green components as before, and for each coloured component, we either keep it as is, or swap the colours of all its edges with the decision being made uniformly at random and independently. All other steps are carried out exactly as before. Notice in particular that, importantly, for any graph $H$, the red and blue copies of $H$ will have the same distribution.

Finally, for the third point it suffices to make sure that each gadget survives the whole random process intact, with decent probability. This is easy to check for the first and third steps, making the second step the only real obstruction. Recall that in the second step, for each monochromatic cycle $C$, we choose an edge $e_C$ in it, uniformly at random, and at the end of the second step either it or one of its neighbouring edges in $C$ has its colour swapped (and all other edges remain as they were). Thus, to guarantee that a gadget $H$ has a decent survival chance, we need to make sure that every monochromatic cycle $C$ in some red-blue colouring of $G$ extending $H$ has two consecutive edges not in $H$. This property is quite annoying to deal with; having to maintain it influences the arguments in \Cref{sec:GeodesicNoCommonNghbs,sec:GeodesicWithCommon,sec:partial-colouring}.


\section{Proof of approximate version}\label{sec:approx}

In this section we prove a weaker version of \Cref{WormaldConj}; see \Cref{lem:MainCol} below. Roughly speaking, we show, using a probabilistic approach, that the edges of the graph can be 2-coloured such that each monochromatic component is a linear forest, and the two graphs are ``almost isomorphic''. 

For the proof we will use McDiarmid's Inequality to show concentration. 
\begin{theorem}[McDiarmid's Inequality \cite{mcdiarmid1989method}] \label{thm:mcdiarmid}
    Let $X = (X_1, \ldots, X_n)$ be a family of independent random variables, with $X_i$ taking values in a set $A_i$, for $i \in [n]$.
    Suppose that $c > 0$ and $f$ is a real-valued function defined on $\prod_{i \in [n]}A_i$ that satisfies $|f(x) - f(x')| \le c$ whenever $x$ and $x'$ differ in a single coordinate. Then, for any $m \ge 0$,
    \begin{equation*}
        \Pr\left(f(X) - \Ex[f(X)] \ge m \right) \le 2 \exp\left(\frac{-2m^2}{c^2n}\right).
    \end{equation*}
\end{theorem}


In order to use the ``almost balanced'' colouring and turn into a decomposition into two isomorphic linear forests, we will pre-colour a small part of the graph that will be used later to balance the colouring. However, we need to make sure that the pre-coloured graph will combine nicely with the almost balanced colouring. For this purpose we define the following. 

\begin{definition}\label{def:GoodPartialCol}
  Let $G$ be a cubic graph of order $n$. A subgraph $G_0 \subseteq G$ whose edges are red-blue coloured is called \textit{extendable} if all of the following hold.  Let $G_0(R)$ and $G_0(B)$ be the red and blue subgraphs of $G_0$, respectively.    
  \begin{enumerate}[label = \rm(E\arabic*)]
    \item \label{itm:extend-degree}
        Every vertex $v$ with $d_{G_0}(v) \ge 2$ satisfies $d_{G_0(R)}(v), d_{G_0(B)}(v) \ge 1$.
    \item \label{itm:extend-compt-size}
        The size of each connected component in $G_0$ is $O(\sqrt{\log n})$ and every two components are of distance at least 10 from each other.
    \item \label{itm:extend-cycle-technical}
        Every cycle $C$ in $G$ which does not contain both a red and a blue edge from the same component of $G_0$ has two consecutive edges outside of $G_0$.
    \item \label{itm:extend-Hs}
        The graph $G\setminus G_0$ has an edge-decomposition into graphs $H_1, H_2, H_3$ such that $H_2$ and $H_3$ are vertex disjoint and have the following properties.
    \begin{itemize}
        \item 
            $H_1$ is a subdivision of a simple graph $F$ all of whose vertices have degrees 1 or 3, where each edge in $F$ is subdivided $O(\sqrt{\log n})$ times to create $H_1$.
        \item 
            $H_2$ is a disjoint union of even cycles of length $O(\sqrt{\log n})$, each cycle touches at most two edges in $H_1$.
        \item 
            $H_3$ is vertex-disjoint collection of subdivisions of  multi-edges with multiplicity 3, where each edge is subdivided at most $O(\sqrt{\log n})$ times. 
    \end{itemize}
    \end{enumerate}
  \end{definition}

We note that  $G_0 = \emptyset$  is always extendable.
The purpose of this section is to prove the following lemma, showing that any extendable partial pre-colouring $G_0$ in a cubic graph $G$ can be extended to a red-blue colouring of $G$ where the two colour classes are almost isomorphic linear forests.

\begin{lemma}\label{lem:MainCol}
Given a cubic graph $G$ of order $n$, where $n$ is sufficiently large, and a red-blue coloured subgraph $G_0$ which is extendable according to \Cref{def:GoodPartialCol}, there exists a red-blue colouring of $G$ such that.
\begin{enumerate}[label = \rm(Q\arabic*)]
    \item \label{itm:approx-only-short-paths}
        Every monochromatic component is a path of length $O(\log n)$.
    \item \label{itm:approx-no-spiders}
        There is no red path of length at most $(\log n)^{2/3}$ that intersects at least $100 \sqrt{\log n}$ blue paths of length at least $100 \sqrt{\log n}$, and similarly with the roles of red and blue reversed.
    \item \label{itm:approx-concentration}
        $|r(P_t) - b(P_t)| \le n^{2/3}$ for every $t \ge 1$.
    \item \label{itm:approx-gadget-survive}
        For every red-blue coloured subcubic graph $H$, if $G_0$ has at least $n^{0.99}$ components that are isomorphic to $H$, then at least $n^{0.9}$ of them appear in $G$ with colours unchanged. 
\end{enumerate}
\end{lemma}

Note that $G_0$ can be taken to be the empty graph, and then \Cref{lem:MainCol} yields a red-blue colouring of $G$ for which the two colours span \textit{almost} isomorphic linear forests . Indeed, this follows from Items \ref{itm:approx-only-short-paths} and \ref{itm:approx-concentration}; Items \ref{itm:approx-no-spiders} and \ref{itm:approx-gadget-survive} are only needed for proving the exact result.  In particular, this gives a short probabilistic proof for an approximate version of \Cref{WormaldConj}, as written in \Cref{thm:main-approx}. As mentioned above, the more involved part will be to prove the conjecture in its full power, which will be the content of the next sections of this paper.

We now turn to the proof of \Cref{lem:MainCol}. 

\begin{proof}
    We will define four random red-blue colourings $\chi_1, \ldots, \chi_4$ of $G$, with the purpose of showing that $\chi_4$ satisfies \ref{itm:approx-only-short-paths} to \ref{itm:approx-gadget-survive}, with positive probability. We will start our colouring process with a deterministic colouring $\chi_0$.
    
    \paragraph{Defining $\chi_0$.}
        Write $G_1 := G \setminus G_0$.
        Let $H_1, H_2, H_3$ be an edge-decomposition of $G_1$ satisfying \ref{itm:extend-Hs}.
        We first define a purple-green colouring of $G_1$.
        Let $F$ be the graph with degrees $1$ and $3$ that $H_1$ is a subdivision of. 
        By \Cref{thm:thomassen}, there is a purple-green colouring of $F$ whose monochromatic components are paths of length at most $5$. This defines a purple-green colouring of $H_1$ in a natural way: if an edge $e$ in $F$ is subdivided into a path $P_e$ in $H_1$ then colour the edges of $P_e$ by the colour of $e$. 
        Now colour $H_2$ as follows. Fix a component $C$ in $H_2$; so $C$ is an even cycle with at most two edges of $G_1 \setminus C$ touching it. If $C$ has no edges of $G_1 \setminus C$ touching it, colour $C$ purple; if $C$ has exactly one such edge $e$ touching it, colour $C$ by the opposite colour to $e$; finally if there are two such edges $e$ and $f$, if they both have the same colour then colour $C$ by the opposite colour, and if they have opposite colours, colour one segment of $C$ between $e$ and $f$ purple and the other green.
        Finally, we colour $H_3$. Let $H'$ be a component in $H_3$; then $H_3$ consists of three paths $P_1, P_2, P_3$ whose interiors are mutually vertex-disjoint and which have the same ends points. Without loss of generality, the lengths of $P_1$ and $P_2$ have the same parity. Colour $P_1 \cup P_2$ purple, and colour $P_3$ green.
     
        Denote the resulting colouring by $\chi_0$.
        
        \begin{claim}\label{claim:proper} 
            $\chi_0$ is a purple-green colouring of $G_1$, whose monochromatic components are even cycles or paths, both of length $O(\sqrt{\log n})$.   Additionally, for every vertex with degree at least  $2$ in $G_1$, there are two edges incident to it which have the same colour in $G_1$.
        \end{claim}

        \begin{proof}
        Note that the vertices who have degree two either are part of even cycles of $H_2$ or internal vertices of the subdivision paths of $H_3$, so they are adjacent to two edges of the same colour. Similarly, by considering the vertices in $H_1$, $H_2$, $H_3$ separately, it is easy to verify that vertices of degree three have at least two edges adjacent to them of the same colour.
        \end{proof}
        
    \paragraph{Defining $\chi_1$.}       
        We now define $\chi_1$ to be the random red-blue colouring of $G$ as follows.  For each monochromatic component in $(G_1, \chi_0)$ we uniformly at random colour it alternating starting red, or starting blue (this colouring step is similar to the first colouring step in \cite{das2021isomorphic}). Then for each monochromatic component of $G_0$ we either with probability $1/2$ keep the colouring of the component  or swap  the colours of all of its edges (red becomes blue, blue becomes red). All these colourings are made independently.
        
        We claim that every vertex in $(G, \chi_1)$ is incident with two red edges and one blue, or vice versa. Indeed, this holds for vertices with at least two neighbours in $G_0$ by \ref{itm:extend-degree}, and by the properties of the green-purple colouring which imply that every vertex with at least two neighbours in $G_1$ is incident to two edges of the same colour (see the second part of Claim~\ref{claim:proper}).
        
        Our next goal is to get rid of monochromatic cycles. 
        In order to later prove concentration, we would like all monochromatic cycles to be broken simultaneously (as opposed to iteratively, like in Akiyama--Chv\'atal's proof \cite{akiyama1981short}), as this makes it easy to analyse the random process. We will break all monochromatic cycles  in the next two steps. 
        
    \paragraph{Defining $\chi_2$.}
        For every monochromatic cycle $C$ in $(G, \chi_1)$, choose, uniformly at random and independently, an edge $e_C$ in $C$ and flip its colour.
        
        Note that it is possible to create new monochromatic cycles in $\chi_2$. In the next step we get rid of them, and  this process does not propagate any further.

    \paragraph{Defining $\chi_3$.}
          Let $C$ be a monochromatic cycle in $\chi_2$ and let $e_1, \ldots, e_k$ be the edges of $C$ whose colour was changed in the previous step; notice that there is at least one such edge. 
        Note that if $C$ and hence $e_1, \ldots e_k$ are blue, then each $e_i$ closes a red cycle $C_i$ ($C_i$ was the red cycle in $\chi_1$ which prompted the colour change of $e_i$), and similarly with the roles of red and blue reversed. We call these $k$ formerly red cycles the \textit{petals} of $C$. We call a monochromatic cycle $C$ together with its petals $C_1,\dots,C_k$ a \textit{monochromatic cycle-petal configuration} (see \Cref{fig:petals}). It is easy to see that all monochromatic cycles together their petals are vertex-disjoint from each others.
        
        \begin{figure}[h]
            \centering
            \begin{subfigure}[b]{.45\textwidth}
                \centering
                \includegraphics[scale = .8]{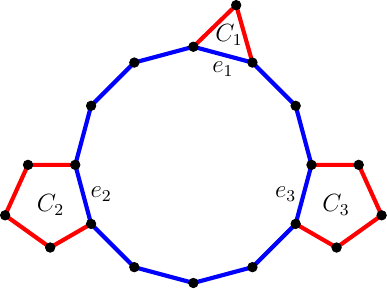}
                \caption*{A monochromatic cycle-petal configuration} 
            \end{subfigure}
            \hspace{1cm}
            \begin{subfigure}[b]{.45\textwidth}
                \centering 
                \includegraphics[scale = .8]{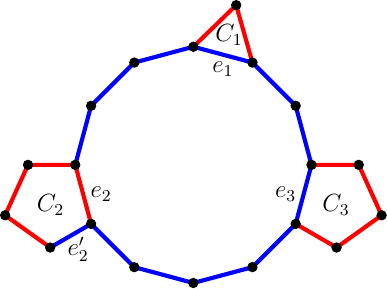}
                \caption*{Its recolouring at $e_2$ and $e_2'$}
            \end{subfigure}
            \caption{Getting rid of monochromatic cycles}
            \label{fig:petals}
        \end{figure}
         
        It is easy to check that the following is true. 
        \begin{claim}\label{cl:cyclespetalsdisjoint}
            Any two distinct  monochromatic cycles-petals configurations in $\chi_2$ are vertex-disjoint.
        \end{claim}
        

        For each monochromatic cycle $C$ in $\chi_2$ with petals $C_1, \ldots, C_k$ and petal edges $e_1, \ldots, e_k$, pick $i \in [k]$ uniformly at random and then pick one of the two edges next to $e_i$ on $C_i$ uniformly at random, call it $e_i'$, and swap the colours of $e_i$ and $e_i'$ (see \Cref{fig:petals}). This clearly destroys the monochromatic cycle $C$; moreover, it does not create vertices of red degree $3$ or blue degree $3$. We do this for all cycles simultaneously, and Claim~\ref{cl:cyclespetalsdisjoint} implies that  the following holds. 
        
        \begin{claim} \label{claim:no-cycles}
            The monochromatic components in $(G, \chi_3)$ are paths.
        \end{claim}
        
        \begin{proof} Note that using \Cref{cl:cyclespetalsdisjoint} we know that no edge changes their colour more than once while defining $\chi_3$, thus it is enough to show that after a swap for one particular cycle $C$, $e_i$ and $e_i'$ are not involved in new monochromatic cycles. Let us assume the edge $e_i$ which was swapped for cycle $C$ was blue before the swap. That means the cycle $C$ was  blue in $\chi_2$ and $e_i'$ was red.   Let $v$ be the vertex common to $e_i$ and $e_i'$ and write $e_i = uv$, and assume $e_i'=vw$. Notice that $v$ has blue degree $2$ in $\chi_2$ and so $e_i$ which is red in $\chi_3$ is not in a monochromatic cycle. For the monochromatic cycles involving $e_i'$, there are two options, either they use $e_i'$ and the vertex $u$ or they use another vertex of cycle $C$; the point being every such cycle has to use $w$. For the first case, $C \setminus \{e_i\} \cup \{e_i'\}$ is a blue path ending with the vertex $u$ which has red degree $2$, showing that $e_i'$ does not belong to a monochromatic cycle which uses the vertex $u$. For the second case, $e_i'$ cannot be involved in a blue cycle using some other vertex of $C$ because the blue graph has maximum degree two. 
        \end{proof}

    \paragraph{Defining $\chi_4$.}
        For each maximal monochromatic path $P$ in $(G, \chi_3)$, pick, uniformly at random and independently, an edge $e_P$ of the opposite colour which is incident to an end of $P$ (so $e_P$ is chosen out of a set of at most four edges).
        For each edge $e = xy$, if there are two distinct paths $P$ and $Q$ such that $e = e_P = e_Q$ (so $x$ is an end of one of exactly one of these paths and $y$ is an end of only the other path), swap the colour of $e$ with probability $1/2$; if no such paths exist, $e$ remains unchanged. Call the resulting colouring $\chi_4$.
        
        \def \chirb {\chi_3^{r \to b}} 
        \def \chibr {\chi_3^{b \to r}} 
        Define $\chirb$ to be the colouring of $G$ obtained from $\chi_3$ by swapping the colour of edges whose colour was swapped from red to blue when going from $\chi_3$ to $\chi_4$ (and keeping the colour of all other edges), and define $\chibr$ analogously with the roles of red and blue reversed. 
        \begin{claim} \label{claim:psi}
            Every red path $P$ in $\chi_3$ touches at most one edge whose colour changed from blue to red when going from $\chi_3$ to $\chibr$. Moreover, such an edge must touch an endpoint of $P$ and contains no other vertices of $P$.
        \end{claim}
        
        \begin{proof}
            The only blue edges touching $P$ whose colour can be swapped are those that touch the ends of $P$, and of those only $e_P$ can be swapped. This proves the first part of the claim.
            To see the second part, write $e_P = xy$ and suppose that $x$ is an end of $P$ and that the colour of $e_P$ was swapped. Notice that, by definition of $\chi_4$, if $e_P$ is indeed swapped then there is another red path $Q$ for which $e_Q = e_P$. But this means that $y$ is an end of $Q$, and so $y \notin V(P)$.
        \end{proof}
        
        The next claim easily follows.
        \begin{claim} \label{cl:chi-4-monocomponentsarepaths}
            All monochromatic components in $\chi_4$ are paths, and contain at most one edge that was of opposite colour in $\chi_3$.
        \end{claim}
        
        
        \begin{claim} \label{claim:chi-4-long-paths}
            With probability at least $4/5$, there are no monochromatic paths of length at least $1000\log n$ in $\chi_4$.
        \end{claim}
        
        \begin{proof}
            We will show that, with probability at least $9/10$, every blue path of length $400\log n$ in $\chi_3$ contains at least one edge whose colour in $\chibr$ is red.
            Fix a path $P$ of length $400\log n$ which is blue in $\chi_3$.
            Let $\cand(P)$ be the set of edges in the interior of $P$ that touch two distinct red components. These are \emph{candidates} for possibly  flipping their colour when going from $\chi_3$ to $\chi_4$. Notice that out of any two consecutive edges in the interior of $P$, at least one is in $\cand(P)$. This is because monochromatic components are paths and they can only touch $P$ at their ends.
            Thus $|\cand(P)| \ge (|P|-2)/2$.
            Consider an edge $e = xy$ in $\cand(P)$, and let $Q_1$ and $Q_2$ be the maximal red paths touching $x$ and $y$, respectively. Then, with probability at least $1/16$ we have $e_{Q_1} = e_{Q_2} = e$ because there are at most four choices for $e_{Q_i}$ and the choices are made independently. Thus, the probability that the colour of $e$ is swapped is at least $1/32$.
            
           Consider an auxiliary graph with vertices $\cand(P)$ and where $ef$ is an edge if and only if there is a red path touching both $e$ and $f$. This graph has maximum degree at most $4$ thus, by Tur\'an's theorem, it has an independent set, let's call it $i(P) \subseteq \cand(P)$ of size at least $|\cand(P)|/5 \ge (|P|-2)/10$. So the properties of edges in $i(P)$ is that no two edges in $i(P)$ touch the same red path. This means that the events $\{\text{$e$ is blue in $\chibr$}\}$, for $e \in i(P)$, are independent. It follows that the probability that none of the edges of $P$ are blue in $\chibr$ is at most $(1 - 1/32)^{(|P|-2)/10}$.
            
            The number of blue paths of length $400\log n$ in $\chi_3$ is at most $2n$ (there are $n$ ways to pick an end point, and at most two possible directions). Hence the probability that some red path of length $400\log n$ in $\chi_3$ survives is at most 
            \begin{equation*}
                2n \cdot \left(1 - \frac{1}{32}\right)^{(400\log n - 2)/10}
                \le 2n \cdot \exp\left(-\frac{400\log n - 1}{320}\right)
                < 1/10.
            \end{equation*}
            In particular, with probability at least $9/10$, there are no blue paths of length at least $400\log n$ in $\chibr$. 
            By \Cref{cl:chi-4-monocomponentsarepaths}, this shows that there are no red paths in $\chi_4$ of length at least $800\log n + 1$. Repeating the same argument with the roles of red and blue reversed, the claim is proved.
        \end{proof}
        
        \begin{claim} \label{claim:chi-4-spiders}
            With probability at least $4/5$, there is no red path of length $(\log n)^{2/3}$ in $\chi_4$ which is intersected by at least $100\sqrt{\log n}$ blue paths of length at least $100\sqrt{\log n}$, and vice verse.
        \end{claim}
        
        \begin{proof}

            \def \frP {\mathfrak{P}}
            Let $\cH$ be the collection of subgraphs $H$ of $\chi_3$ with the following structure: $H$ is the union of a red path $P$ of length at most $(\log n)^{2/3}$ with $100\sqrt{\log n}$ pairwise vertex-disjoint paths $Q$ of length $50\sqrt{\log n}$, each of which is either a blue path that touches $P$, or the union of a maximal blue path $Q_1$ with one end in $P$, a red edge touching the other end of $Q_1$, and another blue path $Q_2$ that touches $e$ and is disjoint of $Q_1$. 
            
        We claim that $|\cH| \le 2n (\log{n})^{2/3} \cdot \binom{(\log n)^{2/3}}{100\sqrt{\log n}} \cdot 4^{100\sqrt{\log n}}$. Indeed, there are at most $2n (\log{n})^{2/3} $ ways to choose a red path $P$ of length at most $\log{n}^{2/3}$, then there are  at most $\binom{(\log n)^{2/3}}{100\sqrt{\log n}}$ ways to choose the vertices on $P$ where the red paths in $\mathcal{Q}$ touch $P$. Finally, for every such interior vertex $v$ of $P$ either there exists a maximal blue path $Q$ of length at least $100\log{n}$ which touches $P$ at $v$ and we add it to $\mathcal{Q}$, or the maximal path has length less than $100\log{n}$ and at the end of it there are two red edges, each of which is possibly adjacent to two blue edges so we  have four possible choices for this specific $Q_2$, thus resulting overall $4^{100\sqrt{\log n}}$ choices.  What is the probability that such an $H$ did not have any of its blue edges swapped to red colour in $\chi_4$? This can be phrased as follows: let $\cP$ be a collection of at most $200 \sqrt{\log n}$ pairwise disjoint blue paths in $\chi_3$ whose total length is at least $5000\log n$. Similarly to the proof of \Cref{claim:chi-4-long-paths}, the probability that all edges in paths in $\cP$ remain blue in $\chi_4$ is at most $(1 - 1/32)^{(1000\log n - 400\sqrt{\log n})/10}$.Putting all this together, we get that the probability that there exists a graph $H \in \cH$ has none of its blue edges swapped is at most 
            \begin{equation*}
                2n \cdot  (\log{n})^{2/3}\cdot  \binom{(\log n)^{2/3}}{100\sqrt{\log n}} \cdot 2^{100\sqrt{\log n}} \cdot \left(1 - \frac{1}{32}\right)^{(1000\log n - 400\sqrt{\log n})/10} \le \frac{1}{10}.
            \end{equation*}

            Notice that \Cref{cl:chi-4-monocomponentsarepaths} shows that any red path of length at most $(\log n)^{2/3}$ in $\chi_4$ which is intersected by at least $100\sqrt{\log n}$ blue paths of length at least $100\sqrt{\log n}$ gives rise to a graph in $\cH$ whose blue edges did not change. The claim thus follows.
        \end{proof}
        
        \begin{claim} \label{claim:chi-4-Q3}
            Property \ref{itm:approx-concentration} holds, with probability at least $4/5$.
        \end{claim}
        
        \begin{proof}
            \def \bX {\bar{X}}
            \def \bY {\bar{Y}}
            \def \bx {\bar{x}}
            \def \by {\bar{y}}
            In order to show concentration, we will first describe how the random colouring described in this section can be obtained as a function of $9n$ independent random variables.
            For this, fix a permutation $\sigma$ of the edges of $G$.
            We define families of sets $A_i$, for $i \in [6]$, as follows. 
            
            \begin{itemize}
                \item 
                    Let $A_1 = A_6 = \{0,1\}$. 
                \item
                    Let $A_2, \ldots, A_5$ be the sets of permutations of $E(G)$ (so the sets $A_2, \ldots, A_5$ are identical). 
            \end{itemize}
            For each edge $e$ and $i \in [6]$, let $X_i(e)$ be the random variable obtained by picking an item from $A_i$, uniformly at random and independently.
            
            For a permutation $\tau$ of $E(G)$ and a subgraph $F \subseteq G$, let $\tau(F)$ be the edge of $F$ that appears first in $\tau$. Note that if $\tau$ is a uniformly random permutation of $E(G)$ then $\tau(H)$ is a uniformly random edge in $H$.
            
            Now define random variables $Y_i(H)$, as follows.
            \begin{itemize}
                \item
                    For $i \in \{1, 6\}$, define $Y_i(H) := X_i(\sigma(H))$ (here $\sigma$ is the permutation we fixed above). So $Y_i(H)$ is either $0$ or $1$, each with probability $1/2$, and the choice depends on $X_i(e)$, for the first edge $e$ in $H$ to appear in $\sigma$.    
                \item
                    For $i \in \{2, 3, 4\}$, define $Y_i(H) := \tau(H)$, where $\tau = X_i(\sigma(H))$. In other words, $Y_i(H)$ is an edge in $H$, chosen uniformly at random (indeed, $\tau = X_i(\sigma(H))$ is a uniformly random permutation of $E(G)$, and thus $\tau(H)$ is an edge of $H$, chosen uniformly at random), and this choice depends on $X_i(e)$, where $e$ is the first edge in $H$ to appear in $\sigma$.
                \item
                    For a path $P$, let $E_P$ be the set of edges outside of $P$ that touch an end of $P$, and define $Y_5(P) := \tau(E_P)$, where $\tau = X_5(\sigma(P))$. So $Y_5(P)$ picks an edge from $E_P$ uniformly at random, and the choice depends on $X_5(e)$, where $e$ is the first edge in $P$ to appear in $\sigma$.
            \end{itemize}
            
            We will show that the random colourings $\chi_1, \ldots, \chi_4$, defined above, can be defined as functions of the sequence of random variables $\bX := (X_i(e))_{i, e}$. Given an assignment $\bx$ of values to $\bX$, write $\chi_i(\bx)$ to be the red-blue colouring $\chi_i$ determined by $\bx$, for $i \in [4]$.
            
            We will then show that, for every $t \ge 1$, if $\bx$ and $\by$ are two assignments of values to $\bX$ that differ on exactly one coordinate, then the number of red (blue) components in $\chi_4(\bx)$ and $\chi_4(\by)$ that are isomorphic to $P_t$ differs by $O(\sqrt{\log n})$; equivalently, $|r_{\chi_4(\bx)}(P_t) - r_{\chi_4(\by)}(P_t)| = O(\sqrt{\log n})$ (with the analogous statement for blue also holding).
            
            Apply McDiarmid's inequality (\Cref{thm:mcdiarmid}), with $f(\bX) = r_{\chi_4(\bx)}(P_t)$, $m = \frac{1}{2}n^{2/3}$, $c := C \sqrt{\log n}$ for a large enough constant $C$, and $n_{\ref{thm:mcdiarmid}} = 6e(G)$. It follows that $|r_{\chi_4(\bX)}(P_t) - \Ex(r_{\chi_4(\bX)})| \le \frac{1}{2}n^{2/3}$, with probability at least $1 - 2\exp(-\Omega(n^{1/4}))$, for every $t$. By symmetry, the same holds for $b_{\chi_4(\bX)}(P_t)$. Noting that $\Ex(r_{\chi_4(\bX)}) = \Ex(b_{\chi_4(\bX)})$, again by symmetry, we have $|r_{\chi_4(bX)}(P_t) - b_{\chi_4(bX)}| \le n^{2/3}$, with probability $1 - 2\exp(-\Omega(n^{1/4}))$. Taking a union bound over all $t \le n$ implies that \ref{itm:approx-concentration} holds.
            
            Let us now explain how to obtain $\chi_i$, with $i \in [4]$, from $\bx$. 
            Fix a colouring $\chi_0$, obtained as explained at the beginning of the proof of \Cref{lem:MainCol}.
            To define $\chi_1$, for each purple or green component $H$, colour $H$ by one of its two proper red-blue edge-colourings, according to $Y_1(H)$ (so we think of $0$ as signifying one of these colourings, and $1$ as signifying the other). Similarly, for a connected component $H$ in the red-blue subgraph of $G$, leave $H$ unchanged if $Y_1(H) = 0$, and otherwise swap red and blue on $H$. Note that these choices are independent, as the subgraphs in question are pairwise edge-disjoint and thus depend on different variables $X_1(e)$.
            
            Next, to define $\chi_2$, for each monochromatic cycle $C$ swap the colour of $Y_2(C)$; as before, these choices are independent. For $\chi_3$, consider a monochromatic (say red) cycle $C$ in $\chi_2$ and let $E_C$ be the set of edges in $C$ whose colour was swapped in the previous step. Write $e := Y_3(E_C)$, let $f_1, f_2$ be the two blue edges touching $e$, and let $f := Y_4(\{f_1, f_2\})$. Swap the colours of $e$ and $f$ to get $\chi_3$. We again have independence, as each pair $f_1, f_2$ originates from a different cycle that was monochromatic in $\chi_1$.
            
            Finally, if there are two distinct maximal monochromatic paths $P$ and $Q$ such that $X_5(P) = X_5(Q)$, write $e := X_5(P)$ and swap the colour of $e$ if $X_6(e) = 1$ (and do nothing otherwise). As usual, since the relevant paths are pairwise edge-disjoint, we have independence.
            It should be easy to check that $\chi_1, \ldots, \chi_4$ defined in this way have the same distribution as their counterparts defined above.
            
            
            In what follows we show that if $\bx$ and $\by$ differ on at most one coordinate then the colourings $\chi_4(\bx)$ and $\chi_4(\by)$ differ on $O(\sqrt{\log n})$ edges. Notice that this means that the collections of monochromatic components differ on $O(\sqrt{\log n})$ elements, as claimed above.
            \begin{itemize}
                \item 
                    If $\bx$ and $\by$ differ on exactly one coordinate, with $i \in \{5, 6\}$, then the colourings $\chi_4(\bx)$ and $\chi_4(\by)$ differ on at most one edge.
                    
                    Suppose that $\bx$ and $\by$ agree on coordinates with $i \ge 5$ and $\chi_3(\bx)$ and $\chi_3(\by)$ differ on at most $a$ edges. Then the collections of maximal monochromatic paths in $\chi_3(\bx)$ and $\chi_3(\by)$ differ on at most $6a$ elements. Suppose that $e$ is an edge which is red in $\chi_4(\bx)$ and blue in $\chi_4(\by)$. Then either $e$'s colour in $\chi_3(\bx)$ and $\chi_3(\by)$ was different, or, without loss of generality, its colour was swapped by $\chi_4(\bx)$ but not $\chi_4(\by)$. In particular, in $\chi_3(\bx)$ the edge $e$ was blue, and touched the ends of two distinct maximal red paths $P$ and $Q$. If this were true for $\chi_3(\by)$ as well, with $P$ and $Q$ unchanged, then $e$'s colour would be swapped by $\chi_4(\by)$ too (using that $\bx$ and $\by$ agree on coordinates with $i \ge 5$). Thus one of $P$ and $Q$ is not a monochromatic component in $\chi_3(\by)$.
                    In summary, an edge $e$ can have a different colour in $\chi_4(\bx)$ and $\chi_4(\by)$ only if its colour in $\chi_3(\bx)$ and $\chi_3(\by)$ is not the same (this can happen at most $a$ times), or it touches the end of a path which is a monochromatic component in exactly one of $\chi_3(\bx)$ and $\chi_3(\by)$ (this can happen at most $4 \cdot 6a$ times). Thus $\chi_4(\bx)$ and $\chi_4(\by)$ differ on at most $25a$ edges.
                \item
                    If $i = 4$ then $\chi_3(\bx)$ and $\chi_3(\by)$ differ on at most two edges, and if $i = 5$ then $\chi_3(\bx)$ and $\chi_3(\by)$ differ on at most four edges. By the previous item, this shows that $\chi_4(\bx)$ and $\chi_4(\by)$ differ on at most $100$ edges.
                    
                    Suppose that $\bx$ and $\by$ agree on coordinates with $i \ge 4$ and $\chi_j(\bx)$ and $\chi_j(\by)$ differ on at most $a$ edges for $j \in [2]$. Thus the collections of monochromatic cycles-and-petals, defined according to $(\chi_1(\bx), \chi_2(\bx))$ and $(\chi_1(\bx), \chi_2(\bx))$, differ on at most $4a$ such structures. Indeed, if $H$ is a cycle-and-petals structure as defined by $(\chi_1(\bx), \chi_2(\bx))$ but not by $(\chi_2(\by), \chi_2(\by))$, then one of its edges has different colours in $\chi_j(\bx)$ and $\chi_j(\by)$, for some $j \in [2]$. Since each such $H$ gives rise to two colour swaps, and edges not in such structures do not change colours, it follows that $\chi_3(\bx)$ and $\chi_3(\by)$ differ on at most $9a$ edges. The previous item implies that $\chi_4(\bx)$ and $\chi_4(\by)$ differ on at most $225a$ edges.
                \item
                    If $i = 2$ then $\chi_2(\bx)$ and $\chi_2(\by)$ differ on at most two edges. By the previous item (using that $\chi_1(\bx) = \chi_1(\by)$), the colourings $\chi_4(\bx)$ and $\chi_4(\by)$ differ on at most $450$ edges.
                    
                    Suppose that $\bx$ and $\by$ agree on coordinates with $i \ge 2$ and that $\chi_1(\bx)$ and $\chi_1(\by)$ differ on at most $a$ edges. Then the collections of monochromatic cycles in $\chi_1(\bx)$ and $\chi_1(\by)$ differ by at most $2a$ cycles, implying that $\chi_2(\bx)$ and $\chi_2(\by)$ differ on at most $3a$ edges. By the previous item, this shows that $\chi_4(\bx)$ and $\chi_4(\by)$ differ on at most $675a$ edges.
                \item
                    If $i = 1$ then $\chi_1(\bx)$ and $\chi_2(\by)$ differ on $O(\sqrt{\log n})$ edges. Thus $\chi_4(\bx)$ and $\chi_4(\by)$ differ on $O(\sqrt{\log n})$ edges.
            \end{itemize} 
            
            To summarise, we have shown that if $\bx$ and $\by$ differ on at most one coordinate then $|r_{\chi_4(\bx)}(P_t) - r_{\chi_4(\by)}| = O(\sqrt{\log n})$, for every $t \ge 1$. It follows from McDiarmid's inequality \Cref{thm:mcdiarmid} that, with probability at least $9/10$, we have $\left|r_{\chi_4(\bX)}(P_t) - \Ex[r_{\chi_4(\bX)}(P_t)]\right| \le \frac{1}{2} n^{2/3}$, for every $t \ge 1$.
            
            Observe that red and blue were completely symmetric throughout the random process, and in all steps all actions were performed simultaneously. 
            Thus, $\Ex[r_{\chi_4(\bX)}(P_t)] = \Ex[b_{\chi_4(\bX)}(P_t)]$. It follows from the previous paragraph that 
            $\left|r_{\chi_4(\bX)}(P_t) - b_{\chi_4(\bX)}(P_t)\right| \le n^{2/3}$ for every $t \ge 1$, with probability at least $4/5$. This proves the claim.
        \end{proof}
        
        \begin{claim} \label{claim:chi-4-Q4}
            Property \ref{itm:approx-gadget-survive} holds for $\chi_4$ with probability at least $4/5$.
        \end{claim} 
        
        \begin{proof}
            Let $H$ be a red-blue subcubic graph, and let $\cH_0$ be the collection of components in $G_0$ that are isomorphic to $H$. We assume that $|\cH| \ge n^{0.99}$, as otherwise there is nothing to prove. It is easy to see that, with probability at least $0.99$, in $(G_0, \chi_1)$ there are at least $n^{0.9}$ components that are isomorphic to $H$; denote the set of such components by $\cH_1$. The main challenge of this claim is to show that many graphs in $\cH_1$ remain unchanged in $\chi_2$ and $\chi_3$, and for this we will use Property \ref{itm:extend-cycle-technical}. It is then quite easy to show that such graphs have decent probability of remaining the same in $\chi_4$, too.
            
            For this claim it will be convenient to think of the process defining $\chi_2$ and $\chi_3$, as follows. Fix a direction for each monochromatic cycle in $\chi_1$. Each such cycle chooses, uniformly at random, one of its edges $e_C^1$. Denote by $e_C^2$ the edge next to $e_C^1$ (according to the direction we fixed). Now, the colour of one of $e_C^1$ and $e_C^2$ (chosen uniformly at random) is swapped. This yields $\chi_2$. Recall that in $\chi_2$, we consider monochromatic cycles $C$ along with their edges $e_1, \ldots, e_{k_C}$, that were swapped previously, and the corresponding petals $C_1, \ldots, C_{k_C}$ (where $C_i$ is a cycle that was monochromatic in $\chi_1$ and elected to swap $e_i$). Now, $i$ is chosen uniformly at random from $[k_C]$. Notice that $e_i = e_{C_i}^j$ for some $j \in [2]$. Finally, the colour of $e_i$ is swapped back, and the colour of $e_{C_i}^{3-j}$ is swapped. It is easy to see that this emulates the processes used to obtain $\chi_3$ and $\chi_4$, with the slight difference that $e_i'$ is chosen pre-emptively.
            
            The point of this discussion is that for some $H \in \cH_1$ to survive, it suffices to make sure that $e_C^1$ is chosen so that it and its successor in $C$ are not in $H$, for every monochromatic cycle $C$ that intersects $E(H)$. While it is not hard to show that this is true with not-too-small probability for any $H \in \cH_1$, we need to work a little harder to define events which are independent.
            
            For each $H \in \cH_1$, we define a collection $\cC_H$ of subgraphs of monochromatic cycles as follows. For a monochromatic cycle $C$, let $H_1, \ldots, H_k$ be the graphs in $\cH_1$ that contain edges of $C$. Let $E_i$ be the collection of edges in $C$ that are at distance at most $2$ from $H_i$ (notice that the sets $E_i$ are pairwise disjoint by the assumption that the graphs in $\cH_1$ are at distance at least $10$ from each other). Add $E_i$ to $\cC_{H_i}$, for each $i$. Add $E_{C_0} := E(C) \setminus (E_{C, 1} \cup \ldots \cup E_{C, k})$ to a ``reserve collection'' $\cC_0$. 
            To choose $e_C^1$, we let each set $E_i$ choose an edge $e_{E_i} \in E_i$ uniformly at random and then pick $i \in [0,k]$ with probability $\frac{|E_i|}{|C|}$, and take $e_C^1 = e_{E_i}$. Observe that this process picks $e_C^1$ uniformly at random from $C$.
            
            We claim that, for every $H \in \cH_1$, every $E \in \cC_H$ contains two consecutive edges outsider of $H$. Indeed, if a monochromatic cycle $C$ intersects some $H' \in \cH_1$ other than $H$, then this holds due to the distance assumption on $H$ and $H'$, and otherwise it follows from \ref{itm:extend-cycle-technical}.
            Moreover, by \ref{itm:extend-compt-size}, every $E \in \cC_H$ has size $O(\sqrt{\log n})$.
            Thus, with probability at least $\Omega(1/\sqrt{\log n})$, both $e_E$ and its successor in its monochromatic cycle $C$ are not in $H$. Importantly, if this is the case then it is guaranteed that $e_C^1$ and its successor in $C$ are not in $H$. Denote by $A_H$ the event that $e_E$ and its successor in its monochromatic cycle are not in $H$, for every $E \in \cC_H$. Because $H$ has $O(\sqrt{\log n})$ edges, and the elements in $\cC_H$ are pairwise edge-disjoint, $\Pr[A_H] \ge (\sqrt{\log n})^{-O(\sqrt{\log n})} \ge n^{-0.01}$.
            Since the events $(A_H)_{H \in \cH_1}$ are independent, it follows that with probability at least $0.99$, at least $n^{0.8}$ of them hold, showing that the collection $\cH_3$ of graphs $H \in \cH_1$ whose colours remain unchanged in $\chi_3$ has size at least $n^{0.8}$. 
            
            To finish, notice that, with probability at least $\exp(-O(\sqrt{\log n})) \ge n^{-0.01}$, all edges in $H$ decide to veto being swapped, for each $H \in \cH_3$. Since these events are independent, it follows that with probability at least $0.99$, at least $n^{0.7}$ graphs in $\cH_3$ remain unchanged in $\chi_4$. This proves the claim.
        \end{proof}
    The proof of \Cref{lem:MainCol} now follows very easily. Indeed, with probability at least $1/5$, Properties \ref{itm:approx-only-short-paths} to \ref{itm:approx-gadget-survive} all hold, using \Cref{cl:chi-4-monocomponentsarepaths,claim:chi-4-long-paths,claim:chi-4-spiders,claim:chi-4-Q3,claim:chi-4-Q4}.
\end{proof}

\section{Exact solution - proof of the main theorem}\label{sec:Exact}
In this section we will state the main lemmas needed to turn the approximate version (\Cref{lem:MainCol}) into an exact solution of Wormald's conjecture. We will then show how to prove our main theorem (\Cref{thm:main}) using these lemmas.

An important element of these lemmas is the notion of \emph{gadgets}, defined here. We only give an abstract definition here, and later (in \Cref{sec:gadgets}) we give a concrete definition, that will be shown to satisfy the abstract one.

\begin{definition} \label{def:gadget-abstract}
    A \emph{blue $\ell$-gadget} is a red-blue subcubic graph $H$, satisfying the following. 
    There is another red-blue colouring $H'$ of $H$, such that for every red-blue cubic graph $G$ that contains $H$, and whose monochromatic components are paths, the graph $G'$, obtained from $G$ by replacing $H$ by $H'$, satisfies the following properties.
    \begin{itemize}
        \item 
            the monochromatic components in $G'$ are paths,
        \item
            $r_{G}(P_t) = r_{G'}(P_t)$ for every $t$,
        \item
                $b_{G}(P_t) - b_{G'}(P_t) = 
                \left\{
                \begin{array}{ll}
                    0 & \text{if $t > \ell$} \\
                    1 & \text{if $t = \ell$}.
                \end{array}
                \right.$
        \item
            $H$ and $H'$ differ on at most two edges,
        \item
            $H$ contains a blue path of length exactly $\ell$, whose ends are incident with two red edges.
    \end{itemize}
    
    A \emph{red $\ell$-gadget} is defined analogously, with the roles of red and blue replaced.
\end{definition}
We remark that the last two properties are not crucial, but it will later be convenient to have them, and they are satisfied naturally by the explicit gadgets that we use.

We now state a key lemma. It will give us a partial colouring of a cubic graph $G$ with properties that make it amenable to an application of \Cref{lem:MainCol}.

\begin{lemma}\label{lem:partialColMainLemma}
Let $G$ be a connected cubic graph on $n$ vertices, where $n$ is large. Then there exists a red-blue colouring of a subgraph $G_0\subseteq G$, such that the following holds.
\begin{enumerate}
    \item \label{itm:partial-extendable}
    $G_0$ is extendable according to \Cref{def:GoodPartialCol},
    \item \label{itm:partial-gadgets}
    for every $\ell\in[3,10^{10}\sqrt{\log n}]$ there are at least $n^{0.999}$ components in $G_0$ that contain a blue $\ell$-gadget.
\end{enumerate}
\end{lemma}

Most of the work towards the proof of \Cref{lem:partialColMainLemma} will go into finding a partial colouring satisfying \ref{itm:extend-degree} and \ref{itm:extend-cycle-technical} from \Cref{def:GoodPartialCol}. Later, we will show that this colouring also satisfies \ref{itm:extend-compt-size} and \ref{itm:extend-Hs}. To obtain such a colouring, we find many geodesic paths that are sufficiently far apart from each other (see \Cref{cl:geodesics}), and colour small balls around so that the desired properties hold. The latter is the content of the next lemma.

\begin{lemma}\label{lem:GoodGadgetsExist}
    Let $P$ be a geodesic of length at least $5\ell + 200$ in a cubic graph $G$. Then there is a partial colouring $\phi_P$, such that 
    \begin{enumerate}
        \item 
            the edges coloured by $\phi_P$ are at distance at most $4$ from $P$ and form a connected subgraph,
        \item
            the edges coloured by $\phi_P$ contains a blue $\ell$-gadget, 
        \item
            the partial colouring $\phi_P$ satisfies \ref{itm:extend-degree} and \ref{itm:extend-cycle-technical} from \Cref{def:GoodPartialCol}.
    \end{enumerate}
\end{lemma}

The case that where no two vertices in the geodesic have a common neighbour outside of $P$ is easier to handle, so we consider it separately in the next lemma. Given this lemma, the proof of \Cref{lem:GoodGadgetsExist} will focus on geodesics with (many) pairs of vertices having a common neighbour outside.

\begin{lemma}\label{lem:GoodCaterpillarGadgetsExist}
    Let $P$ be a geodesic of length at least $\ell + 36$ in a cubic graph $G$. Assume that no two vertices of $P$ have a common neighbour outside of $P$. Then there is a partial colouring $\phi_P$, such that
    \begin{enumerate}
        \item 
            the edges coloured by $\phi_P$ are at distance at most $3$ from $P$ and form a connected subgraph,
        \item
            the edges coloured by $\phi_P$ contains a blue $\ell$-gadget, 
        \item
            the partial colouring $\phi_P$ satisfies \ref{itm:extend-degree} and \ref{itm:extend-cycle-technical} from \Cref{def:GoodPartialCol}.
    \end{enumerate}
\end{lemma}

The rest of the paper is organised as follows. In the next subsection (\Cref{sec:finish}) we show how to prove the main theorem (\Cref{thm:main}) using \Cref{lem:MainCol} and \Cref{lem:partialColMainLemma}. In \Cref{sec:gadgets} we present the explicit gadgets that we find in the proofs of \Cref{lem:GoodGadgetsExist,lem:GoodCaterpillarGadgetsExist}. 
We prove \Cref{lem:GoodCaterpillarGadgetsExist} and \Cref{lem:GoodGadgetsExist} in Sections \ref{sec:GeodesicNoCommonNghbs} and \ref{sec:GeodesicWithCommon}, respectively.
Finally, in \Cref{sec:partial-colouring} we show how to combine everything to prove \Cref{lem:partialColMainLemma}.

\subsection{Proof of the main theorem}\label{sec:finish}

In this section we prove our main result, restated here, from previously stated lemmas.
\thmMain*

The proof proceeds as follows. We start by fixing a partial colouring $G_0$, as guaranteed by \Cref{lem:partialColMainLemma}. We then swap the colours of some of the component of $G_0$, to ensure that there are many blue and red $\ell$-gadgets, for all relevant $\ell$. We then consider a full red-blue colouring of $G$, as guaranteed by \Cref{lem:MainCol}. We now equalise $r(P_t)$ and $b(P_t)$, for all $t \ge 1$, in three steps. First, we equalise $r(P_t)$ and $b(P_t)$ for all $t \ge 4000 \sqrt{\log n}$. For this, we use that all components in $G$ are monochromatic paths of length $O(\log n)$, and that $G$ has no short red paths with many long blue paths touching it. We then equalise the number of red and blue edges. Finally, we equalise $r(P_t)$ and $b(P_t)$, one by one, from the largest $t$ for which they differ, to $t = 3$, using gadgets.
It is not hard to see that if $r(P_t) = b(P_t)$ for $t \ge 3$ and the number of red and blue edges is the same, then also $r(P_2) = b(P_2)$ and $r(P_1) = b(P_1)$, so we are done.

\begin{proof}[Proof of \Cref{thm:main} using \Cref{lem:partialColMainLemma}]
    Apply \Cref{lem:partialColMainLemma} to obtain a subgraph $G_0 \subseteq G$ along with a red-blue colouring which is extendable (recall \Cref{def:GoodPartialCol}) and contains at least $n^{0.999}$ components containing blue $\ell$-gadgets, for each $\ell \in [3, 10^{10}\sqrt{\log n}]$. Swap the colours of half of the blue $\ell$-gadgets to obtain a new extendable colouring of $G_0$, denoted $G_0'$, for which there are at least $\frac{1}{2}n^{0.999}$ components containing blue (respectively red) $\ell$-gadgets. It is easy to see that this new colouring is also extendable.
    
    Notice that the number of red-blue coloured subcubic graphs of order at most $10^{10}\sqrt{\log n}$ is at most $2^{3 \cdot 10^{10}\sqrt{\log n}} \left(10^{10} \sqrt{\log n}\right)^{3 \cdot 10^{10} \sqrt{\log n}} \le n^{0.001}$, using that $n$ is large. Thus, for every $\ell \in [3, 10^{10}\sqrt{\log n}]$, there are at least $n^{0.998}$ isomorphic components in $G_0'$ containing a blue (respectively red) $\ell$-gadget.
    
    Apply \Cref{lem:MainCol} to obtain a red-blue colouring $\chi_1$ of $G$ satisfying properties \ref{itm:approx-only-short-paths} to \ref{itm:approx-gadget-survive}. 
    In the next claim we equalise the number of red and blue components that are isomorphic to $P_t$, for all $t \ge 4000\sqrt{\log n}$, while changing few edges.
    
    \begin{claim} \label{claim:chi-2}
        There is a red-blue colouring $\chi_2$ of $G$ such that
        \begin{enumerate}[label = \rm(A\arabic*)]
            \item \label{itm:sqrt-no-cycles}
                the monochromatic components in $\chi_2$ are paths of length $O(\log n)$,
            \item \label{itm:sqrt-small-change}
                $\chi_1$ and $\chi_2$ differ on at most $n^{3/4}$ edges of $G$,
            \item \label{itm:sqrt-balanced}
                $r_{\chi_2}(P_t) = b_{\chi_2}(P_t)$ for every $t \ge 4000 \sqrt{\log n}$.
        \end{enumerate}
    \end{claim}
    
    \begin{proof}
        Let $\cP_r$ and $\cP_b$ be collections of, respectively, red and blue (maximal) paths in $\chi_1$, obtained as follows.
        For $t \ge 4000\sqrt{\log n}$ write $x_t = r(P_t) - b(P_t)$; if $x_t > 0$ add $x_t$ (maximal) red paths of length $t$ to $\cP_r$, and if $x_t < 0$ add $-x_t$ (maximal) blue paths of length $t$ to $\cP_b$.
        Form $\cJ_r$ by decomposing each path in $\cP_r$ into subpaths of length in $[1000\sqrt{\log n}, 2000\sqrt{\log n}]$, and define $\cJ_b$ analogously.
        
        For a path $J \in \cJ_r$, let $\cand(J)$ be the set of edges $xy$ in $J$ such that the blue components containing $x$ and $y$ are distinct paths of length at most $c\sqrt{\log n}$. Observe that if $x, y, z$ are three consecutive vertices in $J$, and $x$ and $y$ are in the same blue path $P$, then $z \notin P$. Thus at least $(|J|-3)/2$ edges $xy$ in $J$ are such that $x$ and $y$ belong to distinct blue components. By \ref{itm:approx-no-spiders}, at most $100\sqrt{\log n}$ vertices $x$ in the interior of $J$ are incident to blue paths of length at least $100\sqrt{\log n}$. It follows that there are at most $200\sqrt{\log n} \le |J|/5$ edges $xy$ in the interior of $J$ which touch blue paths of length at least $100\sqrt{\log n}$. Altogether, it follows that $|\cand(J)| \ge |J|/5 \ge 200\sqrt{\log n}$.
        
        Form a graph $F_r$ on vertex set $I_r = \bigcup_{J \in \cJ_r}\cand(J)$, where $ef$ is an edge 
        whenever $e$ and $f$ belong to distinct paths in $\cJ_r$ and there is a blue path that touches both $e$ and $f$. We claim that $F_r$ has an \emph{independent transversal}, namely an independent set $\{e_J : J \in \cJ_r\}$, where $e_J \in \cand(J)$. This follows from Lov\'asz's local lemma and the observation that $F_r$ has maximum degree at most $4$ (in fact, Loh and Sudakov \cite{loh2007independent} prove a similar but stronger result). 
        
        Define $\cJ_b$ and $F_b$ analogously, and let $I_b = \{e_J : J \in \cJ_b\}$ be an independent transversal in $F_b$.
        
        Now obtain $\chi_2$ from by swapping the colours of edges in $I_r \cup I_b$. We claim that $\chi_2$ satisfies the requirements of the claim. 
        
        Indeed, for \ref{itm:sqrt-no-cycles}, notice that we only swap the colours of edges in the interior of a monochromatic path, and thus we do not form vertices with monochromatic degree $3$. Moreover, for every maximal monochromatic path $P$, we swap the colour of at most one edge touching $P$, and only if it touches only one vertex of $P$ (which is an end of $P)$.
        
        Notice that by \ref{itm:approx-concentration}, $|r_{\chi_1}(P_t) - b_{\chi_1}(P_t)| \le n^{2/3}$ for every $t$, and thus $|\cP_r|, |\cP_b| = O(n^{2/3} \log n)$, implying that $|\cJ_r|, |\cJ_b| = O(n^{2/3} (\log n)^2)$. 
        Since the number of edges whose $\chi_1$ and $\chi_2$ colours differ is $|\cJ_r| + |\cJ_b|$, property \ref{itm:sqrt-small-change} follows, using that $n$ is large.
        
        To see \ref{itm:sqrt-balanced}, notice that maximal red paths in $\chi_1$, whose length is at least $100 \sqrt{\log n}$ and which are not in $\cP_r$, remain maximal red paths in $\chi_2$, whereas for every path $P \in \cP_r$, at least one among any $4000\sqrt{\log n}$ consecutive edges in $P$ becomes blue in $\chi_2$. Moreover, denoting by $\chi_1'$ the colouring obtained by swapping the colour of edges in $I_b$ (but not in $I_r$), the maximal red paths in $\chi_1'$ are either maximal red paths in $\chi_1$, or consist of two maximal red paths in $\chi_1$, of length at most $100\sqrt{\log n}$ each, and a single additional edge. Thus the maximal red paths in $\chi_2$ whose length is at least $4000\sqrt{\log n}$ are exactly the maximal red paths in $\chi_1$ whose length is at least $4000\sqrt{\log n}$ and which are not in $\cP_r$. An analogous reasoning holds for maximal blue paths, and yields \ref{itm:sqrt-balanced}.
    \end{proof}
    
    \begin{claim} \label{claim:chi-3}
        There is a red-blue colouring $\chi_3$ of $G$ such that 
        \begin{enumerate}[label = \rm(B\arabic*)]
            \item \label{itm:balance-edges-no-cycles}
                the monochromatic components in $\chi_3$ are paths of length $O(\log n)$,
            \item  \label{itm:balance-edges-small-change} 
                $\chi_1$ and $\chi_3$ differ on at most $n^{4/5}$ edges,
            \item   \label{itm:balance-edges-balance-long-paths}
                $r_{\chi_3}(P_t) = b_{\chi_3}(P_t)$ for $t \ge 5000\sqrt{\log n}$.
            \item   \label{itm:balance-edges-equal-edges}
                the number of blue edges in $\chi_3$ equals the number of red edges in $\chi_3$.
        \end{enumerate}
    \end{claim}
    
    \begin{proof}
        We use an argument similar to the proof of \Cref{claim:chi-2}.
        Let $x_r$ and $x_b$ be the number of, respectively, red and blue edges in $\chi_2$. By \ref{itm:approx-concentration} and \ref{itm:sqrt-small-change}, $|x_r - x_b| \le n^{2/3} + n^{3/4} \le 2n^{3/4}$. Without loss of generality, assume that $x_r \ge x_b$. Recall that by \ref{itm:approx-gadget-survive}, $\chi_1$ has at least $n^{0.9}$ red $\ell_0$-gadgets, where $\ell_0 := 4000\sqrt{\log n}$. Since at most $n^{3/4}$ of them are destroyed when going from $\chi_1$ to $\chi_2$, this means in particular that there is a collection $\cP$ of exactly $(x_r - x_b)/2$ distinct maximal red paths of length $\ell_0$ (here we use the assumption that $n$ is divisible by $4$).
        Following the reasoning in the proof of \Cref{claim:chi-2}, there is a set $I = \{e_P: P \in \cP\}$, where $e_P$ is an edge in the interior of $P$ which is incident to two distinct blue components, none of which is a path of length at least $100\sqrt{\log n}$, and moreover no two edges from $P$ touch the same blue component.
        Define $\chi_3$ to be the colouring obtained from $\chi_2$ by swapping the colour of the edges in $I$.
        The proof that items \ref{itm:balance-edges-no-cycles} to \ref{itm:balance-edges-equal-edges} hold is very similar to the end of the proof of \Cref{claim:chi-2}; we omit the details.
    \end{proof}
    
    \def \lm {\ell_{\max}}
    Let $\chi_3$ be a red-blue colouring of $G$ satisfying \ref{itm:balance-edges-no-cycles} to \ref{itm:balance-edges-equal-edges} above.
    Our plan now is to equalise the number of maximal red and blue paths of length $\ell$, starting with the maximal $\ell$ for which they differ, using $\ell$-gadgets. 
    Let $\lm$ be the maximal $\ell$ for which $r_{\chi_3}(P_t) \neq b_{\chi_3}(P_t)$; by \ref{itm:balance-edges-balance-long-paths}, $\lm \le 5000\sqrt{\log n}$.
    
    \begin{claim} \label{claim:chi-4}
        For $\ell \in [3, \lm+1]$ there is a red-blue colouring $\psi_{\ell}$ of $G$, such that
        \begin{enumerate}[label = \rm(C\arabic*)]
            \item \label{itm:apply-gadgets-no-cycles}
                all monochromatic components of $\psi_{\ell}$ are paths of length $O(\log n)$,
            \item \label{itm:apply-gadgets-paths-balanced}
                $r_{\psi_{\ell}}(P_t) = b_{\psi_{\ell}}(P_t)$ for $t \ge \ell$,
            \item \label{itm:apply-gadgets-edges-balanced}
                the number of blue edges in $\psi_{\ell}$ equals the number of edges edges,
            \item \label{itm:apply-gadgets-small-change}
                $\psi_{\ell}$ differs from $\chi_1$ on at most $f(\ell)$ edges, where $f(\ell) = n^{4/5} \cdot 6^{\lm - \ell - 1}$.
        \end{enumerate}
    \end{claim}
    \begin{proof}
        We prove the claim by induction on $\ell$. Notice that, by taking $\psi_{\lm+1} := \chi_3$, the claim holds for $\ell = \lm+1$. 
        Suppose that $\psi_{\ell+1}$ is a suitable for $\ell+1$, where $\ell \in [3, \lm]$. 
        Write $x = r_{\psi_{\ell+1}}(P_{\ell}) - b_{\psi_{\ell+1}}(P_{\ell})$.
        Because $\psi_{\ell+1}$ and $\chi_1$ differ on at most $f(\ell+1)$ edges, and by \ref{itm:approx-concentration}, $|x| \le n^{2/3} + 2f(\ell+1) \le n^{5/6}$. Without loss of generality, assume $x \ge 0$. By \ref{itm:approx-gadget-survive}, we may pick a collection $\cH$ of $x$ red $\ell$-gadgets in $\psi_{\ell+1}$ (that are a distance at least $10$ away from each other). For each gadget $H \in \cH$, let $H'$ be its alternate colouring, which satisfies the properties in \Cref{def:gadget-abstract}. Obtain $\psi_{\ell}$ from $\psi_{\ell+1}$ by replacing $H$ by $H'$ for each $H \in \cH$. It is easy to see that items \ref{itm:apply-gadgets-no-cycles} and \ref{itm:apply-gadgets-paths-balanced} hold for $\psi_{\ell}$. Notice that by the second item in \Cref{def:gadget-abstract}, the number of red edges in $\psi_{\ell+1}$ equals the number of red edges in $\psi_{\ell}$, and thus \ref{itm:apply-gadgets-edges-balanced} holds. Finally, notice that $\psi_{\ell+1}$ and $\psi_{\ell}$ differ only on edges in gadgets in $\cH$, and in each gadget they differ on at most two edges. Thus, they differ on at most $2x$ edges, implying that $\psi_{\ell}$ and $\chi_1$ differ on at most the following number of edges
        \begin{equation*}
            2x + f(\ell+1) \le 2 n^{2/3} + 4f(\ell+1) \le 6 f(\ell+1) = f(\ell),
        \end{equation*}
        using $f(\ell + 1) \ge n^{2/3}$. Thus \ref{itm:apply-gadgets-small-change} holds and the claim is proved.
    \end{proof}

    We claim that $\psi_3$ is a red-blue colouring of $G$ whose colour classes are isomorphic linear forests. Indeed, by \ref{itm:apply-gadgets-no-cycles}, the two colour classes are linear forests. By \ref{itm:apply-gadgets-paths-balanced}, $r(P_t) = b(P_t)$, for $t \ge 3$. By \ref{itm:apply-gadgets-edges-balanced}, the numbers of red and blue edges are the same. Because $G$ is a cubic graph all of whose vertices are incident to two edges of one colour and one edge of the other colour, this  means that the number of red $P_2$'s is the same as the number of blue $P_2$'s. Since we already know that the number of red $P_2$'s that are contained in a longer red path is the same as the number of blue $P_2$'s contained in a longer blue path, we find that $r(P_2) = b(P_2)$. Finally, a similar argument shows that $r(P_1) = b(P_1)$.
\end{proof}

\section{Completing gadgets to an extendable pre-colouring} \label{sec:partial-colouring}

In this section we prove \Cref{lem:partialColMainLemma}, which asserts that every large connected cubic graph has an extendable (according to \Cref{def:GoodPartialCol}) red-blue subgraph $G_0$ which has many gadgets of all relevant lengths. The proof will be conditional on \Cref{lem:GoodGadgetsExist}, which finds a gadget within a small-radius neighbourhood of any long geodesic. 
Thus, after this section our only remaining task will be to prove \Cref{lem:GoodGadgetsExist}.

We will apply \Cref{lem:GoodGadgetsExist} to a large collection of geodesics that are sufficiently far apart, and then modify the resulting partial colouring to an extendable one. 
The first step is an easy claim that shows that there are many long geodesics that are far apart from each other. 
Here, given a graph $G$, we write $d(u,v)$ for the distance in $G$ between the vertices $u$ and $v$, and $d(H, F)$ for the distance between subgraphs $H$ and $F$, namely the minimum of $d(u,v)$ over $u \in V(H)$, $v \in V(F)$.

\begin{claim}\label{cl:geodesics}
    Let $0 < \eps < 1$ and let $n$ be large.
    Suppose that $G$ is a connected cubic graph on $n$ vertices. Then there exists a collection $\cP$ of at least $\frac{1}{6}n^{1 - \eps}$ geodesics, each of length $m := 10^{10}\sqrt{\log n}$, such that $d(P, Q) \ge 50$ for every distinct $P, Q \in \cP$.
\end{claim}

\begin{proof}
    We will show that there are at least $k=\frac{1}{6}n^{1- \eps}$ vertices $v_1,\dots,v_k$ in $G$ such that $d(v_i,v_j)\geq \eps\log n$ for every $i\neq j$. Note that every subpath of a geodesic is also a geodesic, thus if we take one subpath of length $m$ of the geodesic touching $v_i$ and some $v_k$, for each $i$, the collection of these paths will be as desired. 
    
    Let $k_0$ be the maximum such number of vertices. Then since the graph is cubic, the number of vertices in each ball of radius $\eps\log n$ is at most $ 3\cdot 2^{\eps\log n}$. Thus, $\big|\bigcup_{i\geq k_0}B(v_i,\eps\log n)\big|\leq k_0\cdot 2^{\eps\log n + 1}=2k_0\cdot n^{\eps}$. On the other hand, by the maximality of $k_0$, we have that $\big|\bigcup_{i\geq k_0}B(v_i,\eps\log n)\big|\geq n$. This gives $k_0\geq \frac{1}{6}n^{1-\eps}$.
\end{proof}

Now we can use \Cref{lem:GoodGadgetsExist}. Recall that \Cref{lem:GoodGadgetsExist} guarantees that for every geodesic $P$ of length $5\ell + 200$ in a cubic graph $G$ (where $\ell\in[3,10^{10}\sqrt{\log n}]$), there is a partial colouring $\phi_P$, such that 
\begin{enumerate}
    \item \label{itm:phiP-1}
        the edges coloured by $\phi_P$ are at distance at most $4$ from $P$,
    \item\label{itm:phiP-2}
        the edges coloured by $\phi_P$ contain a blue $\ell$-gadget, 
    \item\label{itm:phiP-3}
        the partial colouring $\phi_P$ satisfies \ref{itm:extend-degree} and \ref{itm:extend-cycle-technical} from \Cref{def:GoodPartialCol}.
\end{enumerate}
    
We recall that \ref{itm:extend-degree} says that every vertex with coloured degree at least $2$ has both red and blue neighbours, and \ref{itm:extend-cycle-technical} says that every cycle with only blue edges or only red edges has two consecutive uncoloured edges.

We next unite all of these colouring around paths into one partial colouring $\phi$.

\begin{definition}\label{def:ParColAfterAlg}
Let $G$ be a cubic graph, let $\cP$ be a collection of geodesics as in \Cref{cl:geodesics} and, for $P \in \cP$, let $\phi_P$ be a partial colouring satisfying \ref{itm:phiP-1}, \ref{itm:phiP-2} and \ref{itm:phiP-3} above. Let $\phi$ be the colouring obtained by the union of all of those partial colourings $\phi:=\bigcup_{P\in \cP}\phi_P$ (notice that there are no conflicts, using Property \ref{itm:phiP-1} above and the assumption that any two geodesics in $\cP$ are a distance at least $50$ apart).
\end{definition}

\begin{claim}\label{cl:TotalPartialColouringE1E3}
The colouring $\phi$ satisfies \ref{itm:extend-degree} and \ref{itm:extend-cycle-technical}. In addition, the size of each connected component of the coloured graph is $O(\sqrt{\log n})$ and every two components are a distance at least 20 apart.
\end{claim}

\begin{proof}
Since the edges coloured by $\phi_P$ are at distance at most $4$ from $P$, for every $P \in \cP$, and the geodesics in $\cP$ are at distance at least 50 from each other, any two coloured components in $\phi$ are at distance at least 20 from each other. In particular, the edges coloured by distinct colourings $\phi_P$ are pairwise vertex-disjoint. Thus, because \ref{itm:extend-degree} holds for every $\phi_P$, it also holds for $\phi$.

For \ref{itm:extend-cycle-technical}, let $C$ be a cycle with no red edges or no blue edges. If $C$ contains edges coloured by at most one $\phi_P$ then $C$ has two consecutive uncoloured edges, due to $\phi_P$ satisfying \ref{itm:extend-cycle-technical}. Otherwise, it has edges from two coloured components, so by the fact that they are at distance at least 20 from each other, $C$ has two consecutive uncoloured edges. So \ref{itm:extend-cycle-technical} holds for $\phi$.

Finally, since $|P| = O(\sqrt{\log n})$ and the edges coloured by $\phi_P$ are at distance at most $4$ from $P$, every coloured component has size $O(\sqrt{\log n})$.
\end{proof}

We next show that the colouring $\phi$ can be extended to a colouring that also satisfies \ref{itm:extend-Hs}. The proof of the following lemma will be postponed to the next subsection. 

\begin{lemma}\label{Lemma_get_E4}
    Let $G$ be a cubic graph and $\chi$ a partial colouring satisfying \ref{itm:extend-degree} and \ref{itm:extend-cycle-technical}.
    
    Then there is a partial colouring $\chi'$ which extends $\chi$, satisfies \ref{itm:extend-degree}, \ref{itm:extend-cycle-technical} and \ref{itm:extend-Hs} and, additionally, every $\chi'$-coloured edge is within distance $2$ of a $\chi$-coloured edge, and connected to it via coloured edges.
\end{lemma}

It is now easy to prove \Cref{lem:partialColMainLemma}.
\begin{proof}[Proof of \Cref{lem:partialColMainLemma} using \Cref{lem:GoodGadgetsExist,Lemma_get_E4}]
    Let $\cP$ be a collection of $\Omega(n^{0.9999})$ geodesics of length $10^{11}\sqrt{\log n}$ that are at distance at least $50$ from each other; such a collection exists by \Cref{cl:geodesics}. For each $\ell \in [3, 10^{10}\sqrt{\log n}]$, pick $n^{0.999}$ geodesics in $\cP$ (such that each geodesic is picked at most once), and apply \Cref{lem:GoodGadgetsExist} to $P$ to find a partial colouring $\Phi_P$ that forms a blue $\ell$-gadget, its coloured edges are at distance at most $4$ from $P$, and which satisfies \ref{itm:extend-degree} and \ref{itm:extend-cycle-technical}. Consider the union $\phi$ of these partial colourings, like in \Cref{def:ParColAfterAlg}; so $\phi$ satisfies \ref{itm:extend-degree} and \ref{itm:extend-cycle-technical}, its coloured components have size $O(\sqrt{\log n})$, and are at distance at least 50 from each other. Now apply \Cref{Lemma_get_E4} to find a partial colouring $\phi'$ that extends $\phi$, satisfies \ref{itm:extend-degree}, \ref{itm:extend-cycle-technical}, and \ref{itm:extend-Hs}, and whose newly coloured edges are at distance at most $2$ from previously coloured edges. Thus coloured components in $\phi'$ have size $O(\sqrt{\log n})$ and are at distance at least $10$ from each other, as required for \ref{itm:extend-compt-size}. The partial colouring $\phi'$ thus satisfies the requirements of the lemma.
\end{proof}

\subsection{Getting property \ref{itm:extend-Hs}}
The goal of this section is to prove Lemma~\ref{Lemma_get_E4}, which allows us to extend an arbitrary colouring satisfying \ref{itm:extend-degree} and \ref{itm:extend-cycle-technical} in order to get a new colouring satisfying \ref{itm:extend-Hs}. Getting property \ref{itm:extend-Hs} turns out to be essentially equivalent to asking for a colouring without cycles satisfying the following definition:

\begin{definition}
    Let $\chi$ be a partial colouring of a graph $G$.
    Say that a cycle $C$ is $\chi$-\ref{itm:extend-Hs}-bad if it is odd, its edges are uncoloured and it contains at most two vertices of uncoloured degree $3$ (and the remaining vertices have uncoloured degree $2$).
\end{definition}
We will also use the following definition for maintaining \ref{itm:extend-cycle-technical}.

\begin{definition}
    Let $\chi$ be a partial colouring of a graph $G$.
    Say that a cycle $C$ is $\chi$-\ref{itm:extend-cycle-technical}-bad if it contains only one colour, and does not contain two consecutive uncoloured edges.
\end{definition}
We will often use the useful fact that $\chi$-\ref{itm:extend-cycle-technical}-bad cycles cannot pass through vertices with uncoloured degree 3 (as if they did, they would have consecutive uncoloured edges).

The following three lemmas allow us to get rid of particular kinds of bad cycles. Throughout this section, when we say that a colouring $\chi'$ extends $\chi$, we mean that $\chi'$ is a partial colouring that agrees with $\chi$, and each $\chi'$-coloured connected component contains a $\chi$-coloured connected component.

\begin{lemma}\label{Lemma_adjacent_uncoloured_3}
    Let $G$ be a cubic graph and $\chi$ a partial colouring satisfying \ref{itm:extend-degree} and \ref{itm:extend-cycle-technical}. Let $C$ be a $\chi$-\ref{itm:extend-Hs}-bad cycle with two adjacent vertices of uncoloured degree $3$. 
    
    Then there is a partial colouring $\chi'$ which extends $\chi$, satisfies \ref{itm:extend-degree} and \ref{itm:extend-cycle-technical}, has that all $\chi'$-\ref{itm:extend-Hs}-bad cycles are $\chi$-\ref{itm:extend-Hs}-bad, and has $C$ not $\chi'$-\ref{itm:extend-Hs}-bad.
Additionally, every $\chi'$-coloured edge touches $C$.
\end{lemma}
\begin{proof}
    Let $C = (c_1 \ldots c_k)$ with $c_2, c_3$ having uncoloured degree $3$. 
    For each $i$, let $c_i'$ be the unique neighbour of $c_i$, different from $c_{i-1}, c_{i+1}$. 
    Since $C$ is $\chi$-\ref{itm:extend-Hs}-bad, we know that $c_i$, with $i \notin \{2,3\}$ has uncoloured degree $2$, so $c_ic_i'$ is coloured. Without loss of generality, suppose that $c_1c_1'$ is red. 
    We construct a colouring $\chi'$ that will satisfy the lemma as follows:
    \begin{enumerate}[label = \rm(\alph*)]
        \item \label{case:chi-a} {Suppose that $c_2'$ has uncoloured degree 3.}
        Colour $c_1c_2$ blue, $c_2c_3$ red. For $i=4, \dots, k$, colour $c_ic_{i+1}$ by the opposite colour of $c_ic_i'$. 
        \item \label{case:chi-b} {Suppose that $c_2'$ touches a red edge $c_2'c_2''$.}
        Colour $c_1c_2$ blue, $c_2c_2'$ blue, $c_2c_3$ red. For $i=4, \dots, k$, colour $c_ic_{i+1}$ by the opposite colour of $c_ic_i'$.
        \item \label{case:chi-c} {Suppose that $c_2'$ has touches a blue edge $c_2'c_2''$.} Colour $c_1c_2$ blue, $c_2c_2'$ red.
    \end{enumerate}
    To see that $\chi'$ satisfies  \ref{itm:extend-degree}, we need to check the property at all vertices whose edges changed colour: In case \ref{case:chi-a}, $c_1, c_2, c_4, \dots, c_k$ all see both colours, while $c_3$ has uncoloured degree $2$. In case \ref{case:chi-b}, $c_2', c_1, c_2, c_4, \dots, c_k$ all see both colours, while $c_3$ has uncoloured degree $2$. In case \ref{case:chi-c}, $c_1, c_2, c_2'$ all see both colours. 
    
    To see that $\chi'$ satisfies  \ref{itm:extend-cycle-technical}, suppose for contradiction that we have a $\chi'$-\ref{itm:extend-cycle-technical}-bad cycle $B$. Then $B$ must contain some edge whose colour changed in $\chi'$ (since $\chi$ satisfied \ref{itm:extend-cycle-technical}). Note that $B$ contains $c_ic_{i+1}$ for some $i$ (the only other edge whose colour can change between $\chi$ and $\chi'$ is $c_2c_2'$. But then $B$ would also contain another edge through $c_2$ i.e.\ the edge $c_2c_1$ or $c_2c_3$). It is impossible that $B=C$ (in \ref{case:chi-a} and \ref{case:chi-b}, $C$ contains both colours on $c_1c_2, c_2c_3$. In \ref{case:chi-c}, $C$ contains adjacent uncoloured edges $c_2c_3, c_3c_4$), so $B$ must leave $C$ at some point i.e.\ for some $i$, $B$ contains both $c_ic_{i+1}$ and $c_ic_i'$. This must mean that $i=2$ since for all other $i$, we either have that $c_ic_{i+1}$ and $c_ic_i'$ have opposite colours or $c_ic_{i+1}$ and $c_ic_i'$ are both uncoloured.  Finally, for $i=2$ note that in case \ref{case:chi-a} $B$ cannot go through $c_2'$ due to it having uncoloured degree 3, in case \ref{case:chi-b} $C'$ we cannot have $c_2c_3, c_2c_2'\in B$ due to them having opposite colours, and in case \ref{case:chi-c} $B$ cannot go through $c_3$ due to it having uncoloured degree 3.

    Suppose that there is some $\chi'$-\ref{itm:extend-Hs}-bad cycle $B$ which is not $\chi$-\ref{itm:extend-Hs}-bad. Then it would have to go through some vertex $x$ whose uncoloured degree is 3 in $\chi$ and 2 in $\chi'$. It is easy to check that the  only such vertex is $c_3$, which, in cases \ref{case:chi-a}, \ref{case:chi-b}, has the edge $c_2c_3$ coloured and $c_3c_3', c_3c_4$ uncoloured (and in Case \ref{case:chi-c} still has uncoloured degree $3$ in $\chi'$).  Then $B$ goes through $c_3c_3', c_3c_4$ also. However, in cases \ref{case:chi-a}, \ref{case:chi-b} the vertex  $c_4$ has uncoloured degree $1$, so it cannot be part of a $\chi'$-\ref{itm:extend-Hs}-bad cycle.
\end{proof}

\begin{lemma}\label{Lemma_bad_cycle_touches_different_colours}
    Let $G$ be a cubic graph and $\chi$ a partial colouring satisfying \ref{itm:extend-degree}, \ref{itm:extend-cycle-technical}. Let $C$ be a $\chi$-\ref{itm:extend-Hs}-bad cycle whose vertices touch both red and blue edges. 
    
    Then there is a partial colouring $\chi'$ which extends $\chi$, satisfies \ref{itm:extend-degree} and \ref{itm:extend-cycle-technical},  has that all $\chi'$-\ref{itm:extend-Hs}-bad cycles are $\chi$-\ref{itm:extend-Hs}-bad, and has $C$ not $\chi'$-\ref{itm:extend-Hs}-bad. Additionally, every $\chi'$-coloured edge touches $C$.
\end{lemma}
\begin{proof}
    If $C$ contains two adjacent vertices of uncoloured degree $3$, then we are done by Lemma~\ref{Lemma_adjacent_uncoloured_3}, so suppose that this does not happen.
    Let $C = (c_1 \ldots c_k)$. For each $i$, let $c_i'$ be the unique neighbour of $c_i$ which is not $c_{i-1}$ or $c_{i+1}$. Define the following colouring $\chi'$ that will satisfy the lemma: for each $i$ with $c_ic_i'$ coloured in $\chi$, colour $c_{i}c_{i+1}$ by the opposite colour of $c_{i}c_i'$.
    
    For \ref{itm:extend-degree}, note that since $\chi$ satisfied \ref{itm:extend-degree}, the property can only fail at some $c_i$ with $c_ic_{i+1}$ or $c_{i-1}c_i$ coloured. If $c_ic_{i+1}$ was coloured, then $c_ic_{i+1}$, $c_i'c_i$ have different colours. If $c_{i-1}c_i$ was coloured and $c_ic_{i+1}$ was not coloured, then  $c_i'c_i$ is uncoloured and so we have that $c_{i-1}c_i$ is the only $\chi'$-coloured edge through $c_i$.
    
    For \ref{itm:extend-cycle-technical}, suppose that we have some $\chi'$-\ref{itm:extend-cycle-technical}-bad cycle $B$. Since $\chi$ satisfied \ref{itm:extend-cycle-technical}, $B$ must contain some edge $c_ic_{i+1}$ which was coloured in $\chi'$. Note that it is impossible that $B=C$, since $C$ contains both a red and a blue edge (to see this first recall that by the lemma's assumptions we have some $i,j$ with $c_i'c_i$ red and $c_j'c_j$ blue. By definition of $\chi'$, we have $c_ic_{i+1}$ blue and $c_jc_{j+1}$ red). Thus $B$ must contain the sequence $c_a'c_ac_{a+1}\dots c_bc_b'$ for some $a,b$. But then either $c_a'c_a, c_ac_{a+1}$ are both uncoloured, or they have opposite colours. In both cases, we get a contradiction to $B$ being $\chi'$-\ref{itm:extend-cycle-technical}-bad.
    
    Notice that the only new vertices of uncoloured degree $2$  are $c_i$ with $c_i'c_i$ uncoloured and $c_{i-1}'c_{i-1}$ coloured in $\chi$. Any $\chi'$-\ref{itm:extend-Hs}-bad cycle $C'$ that is not $\chi$-\ref{itm:extend-Hs}-bad must pass through such a vertex and so contain the sequence $c_{i}'c_ic_{i+1}$. If $c_{i+1}'c_{i+1}$ is coloured in $\chi$, then we have that $c_{i+1}$ has uncoloured degree 1 in $\chi'$ (and so could not be contained in the uncoloured cycle $C$). On the other hand, if $c_{i+1}'c_{i+1}$ is uncoloured in $\chi$, then we have  two consecutive vertices $c_i, c_{i+1}$ of  uncoloured degree $3$ in $\chi$ (which we have assumed does not happen).
\end{proof}

\begin{lemma}\label{Lemma_adjacent_coloured_edges}
Let $G$ be a cubic graph and $\chi$ a partial colouring satisfying \ref{itm:extend-degree} and \ref{itm:extend-cycle-technical}. Let $C$ be a $\chi$-\ref{itm:extend-Hs}-bad cycle with two consecutive red edges. 

Then there is a partial colouring $\chi'$ which extends $\chi$, satisfies \ref{itm:extend-degree}, \ref{itm:extend-cycle-technical},  has that all $\chi'$-\ref{itm:extend-Hs}-bad cycles are $\chi$-\ref{itm:extend-Hs}-bad, and has $C$ not $\chi'$-\ref{itm:extend-Hs}-bad. Additionally, every $\chi'$-coloured edge touches $C$.
\end{lemma}

\begin{proof}
By Lemma~\ref{Lemma_bad_cycle_touches_different_colours}, we can assume that $C$ touches only red edges.
Let $C = (c_1 \ldots c_k)$. For each $i$, let $c_i'$ be the unique neighbour of $c_i$ outside $C$.
Colour $c_1c_2$ blue to get a colouring $\chi'$. We claim that $\chi'$ satisfies the lemma.

For \ref{itm:extend-degree}, note that since $\chi$ satisfied \ref{itm:extend-degree}, the property can only fail at $c_1$ or $c_2$. But both $c_1, c_2$ touch edges of both colours.

For \ref{itm:extend-cycle-technical}, suppose that we have some $\chi'$-\ref{itm:extend-cycle-technical}-bad  cycle $B$. Since $\chi$ satisfied \ref{itm:extend-cycle-technical}, $B$ must contain $c_1c_2$. Note that since $c_1'c_1, c_2'c_2$ are red, and  $c_1c_2$ is blue, $B$ must contain the sequence $c_1 c_2 c_3$ and so contain one of the sequences $c_1c_2c_3c_4$ or $c_1c_2c_3c_3'$. The sequence $c_1c_2c_3c_4$ contains two consecutive uncoloured edges. If $c_3'c_3$ is uncoloured, then $c_1c_2c_3c_3'$ contains two consecutive uncoloured edges. If $c_3'c_3$ is coloured, then it must be red, and so $c_1c_2c_3c_3'$ contains both red and blue edges.

To see that all $\chi'$-\ref{itm:extend-Hs}-bad cycles are $\chi$-\ref{itm:extend-Hs}-bad, notice that there are no new vertices of uncoloured degree $2$.
\end{proof}

Combining the above three lemmas allows us to eliminate any given bad cycle.
\begin{lemma}\label{Lemma_fix_bad_cycle}
    Let $G$ be a cubic graph and $\chi$ a partial colouring satisfying \ref{itm:extend-degree} and \ref{itm:extend-cycle-technical}. Let $C$ be a $\chi$-\ref{itm:extend-Hs}-bad cycle. 
    
    Then there is a partial colouring $\chi'$ which extends $\chi$, satisfies \ref{itm:extend-degree} and \ref{itm:extend-cycle-technical},  has that all $\chi'$-\ref{itm:extend-Hs}-bad cycles are $\chi$-\ref{itm:extend-Hs}-bad, and has $C$ not $\chi'$-\ref{itm:extend-Hs}-bad. Additionally, every $\chi'$-coloured edge touches $C$.
\end{lemma}
\begin{proof}
    If $C$ touches both colours, then we are done by Lemma~\ref{Lemma_bad_cycle_touches_different_colours}. So $C$ touches just one colour. Without loss of generality, $C$ only touches red edges. If $|C|\geq 5$, then it will have two consecutive vertices touching red edges (since in a $\chi$-\ref{itm:extend-Hs}-bad cycle there are at most two vertices which do not touch coloured edges). Thus, when $|C|\geq 5$, Lemma~\ref{Lemma_adjacent_coloured_edges} applies to give what we want. The remaining case, not covered by  Lemma~\ref{Lemma_adjacent_coloured_edges}, is when $|C|=3$ and $C$ has only one vertex touching red edges (recalling that a $\chi$-\ref{itm:extend-cycle-technical}-bad cycle is odd by definition). Then the remaining two vertices are adjacent vertices of uncoloured degree 2, and so Lemma~\ref{Lemma_adjacent_uncoloured_3} gives what we want.
\end{proof}

By iterating the previous lemma, we can get rid of all $\chi$-\ref{itm:extend-Hs}-bad cycles by extending a colouring. 
\begin{lemma}\label{Lemma_destroy_all_bad_cycles}
    Let $G$ be a cubic graph and $\chi_0$ a partial colouring satisfying \ref{itm:extend-degree} and \ref{itm:extend-cycle-technical}. 
    
    Then there is a partial colouring $\chi'$ which extends $\chi$, satisfies \ref{itm:extend-degree} and \ref{itm:extend-cycle-technical}, and has no $\chi'$-\ref{itm:extend-Hs}-bad cycles. Additionally, every $\chi'$-coloured edge is within distance $2$ of a $\chi$-coloured edge.
\end{lemma}
\begin{proof}
Let $D_1, \dots, D_t$ be the $\chi_1$-\ref{itm:extend-Hs}-bad cycles listed in some order.  Construct colourings $\chi_2, \chi_3, \dots$ as follows. For each $i$, if $D_i$ is $\chi_i$-\ref{itm:extend-Hs}-bad, then apply Lemma~\ref{Lemma_fix_bad_cycle} to $D_i, \chi_i$ to get a colouring $\chi_{i+1}$ extending $\chi_i$. Otherwise, if $D_i$ is not $\chi_i$-\ref{itm:extend-Hs}-bad, then set $\chi_{i+1}=\chi_i$. In either case we have $D_i$ is not $\chi_{i+1}$-\ref{itm:extend-Hs}-bad and that all $\chi_{i+1}$-\ref{itm:extend-Hs}-bad cycles are $\chi_i$-\ref{itm:extend-Hs}-bad. Putting these together gives that $D_1, \dots, D_t$ are not $\chi_{t+1}$-\ref{itm:extend-Hs}-bad  and that all $\chi_{t+1}$-\ref{itm:extend-Hs}-bad cycles are $\chi_1$-\ref{itm:extend-Hs}-bad  i.e.\ there are no $\chi_{t+1}$-\ref{itm:extend-Hs}-bad cycles at all.

For the ``additionally'' part, note that all $\chi_{t+1}$-coloured edges that are not $\chi_1$-coloured touch some $D_i$. In each $D_i$, there are at most two vertices of uncoloured degree $3$ in $\chi_1$, and so every vertex is within distance $2$ of a $\chi_1$-coloured edge.
\end{proof}

The following lemma characterises graphs which do not contain odd cycles with at most three degree 3 vertices.
\begin{lemma}\label{Lemma_subcubic_subdivision}
    Let $G$ be a subcubic graph in which every odd cycle contains at least three vertices of degree 3. Then we can edge-decompose $G$ into $H_1, H_2, H_3$ where $H_2, H_3$ are vertex disjoint and:
    \begin{enumerate}[label = \rm(\arabic*)]
    \item \label{itm:H1} $H_1$ is a subdivision of a simple graph with all degrees $\in \{1,3\}$.
    \item \label{itm:H2} $H_2$ is a union of disjoint even cycles, each of which has at most two vertices in $H_1$.
    \item \label{itm:H3} $H_3$ is a collection of subdivisions of multiedges with multiplicity 3, each of which has no vertices in $H_1$.
    \end{enumerate} 
\end{lemma}


\begin{proof}
    Repeatedly contract vertices of degree $2$ for as long as possible without creating loops. The result is a subcubic multigraph $H$ in which the only degree 2 vertices occur as the endpoint of some multiedge (of multiplicity $2$). 
    Let $G_2, G_3\subseteq G$ be the unions of multiedges of multiplicity $2$ and $3$, respectively. Since $\Delta(G)\leq 3$, we have that the multiedges in $G_2\cup G_3$ are all vertex-disjoint. Let $G_1=G \setminus E(G_2\cup G_3)$, noting that $G_1$ is a simple graph. 
    Also, each multiedge in $G_2$ can trivially intersect at most two vertices of $G_1$.
    We claim that $G_1$ has no vertices of degree $2$. Indeed suppose that $v$ is such a vertex. If $d_G(v)=2$, then $v$ is part of some multiedge which shows that $d_{G_1}(v)=0$. Otherwise, if $d_G(v)=1$ or $3$, note that we either have $d_{G_1}(v)=d_G(v)$
    or  $d_{G_1}(v)\leq d_G(v)-2$. So in all cases, we cannot get $d_{G_1}(v)=2$. 
    
    Uncontract $G_1, G_2, G_3$ to get $H_1, H_2, H_3$ decomposing $G$. Properties \ref{itm:H1} and \ref{itm:H3} are immediate from the construction of $G_1, G_3$. For \ref{itm:H2}, it is immediate that $H_2$ is a union of cycles each of has at most vertices in $H_1$. The property ``every odd cycle contains at least three vertices of degree 3'' of  $G$ shows that all the cycles in $H_2$ are even.
\end{proof}

We are now ready to prove \Cref{Lemma_get_E4}.

\begin{proof}[Proof of \Cref{Lemma_get_E4}]
    Let $\chi'$ be the colouring given by Lemma~\ref{Lemma_destroy_all_bad_cycles}. Properties \ref{itm:extend-degree} and \ref{itm:extend-cycle-technical} are immediate from that lemma, as is ``every $\chi'$-coloured edge is within distance $2$ of a $\chi$-coloured edge''. For \ref{itm:extend-Hs}, recall that Lemma~\ref{Lemma_destroy_all_bad_cycles} and the definition of ``$\chi'$-\ref{itm:extend-Hs}-bad cycle'', tell us that all uncoloured odd cycles contains at least three vertices of uncoloured degree 3. By \Cref{cl:TotalPartialColouringE1E3}, each coloured component in $\chi'$ has size $O(\sqrt{\log n})$, and by Lemma~\ref{Lemma_subcubic_subdivision}, the $\chi'$-uncoloured edges can be partitioned into graph $H_1, H_2, H_3$ having the three properties in \Cref{Lemma_subcubic_subdivision}. Notice that each component $K$ in $H_2$ or $H_3$ has all but at most two of its vertices incident to a coloured edge in $\chi$. Because coloured components in $\chi$ are far away from each other, all the coloured edges touching $K$ belong to the same coloured component, and thus $|K| = O(\sqrt{\log n})$.
\end{proof}

\section{Explicit structure of gadgets} \label{sec:gadgets}

    From this section onward, we will focus on the proof of \Cref{lem:GoodGadgetsExist}, which will be the final element of the proof of the main theorem. Recall that in \Cref{sec:Exact} we gave an abstract definition of a gadget. In this section, we will explicitly describe the gadgets that we find in our graph. In the two subsequent sections, we show that such gadgets can be found in small balls around sufficiently long geodesics: in \Cref{sec:GeodesicNoCommonNghbs} we will prove \Cref{lem:GoodGadgetsExist} for geodesics whose vertices have no common neighbours (which is actually \Cref{lem:GoodCaterpillarGadgetsExist}), and in \Cref{sec:GeodesicWithCommon} we will prove \Cref{lem:GoodGadgetsExist} for the remaining case. 
    Throughout the next three sections, we describe the gadgets both in words and in a figure, with the expectation that the figure will be much easier the understand.

\begin{definition}[Type I gadgets]\label{def:gadget}
    Let $\ell \geq 3$.
    An \emph{$\ell$-gadget of Type I} is a graph $H$, along with two red-blue colourings $\chi_1$ and $\chi_2$, defined as follows. The graph $H$ consists of paths $Q_0, Q_1, Q_2$ (where the ends of $Q_i$ are denoted $Q_i^+, Q_i^-$) and edges $e_1, \ldots, e_5, f_1, \ldots, f_4$, satisfying the following properties (see \Cref{fig:gadget-type-1}).
    \begin{itemize}
        \item 
            $Q_0 = (q_0, \ldots, q_{\ell+2})$, so $Q_0$ has length $\ell+2$,
        \item
            $Q_1$ and $Q_2$ are vertex-disjoint and have the same length,
        \item
            $Q_1^- = q_1$, $Q_2^- = q_3$, and $Q_i \setminus \{Q_i^-\}$ is vertex-disjoint of $Q_0$, for $i \in [2]$,
        \item
            the edges $e_i$ and $f_j$ are edge-disjoint of $Q_0, Q_1, Q_2$, for $i \in [5]$ and $j \in [4]$,
        \item
            $e_1, e_2$ are distinct edges that contain $q_0$; $e_3$ contains $q_2$; $e_4, e_5$ are distinct edges containing $q_{\ell+2}$; $f_1, f_2$ are distinct edges containing $Q_1^+$; and $f_3, f_4$ are distinct edges containing $Q_2^+$.
    \end{itemize}
    We remark that an edge $e_i$ or $f_j$ could intersect vertices of $Q_0 \cup Q_1 \cup Q_2$ unless specified otherwise (e.g.\ $e_1$ cannot intersect $q_1$), and similarly two elements in $(e_1, \ldots, e_5, f_1, \ldots, f_4)$ could be the same edge or be intersecting edges, unless specified otherwise.
    
    Let $\chi_1$ to be the red-blue colouring of $H$, defined as follows.
    \begin{equation*}
    \chi_1(e) = 
    \left\{\begin{array}{ll}
        \red & e \in Q_1 \cup Q_2 \text{ or } e \in \{e_1, \ldots, e_5,\, q_1 q_s\}, \\
        \blue & e \in Q_0 \setminus \{q_1 q_2\} \text{ or } e \in \{f_1, \ldots, f_4\}.
    \end{array}\right.
    \end{equation*}
    Let $\chi_2$ be the red-blue colouring of $H$, obtained from $\chi_1$ by swapping the colours of $q_1q_2$ and $q_2q_3$. Namely,
    \begin{equation*}
    \chi_2(e) = 
    \left\{\begin{array}{ll}
        \red & e \in Q_1 \cup Q_2 \text{ or } e \in \{e_1, \ldots, e_5,\, q_2 q_3\}, \\
        \blue & e \in Q_0 \setminus \{q_2 q_3\} \text{ or } e \in \{f_1, \ldots, f_4\}.
    \end{array}\right.
    \end{equation*}
    \begin{figure}[h]
        \centering
        \begin{subfigure}[b]{.45\textwidth}
            \includegraphics[scale = .8]{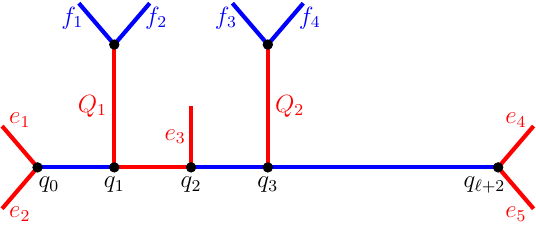} 
            \vspace{.1cm}
            \caption*{The colouring $\chi_1$}
        \end{subfigure}
        \hspace{.4cm}
        \begin{subfigure}[b]{.45\textwidth}
            \includegraphics[scale  = .8]{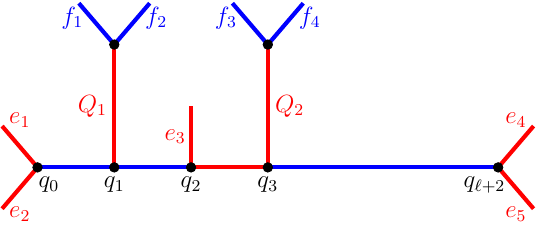} 
            \vspace{.1cm}
            \caption*{The colouring $\chi_2$}
        \end{subfigure}
        \caption{The two colourings of a Type I gadget}
        \label{fig:gadget-type-1}
    \end{figure} 
\end{definition}
    
\begin{observation}\label{obs:gadget-type-1}
    An $\ell$-gadget of Type I is a blue $\ell$-gadget, for $\ell \ge 3$.
\end{observation}

\begin{proof}
    Consider an $\ell$-gadget of Type I, defined by $H, \chi_1, \chi_2$. Let $G$ be a red-blue cubic graph, whose monochromatic components are paths, and suppose that $G$ contains a copy $H'$ of $(H, \chi_1)$. Consider the red-blue coloured graph $G'$ obtained from $G$ by replacing the colouring of $H'$ by $(H, \chi_2)$. Let $\cP_b, \cP_r$ be the collections of blue and red components in $G$, and let $\cP_b', \cP_r'$ be the collections of blue and red components in $G'$. It is easy to check that 
    \begin{align*}
        \cP_b' & = \left(\cP_b \setminus \{q_0 q_1,\, q_2 \ldots q_{\ell+2}\}\right) \cup \{q_0 q_1 q_2, \, q_3 \ldots q_{\ell+2}\}   \\
        \cP_r' & = \left(\cP_r \setminus \{Q_2, \,\text{the red component in $G$ containing $Q_1 q_1 q_2 e_3$}\} \right) \\
        & \qquad \,\,\cup \{Q_1, \,\text{the red component in $G'$ containing $Q_2 q_3 q_2 e_3$}\}.
    \end{align*}
    Since $Q_1$ and $Q_2$ have the same length, the red component in $G$ containing $Q_1 q_1 q_2 e_3$ has the same length as the red component in $G'$ containing $Q_2 q_3 q_2 e_3$.
    Thus $r_G(P_k) = r_{G'}(P_k)$ for every $k$, and
    \begin{equation*}
        b_{G'}(P_k) = \left\{
            \begin{array}{ll}
                b_G(P_k) & k \neq \ell, \ell-1, 1, 2 \\
                b_G(P_k) + 1 & k \in \{\ell-1, 2\} \\ 
                b_G(P_k) - 1 & k \in \{\ell, 1\}.
            \end{array}
        \right.
    \end{equation*}
    Notice also that $(H, \chi_1)$ and $(H, \chi_2)$ differ on exactly two edges, and that $(H, \chi_1)$ has a maximal blue path of length $\ell$ whose ends are incident with two red edges.
    It follows that $(H, \chi_1, \chi_2)$ is a blue $\ell$-gadget.
\end{proof}

\begin{definition}[Type II gadgets]
    Let $\ell \ge 3$. An \emph{$\ell$-gadget of Type II} is a graph $H$, along with two red-blue colourings $\chi_1$ and $\chi_2$ of $H$, defined as follows. The graph $H$ consists of paths $Q_0, Q_1$ (where the ends of $Q_1$ are denoted $Q_1^+, Q_1^-$) and edges $e_1, \ldots, e_5$, such that
    \begin{itemize}
        \item 
            $Q_0 = (q_0 \ldots q_{\ell+2})$, so $Q_0$ has length $\ell+2$,
        \item
            $Q_1^- = q_1$ and $Q_1^+ = q_3$, and $Q_1 \setminus \{q_1^+, q_1^-\}$ is vertex disjoint of $Q_0$,
        \item
            the edges $e_i$ are edge-disjoint of $Q_0, Q_1$, for $i \in [5]$,
        \item
            $e_1, e_2$ are distinct edges containing $q_0$; $e_3$ contains $q_2$; and $e_4, e_5$ are distinct edges containing $q_{\ell+2}$.
    \end{itemize}
    Let $\chi_1$ be the red-blue colouring of $H$, defined as follows.
    \begin{equation*}
        \chi_1(e) = \left\{
            \begin{array}{ll}
                \red & e \in Q_1 \text{ or } e \in \{e_1, \ldots, e_5, \,q_1q_2\}, \\
                \blue & e \in Q_0 \setminus \{q_1 q_2\}.
            \end{array}
        \right. 
    \end{equation*}
    Let $\chi_2$ be the red-blue colouring of $H$, obtained from $\chi_1$ by swapping the colours of $q_1 q_2$ and $q_2 q_3$. Namely,
    \begin{equation*}
        \chi_2(e) = \left\{
            \begin{array}{ll}
                \red & e \in Q_1 \text{ or } e \in \{e_1, \ldots, e_5, \,q_2q_3\}, \\
                \blue & e \in Q_0 \setminus \{q_2 q_3\}.
            \end{array}
        \right. 
    \end{equation*}
    \begin{figure}[h!]
        \centering
        \begin{subfigure}[b]{.45\textwidth}
            \includegraphics[scale = .8]{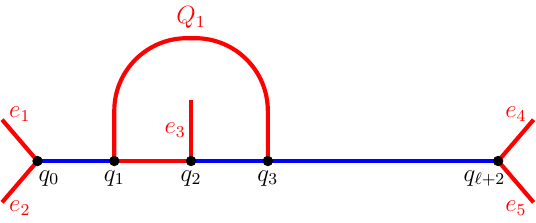} 
            \caption*{The colouring $\chi_1$}
        \end{subfigure}
        \hspace{.4cm}
        \begin{subfigure}[b]{.45\textwidth}
            \includegraphics[scale = .8]{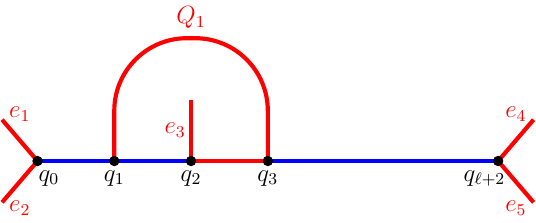}
            \caption*{The colouring $\chi_2$}
        \end{subfigure}
        \caption{The two colourings of a Type II gadget}
        \label{fig:gadget-type-2}
    \end{figure}
\end{definition}

The following observation can be proved similarly to \Cref{obs:gadget-type-1}; we omit the details.
\begin{observation}
    An $\ell$-gadget of Type II is a blue $\ell$-gadget, for $\ell \ge 3$.
\end{observation}

\section{Geodesic with no common neighbours (\Cref{lem:GoodCaterpillarGadgetsExist})}\label{sec:GeodesicNoCommonNghbs}

 In this section we prove \Cref{lem:GoodCaterpillarGadgetsExist}, which finds a gadget in a small ball around a sufficiently long geodesic whose vertices do not have common neighbours (see \Cref{lem:comb-colouring}).
 
 Recall that here we assume that $P$ is a geodesic of length $\ell + 36$, no two of whose vertices have a common neighbour outside of $P$. 
 Throughout this section, fix such a geodesic $P$.
 Write $L = \ell + 36$ and $P = (p_0 \ldots p_{L})$, and denote by $p_i'$ the unique neighbour of $p_i$ outside of $P$, for $i \in [L-1]$; so the $p_i$'s are distinct.

\begin{claim}\label{cl:CaterpillarCases}
    One of the following holds.
    \begin{enumerate}[label = \rm\textbf{\Roman*.}, ref = \rm\textbf{\Roman*}]
        \item \label{case:alg-comb-a}
            There exists $s \in [10, L - \ell - 10]$ such that $p'_{s-1}$ and $p'_{s+1}$ have no common neighbours.
        \item \label{case:alg-comb-b}
           For every $i \in [12, L - \ell - 12]$, the vertices $p_{i-1}'$ and $p_{i+1}'$ have a common neighbour in $\{p_{i-2}', p_i', p_{i+2}'\}$, and $\{p_j' : j \in [14, L - \ell - 14]\}$ induces no cycles.
        \item \label{case:alg-comb-c}
            There exists $s \in [12, L - \ell - 10]$ such that $p'_{s-1}$ and $p'_{s+1}$ have a unique common neighbour $w$, and there is no $i$ such that $w = p_i'$.
    \end{enumerate}
\end{claim}

\begin{proof}
    Assume that Case \ref{case:alg-comb-a} does not hold; then $p_{i-1}'$ and $p_{i+1}'$ have a common neighbour for all $i \in [10, L - \ell - 10]$. We claim that if there is a pair $p_{i-1}', p_{i+1}'$, with $i \in [12, L - \ell -12]$, with two common neighbours, then Case \ref{case:alg-comb-c} must hold. Indeed, suppose that for some $i \in [12, L - \ell - 12]$ the vertices $p_{i-1}'$ and $p_{i+1}'$ have two common neighbours $x, y$. Then, since $p_{i+1}'$ and $p_{i+3}'$ also have a common neighbour, it must be either $x$ or $y$, as otherwise $p_{i+1}'$ would have degree 4. The same holds for $p'_{i-3}$ and $p'_{i-1}$. Without loss of generality, $x$ is joined to $p_{i+3}'$. But then $x \notin \{p_j' : j \in [L-1]\}$, because otherwise $x$ would have degree at least $4$. Moreover, $p_{i+1}$ and $p_{i+3}$ have no common neighbour other than $x$ (because that neighbour would have to be $y$, making $y$'s degree $4$). Thus Case \ref{case:alg-comb-c} holds for $s = i+2$. 
    
    So now we are left with the case where all pairs of vertices $p'_{i-1}$, $p'_{i+1}$, with $i \in [12, L - \ell - 12]$, have exactly one common neighbour. If for one such $i$ the common neighbour is not in $\{p_j' : j \in [L-1]\}$, then Case \ref{case:alg-comb-c} holds, so suppose otherwise. Namely, for $i \in [12, L - \ell - 12]$, the unique common neighbour of $p'_{i-1}$ and $p'_{i+1}$ is $p_j'$, for some $j \in [L-1]$. Clearly, $j \notin \{i-1, i+1\}$. Also, $j$ is not larger than $i+2$ or smaller than $i-2$ because $P$ is a geodesic. Thus $j \in \{i-2, i, i+2\}$.
    It follows that, for every $i \in [14, L - \ell - 14]$, the vertex $p_i'$ has exactly two neighbours in $\{p_{i-3}', p_{i-1}', p_{i+1}', p_{i+3}'\}$ (because it has a common neighbour with both $p_{i+2}'$ and $p_{i-2}'$, and that cannot be the same neighbour because such a neighbour would have degree at least $4$). 
    We show that $\{p_j' : j \in [14, L - \ell - 14]\}$ has no cycles. Suppose to the contrary that $C$ is a cycle in this set, and let $j$ be the least index of a vertex in $C$. Then $p_j'$ is joined to both $p_{j+1}'$ and $p_{j+3}'$, and $p_{j+1}'$ has is joined to either $p_{j+2}'$ or $p_{j+4}'$. It follows that $p_{j-2}'$ and $p_j'$ have no common neighbour (otherwise some vertex would have degree larger than $3$), a contradiction.
    Thus Case \ref{case:alg-comb-b} holds.
\end{proof}

We now show how to proceed in each of the above three cases.
        
\begin{algorithm}\label{alg:comb}
    Let $P$ be as above. 
    \begin{enumerate}[label = \rm\textbf{\Roman*.}, ref = \rm\Roman*]
        \item \label{step:alg-comb-a}
            If Case \ref{case:alg-comb-a} of \Cref{cl:CaterpillarCases} holds, fix $s$ for which it holds, and do the following.
            \begin{itemize}
                \item
                    Colour the edges $p_{s-3}p_{s-2}$, $p_{s-1}p_s$ and $p_{s+\ell} p_{s+\ell+1}$ red, and colour all other edges in the path $(p_{s-10} \ldots p_{s+\ell+10})$ blue.
                \item
                    Colour the edges $p_{s-3} p_{s-3}'$ and $p_{s+\ell}p_{s+\ell}'$ blue, and colour all other edges $p_ip_i'$ with $i \in [s-9, s+\ell+9]$ red.
                \item
                    Denote the two uncoloured neighbours of $p_{s-1}'$ by $u, u'$ and the two uncoloured neighbours of $p_{s+1}'$ by $w, w'$.
                    Colour the edges $p_{s-1}' u, p_{s-1}' u', p_{s+1} w, p_{s+1} w'$ blue (note that these are three or four edges, depending on whether or not $p_{s-1} p_{s+1}$ is an edge).
                \item   
                    Colour all uncoloured edges with both end points in $\{u, u', w, w'\}$ red, unless they contain a cycle (of length $3$ or $4$), in which case we colour one edge in a cycle, with one end in $\{u, u'\}$ and one end in $\{w, w'\}$, blue, and the rest red.
                \item
                    This gives a blue $\ell$-gadget of type II, with $Q_0 = (p_{s-2} \ldots p_{s+\ell})$, $Q_1 = p_{s-1} p_{s-1'}$, and $Q_2 = p_{s+1} p_{s+1}'$; $e_1 = p_{s-3} p_{s-2}$, $e_2 = p_{s-2} p_{s-2}'$, $e_3 = p_sp_s'$, $e_4 = p_{s+\ell} p_{s+\ell+1}$, and $e_5 = p_{s+\ell} p_{s+\ell}'$; $f_1 = p_{s-1}' u$, $f_2 = p_{s-1}' u'$, $f_3 = p_{s+1}' w$ and $f_4 = p_{s+1}' w'$.
            \end{itemize}
            
            \begin{figure}[h]
                \centering
                \includegraphics[scale = .78]{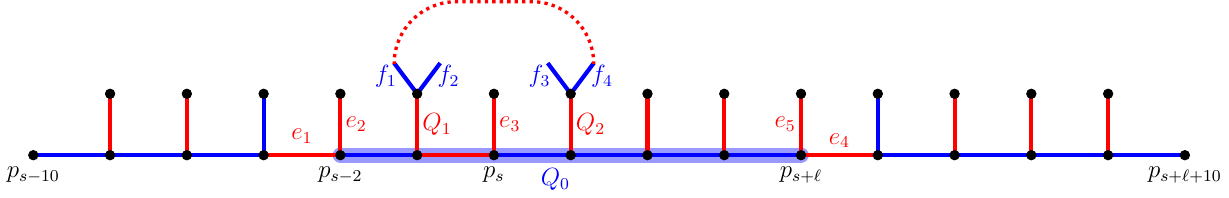}
                \caption{Case \ref{case:alg-comb-a}}
                \label{fig:new-case-1}
            \end{figure}
        \item \label{step:alg-comb-b}
        If Case \ref{case:alg-comb-b} of \Cref{cl:CaterpillarCases} holds, take $s = 18$, and do the following.
            \begin{itemize}
                \item 
                    Colour the edges $p_{s-3} p_{s-2}$, $p_{s-1} p_s$ and $p_{s+\ell} p_{s+\ell+1}$ red, and colour all other edges in the path $(p_{s-2} \ldots p_{s+\ell})$ blue.
                \item
                    Colour all edges $p_i p_i'$, with $i \in [s-2, s+\ell]$, red.
                \item Colour the four edges touching $p'_{s-1}$ and $p'_{s+1}$ and not touching $P$ blue.
                \item
                    This gives a blue $\ell$-gadget of type II, with $Q_0 = (p_{s-2} \ldots p_{s+\ell})$, $Q_1 = p_{s-1} p_{s-1}'$, and $Q_2 = p_{s+1} p_{s+1}'$; $e_1 = p_{s-3} p_{s-2}$, $e_2 = p_{s-2} p_{s-2}'$, $e_3 = p_sp_s'$, $e_4 = p_{s+\ell} p_{s+\ell+1}$, and $e_5 = p_{s+\ell} p_{s+\ell}'$; $f_1, f_2$ are the two blue edges at $p_{s-1}'$, and $f_3, f_4$ are the two blue edges at $p_{s+1}'$.
            \end{itemize}
            \begin{figure}[h]
                \centering
                \includegraphics[scale = .6]{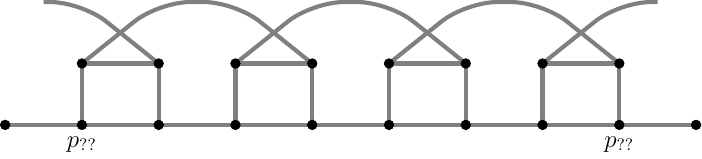}
                \hspace{.5cm}
                \includegraphics[scale = .6]{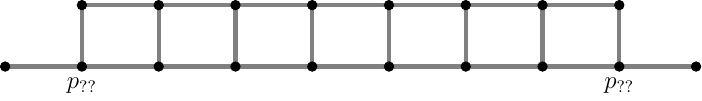}
                
                \vspace{.3cm}
                \includegraphics[scale = .6]{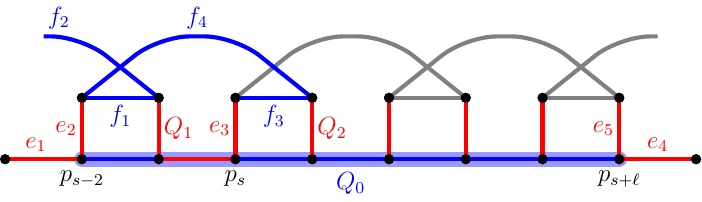}
                \hspace{.5cm}
                \includegraphics[scale = .6]{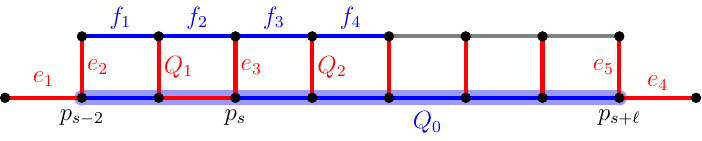}
                \caption{Case \ref{case:alg-comb-b}}
                \label{fig:new-case-2}
            \end{figure}
        \item \label{step:alg-comb-c}
           If Case \ref{case:alg-comb-c} of \Cref{cl:CaterpillarCases} holds, fix $s$ for which it does, and do the following.
           \begin{itemize}
                \item 
                    Colour the edges $p_{s-3} p_{s-2}$, $p_{s-1} p_s$ and $p_{s+\ell} p_{s+\ell+1}$ red, and colour all other edges in the path $(p_{s-2} \ldots p_{s+\ell})$ blue.
                \item
                    Colour all edges $p_i p_i'$, with $i \in [s-2, s+\ell]$, red.
                \item 
                    Colour the edges $p_{s-1}' w$ and $p_{s+1}' w$ red.
                \item
                    Repeat the following steps as long as possible:
                    \begin{itemize}
                        \item 
                            Colour all uncoloured edges incident with a vertex, whose red degree is $2$, blue.
                        \item
                            Colour all uncoloured edges incident with a vertex, whose blue degree is $2$, red.
                    \end{itemize}
                \item
                    This yields a blue $\ell$-gadget of type 2, with $Q_0 = (p_{s-2} \ldots p_{s+\ell})$ and $Q_1 = (p_{s-1} p_{s-1}' w p_{s+1}' p_{s+1})$; $e_1 = p_{s-3} p_{s-2}$, $e_2 = p_{s-2}p_{s-2}'$, $e_3 = p_s p_s'$, $e_4 = p_{s+\ell} p_{s+\ell+1}$ and $e_5 = p_{s+\ell} p_{s+\ell}'$.
            \end{itemize}
            \begin{figure}[h]
                \centering
                \includegraphics[scale = .8]{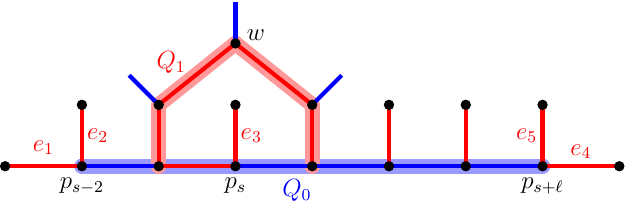}
                \caption{Case \ref{case:alg-comb-c}}
                \label{fig:new-case-3}
            \end{figure}
    \end{enumerate}

\end{algorithm}
    \begin{lemma}\label{lem:comb-colouring}
        The partial colouring $\phi$ obtained by \Cref{alg:comb} satisfies the following, and thus satisfies the requirements of \Cref{lem:GoodCaterpillarGadgetsExist}.
        \begin{enumerate}[label = \rm(\arabic*)]
            \item \label{itm:comb-phi-3-dist}
                all coloured edges are at distance at most 3 from the geodesic $P$ and form a connected subgraph,
            \item \label{itm:comb-phi-3-deg}
                if a vertex $u$ is incident to at least two coloured edges, then at least one of them is red and at least one of them is blue,
            \item \label{itm:comb-phi-3-cycle}
                if $C$ is a cycle in $G$ containing no red edges or no blue edges, then $C$ has two consecutive edges that are not coloured by $\phi_3$.
        \end{enumerate}
    \end{lemma}
    
    \begin{proof}
        We call a cycle that has no red edges or no blue edges a \textit{single-colour cycle} (note that such a cycle can have uncoloured edges, and so this is not the same as a monochromatic cycle). We call a vertex with at least two coloured edges, all in the same colour, a \textit{single-colour} vertex.
        
        We first point out that the last step in Part \ref{step:alg-comb-c} of \Cref{alg:comb} does not last very long, leaves no single-colour vertices, and the edges it colours are not involved in single-colour cycles. The key observation is that initially three blue edges are added, not all sharing a vertex (if they did all share a vertex, then $p_{s-1}'$ and $p_{s+1}'$ would have two common neighbours, contrary to the specification of the Case \ref{case:alg-comb-c}), and from then onward, there is always at most one single-colour vertex. Since each iteration in this step colours exactly one edge, this means that such a single-colour vertex has coloured degree exactly $2$. Thus, at the end of the step, there are no single-colour vertices. 
        In addition, since every newly coloured edge has one endpoint with degree 2 in the opposite colour, none of them is involved in a single-colour cycle.
        It remains to show that the step has few iterations.
        If no two of the initial three blue edges share a vertex then the process stops immediately. Otherwise, since we started with only two vertices with blue degree at least $1$ and colour degree at most $2$, it is easy to check that during the last step of \ref{step:alg-comb-c} we will repeat the second item at most twice. This means that the process has at most four iterations. Since each iteration involves at least one edge which was already coloured in previous steps of the algorithm, all edges coloured in this step are at distance at most 3 from $P$.
        
        Items \ref{itm:comb-phi-3-dist} and \ref{itm:comb-phi-3-deg} are immediate in Case \ref{case:alg-comb-a} and \ref{case:alg-comb-b} and follow from the last paragraph otherwise.
        
        It remains to prove \ref{itm:comb-phi-3-cycle}. For this, first note that if an endpoint of a monochromatic path has both of its other edges coloured in the opposite colour, then this path cannot participate in a single-coloured cycle. Another important observation is that, since $P$ is a geodesic,  there are no edges of the form $p'_jp'_{j+r}$ when $r\geq 4$. So connecting two such vertices will require more than one edge. We consider each case separately. 
        
        Case \ref{case:alg-comb-a}. \,
        All red edges, except for those with two ends in $\{u, u', w, w'\}$, belong to a red path with an end with blue degree 2. Since all the edges in $\{u, u', w, w'\}$ are coloured, and do not contain a red cycle (since this is a graph with maximum degree $2$ and four vertices, there is at most one cycle and such a cycle is not red due to the fourth step in \ref{step:alg-comb-a} in \Cref{alg:comb}. Thus any $C$  with no blue edges has all its red edges in $\{u, u', w, w'\}$ and has at least one vertex $v \notin \{u, u', w, w'\}$, showing that it has two consecutive uncoloured edges.
        Therefore, it remains to consider cycles with no red edges.
        Suppose that $C$ is such a cycle. If it contains the left-most blue path (starting from $p_{s-3}'$), then, because $p_{s-10}$ is at distance at least $5$ from all other blue edges which might participate in a single-colour cycle, $C$ contains at least 5 consecutive uncoloured edges.
        A similar argument works if $C$ contains the right-most path. So it remains to consider the case where $C$ contains $f_1 f_2$ or $f_3 f_4$ and not other blue edges. But there is no cycle with vertices in $\bigcup_{i \in [4]} V(f_i)$ that has only blue edges (due to the fourth step in this case).
        
        Case \ref{case:alg-comb-b}. \, 
        If $C$ is a cycle with blue edges but no red edges, then its blue edges are in $\{f_1, f_2, f_3, f_4\}$. 
        By the specifications of \Cref{case:alg-comb-b}, there are no cycles in the set $\bigcup_{i \in [4]}V(f_i)$, and so $C$ must contain a vertex outside of this set, showing that $C$ has two consecutive uncoloured edges.
        If $C$ has only red edges then it contains one (or both) of the paths $e_1 e_2$ and $e_3 e_4$ (and no other red edges). But if it contains only one of them, then it has length at least $4$ (using the assumption on $P$), and if it contains both, then it has length at least $10$ (because $p_{s-2}$ and $p_{s+\ell}$ are at distance at least $5$ from each other). Either way, it has two consecutive uncoloured edges.
        
        Case \ref{case:alg-comb-c}. \,
        There is no cycle with no red edges. If $C$ is a cycle with red edges but not blue edges, then it contains one or both of $e_1 e_2$ and $e_3 e_4$ (any red edge added in the last step has an end with two blue edges). As in the previous case, $C$ contains two consecutive uncoloured edges.
    \end{proof}

\section{Geodesic with common neighbours (\Cref{lem:GoodGadgetsExist})}\label{sec:GeodesicWithCommon}
This section will be devoted to proving \Cref{lem:GoodGadgetsExist}, which finds a gadget in a small ball around a sufficiently long geodesic (see \Cref{lem:PartialColProp}). Given \Cref{lem:GoodCaterpillarGadgetsExist}, which does the same for geodesics whose vertices have no common neighbours, we only need to prove the former lemma for a geodesic $P$ many pairs of its vertices have a common neighbour outside of $P$. For this, we define a partial colouring process of the gadget and its neighbourhood. 
The partial colouring process is divided into three stages, each will be described in a different subsection.
\begin{itemize}
    \item{\textbf{Stage I} (\Cref{sec:StageI})}\textbf{.} Initial partial colouring. In this stage we partially colour a small ball around them according to \Cref{alg:preColouring}, yielding a long blue path and avoiding vertices with red degree $2$ and blue degree $0$ or vice versa. This colouring will be denoted $\phi_1$.
    A lot of this work in this section will be devoted to the analysis of the colouring $\phi_1$, particularly to showing that monochromatic degree $3$ vertices only arise in a specific structure.
    \item{\textbf{Stage II} (\Cref{sec:StageII})}\textbf{.} Fixing monochromatic vertices. In this stage we modify the colouring $\phi_1$ (in \Cref{alg:FixPreCol}) so as to avoid vertices of monochromatic degree $3$, by recolouring some of the edges around each vertex of monochromatic degree $3$ in $\phi_1$. This colouring will be denoted $\phi_2$.
    \item{\textbf{Stage III} (\Cref{sec:StageIII})}\textbf{.} Creating a gadget. In this stage we will further modify the colouring $\phi_2$ (in \Cref{alg:CreateGadgetsInPreCol}) to create a copy of a type II gadget. This partial colouring will be denoted $\phi_3$.
\end{itemize}

We will finish this section with the proof of \Cref{lem:GoodGadgetsExist}.

\subsection{Stage I.\, Initial partial colouring (\Cref{alg:preColouring})}\label{sec:StageI}
The purpose of this section is to present \Cref{alg:preColouring} that generates a partial colouring $\phi_1$ around a geodesic $P$, and prove key properties of $\phi_1$. 
The algorithm starts in \ref{itm:step-11} by colouring the edges of $P$ blue, and then, in \ref{itm:step-12} and for technical reasons, modifying this colouring slightly for any pair of consecutive vertices in $P$ that have a common neighbour. Next, in \ref{itm:step-13}, \ref{itm:step-14} and \ref{itm:step-15}, whenever we encounter a vertex with two edges of the same colour and a third edge uncoloured, we coloured the third edge with the opposite colour. The purpose of this should be clear: we want to make sure that each vertex with coloured degree at least $2$ is incident to red and blue edges. At the end of the process, namely in \ref{itm:step-16}, we recolour some of the edges touching the ends of $P$, again for the same purpose but also keeping in mind an additional property (namely, \ref{itm:mono-vx-path} below). 

Throughout the process we orient all edges as they are coloured, except for edges of $P$ and the edges coloured at the last step of the process. These orientations are instrumental when proving various properties of the colouring. Additionally, we output a blue path $P'$, modified from $P$ in \ref{itm:step-12}. This path $P'$ (or, more precisely, a modification of $P'$ obtained in the next section) will be the basis of the gadget that we will eventually construct.

\begin{algorithm}[Initial partial colouring]\label{alg:preColouring}  
\hfill

{\bf Input.} A geodesic $P$ in a cubic graph $G$.

{\bf Output.} A partial red-blue colouring $\phi_1$ of $G$; an oriented graph $D$ whose underlying graph is a subgraph of $G$; and a path $P'$.

{\bf Initialisation.} Set $P' = P$; no edges of $G$ are coloured; $D$ is empty.

{\bf Procedure.}
First, perform \ref{itm:step-11} below. Then, repeat \ref{itm:step-12} as long as possible, and, afterwards, repeat \ref{itm:step-13} as long as possible. Next, repeat \ref{itm:step-14} and \ref{itm:step-15} as long as possible, prioritising the former. 
Finally, repeat \ref{itm:step-16}.

\begin{enumerate}[label = \rm(S\arabic*)]
    \item \label{itm:step-11}
        Colour the edges in $P$ blue.
    \item \label{itm:step-12}
        Given this configuration: \raisebox{-11pt}{\includegraphics[scale = .95]{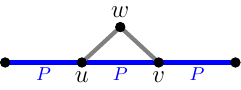}}, recolour like this: \raisebox{-11pt}{\includegraphics[scale = .95]{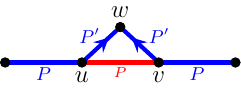}}.
        
        Namely, suppose: $uv$ is a blue edge in the interior of $P$; $uw, wv$ are uncoloured edges; $w \notin V(P)$. 
        
        Then: recolour $uv$ red and $uw, wv$ blue; add $uw, vw$ to $D$; update $P'\leftarrow (P'\setminus\{uv\})\cup\{uw,vw\}$. 
        
    \item \label{itm:step-13}
        Given this configuration: \raisebox{-11pt}{\includegraphics[]{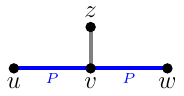}}, colour like this: \raisebox{-11pt}{\includegraphics[]{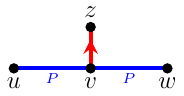}}.
        
        Namely, suppose: $uv, vw$ are blue edges in $P$; $vz$ is uncoloured.
        
        Then: colour $vz$ red; add $vz$ to $D$. 
        
    \item \label{itm:step-14}
        Given this configuration: \raisebox{-11pt}{\includegraphics[]{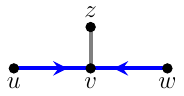}}, colour like this: \raisebox{-11pt}{\includegraphics[]{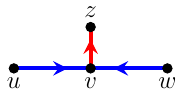}}.
        
        Namely, suppose: $uv, wv$ are directed blue edges; $vz$ uncoloured.
        
        Then: colour $vz$ red; add $vz$ to $D$.
        
    \item \label{itm:step-15}
        Given this: \raisebox{-11pt}{\includegraphics[]{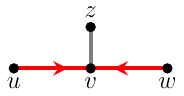}}, colour like this: \raisebox{-11pt}{\includegraphics[]{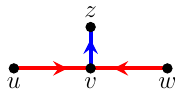}}.
        
        Namely, suppose: $uv, wv$ directed red edges; $vz$ uncoloured.
        
        Then: colour $vz$ blue, add $vz$ to $D$.
    
    \item \label{itm:step-16}
        Given this: \raisebox{-11pt}{\includegraphics[]{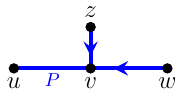}}, recolour like this: \raisebox{-11pt}{\includegraphics[]{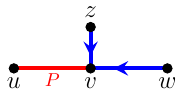}}.
        
        Namely, suppose: $uv$ is a blue edge in $P$; $wv$ and $zv$ two directed blue edges.
        
        Then: recolour $uv$ red and set $P' \leftarrow P' \setminus \{uv\}$.
        
        Given this: \raisebox{-11pt}{\includegraphics[]{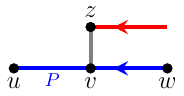}}, recolour like this: \raisebox{-11pt}{\includegraphics[]{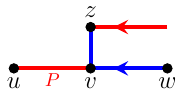}}.
        
        Namely, suppose: $uv$ blue edge in $P$; $wv$ directed blue edge; $vz$ uncoloured edge; $z$ has a red in-neighbour.
        
        Then: colour $vz$ blue, recolour $uv$ red and set $P' \leftarrow P' \setminus \{uv\}$.
        
        Given this: \raisebox{-11pt}{\includegraphics[]{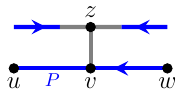}}, recolour like this: \raisebox{-11pt}{\includegraphics[]{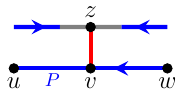}}.
        
        Namely, suppose: $uv$ blue edge in $P$; $wv$ directed blue edge; $vz$ uncoloured edge; each of the other two edges at $z$ are either blue and directed into $z$ or uncoloured.
        
        Then: colour $vz$ red.
\end{enumerate}
\end{algorithm}

In \Cref{lem:OnlyMonoVertex} below we state various properties of any partial colouring obtained by running \Cref{alg:preColouring}. One important property considers vertices with monochromatic degree $3$. It turns out that they only appear given one specific structure, which we describe now.

\begin{definition}\label{def:OnlyMonoVertex}

    Let $\Fexc = $ \raisebox{-11pt}{\includegraphics[]{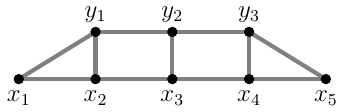}}. Namely, $F$ is the graph on vertices $\{x_1, \ldots, x_5, y_1, y_2, y_3\}$ and edges
    $\{x_1x_2,x_2x_3,x_3x_4,x_4x_5,y_1y_2,y_2y_3,x_1y_1,x_2y_1,x_3y_2,x_4y_3,x_5y_3 \}$. 
    
    Given a path $P$ in a graph $G$, a copy of $\Fexc$ is \emph{parallel} to $P$ if $x_1, \ldots, x_5$ are consecutive vertices in $P$ and $y_1, y_2, y_3 \notin V(P)$. 
    
    Let $\Fexcs = $ \raisebox{-11pt}{\includegraphics[]{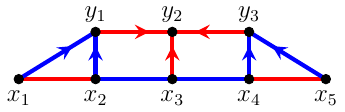}}. 
    Namely, $\Fexcs$ is a red-blue coloured and partially oriented copy of $\Fexc$, where the edges $y_1 y_2, x_3 y_2, y_3 y_2$ are red and directed; $x_1 y_1, x_2 y_1, x_4 y_3, x_5 y_3$ are blue and directed; $x_1 x_2, x_4 x_5$ are red and undirected; and $x_2 x_3, x_3 x_4$ are blue and undirected. 
\end{definition}

Here is the main lemma of the subsection.

\begin{lemma}\label{lem:OnlyMonoVertex}
    Let $G$ be a cubic graph and let $P$ be a geodesic in $G$.
    Consider the colouring and orientation process described in \Cref{alg:preColouring}, applied with $P$. 
    Then 
    \begin{enumerate}[label = \rm(\arabic*)]
        \item \label{itm:mono-vx-two}
            if a vertex is incident to exactly two coloured edges then one of them is red and the other blue,
        \item \label{itm:mono-vx-path}
            every maximal monochromatic path, except for possibly $P'$, has an end touching two edges of the opposite colour,
        \item \label{itm:mono-vx-cycle}
            there are no monochromatic cycles,
        \item \label{itm:mono-vx-dist}
            all coloured edges are at distance at most $3$ from $P$, 
        \item \label{itm:mono-vx-F}
            a vertex has monochromatic degree $3$ if and only if it is the vertex $y_2$ in a copy of $\Fexcs$ which is parallel to $P$.
    \end{enumerate}
\end{lemma}

\begin{proof}[Proof of \Cref{lem:OnlyMonoVertex}]
    We start by proving \ref{itm:mono-vx-two}, namely that every vertex which is incident to exactly two coloured edges, is incident to one red edge and one blue edge. To see this, notice that interior vertices of $P$ have all three edges touching them coloured in \ref{itm:step-11}, \ref{itm:step-12} and \ref{itm:step-13}. Moreover, \ref{itm:step-16} prevents the ends of $P$ from being incident to exactly two coloured edges, both of the same colour. It remains to consider vertices that are not in $P$. Let $v$ be such a vertex, and suppose that it is incident to exactly two colour edges. Then they are both directed (because the only undirected edges are either in $P$ or introduced in \ref{itm:step-16} and do not create vertices of red degree 2 and blue degree 0 or vice versa) and they are directed into $v$ (because whenever a directed edge $xy$ is introduced and coloured, $x$ is already incident to two edges of the other colour). Thus, these edges have different colours (otherwise, the third edges would be coloured in \ref{itm:step-14} or \ref{itm:step-15}), as required for \ref{itm:mono-vx-two}.
    
    Next, we note that right before \ref{itm:step-16} is performed, every coloured edge which is not in $P'$ has a vertex which is incident to two edges of the opposite colour; in particular, if $xy$ is directed and not in $P'$ then $x$ is incident to two edges of the opposite colour. 
    Now, taking $u, v, w, z$ to be as in \ref{itm:step-16}, if $uv$ is recoloured red then $v$ is incident with two blue edges, and so the other red edge at $u$ is in a red path with an end incident to two blue edges. 
    Similarly, if $vz$ is coloured in this step then either it is coloured red and $v$ has two blue neighbours, or it is coloured blue and intersects a blue edge (namely $vw$) with an end with two red neighbours.
    We conclude that every maximal monochromatic path, except for possibly $P'$, has an end with two neighbours of the opposite colour, as stated in \ref{itm:mono-vx-path}.
    
    Now, by the last paragraph, no edges outside of $P'$ are contained in a monochromatic cycle. Since $P'$ is a path, it follows that there are no monochromatic cycles, namely \ref{itm:mono-vx-cycle} holds.
    
    For the remaining items, we will look at the digraph $D$ of edges directed by \Cref{alg:preColouring}. We will show that this subgraph has out-degree at most 1, and that, ignoring directions, it is a forest. This forest will turn out to have a very specific structure: its leaves are vertices of $P$ and all its vertices, except for (maybe) the roots, have degree 1 or 3. In addition, the diameter of each tree in the forest is at least the number of leaves in the tree minus 1. Using these structural properties, we will show that the only ``problematic" trees are ``caterpillars", and among them only the one defined in \Cref{def:OnlyMonoVertex} yields a monochromatic vertex. 

    \begin{claim}\label{cl:MaxDegInTree}
        $D$ has maximum out-degree $1$, and its underlying graph is a forest.
    \end{claim}
    \begin{proof}
        Suppose that $xy$ is a directed edge in $D$. Then when \Cref{alg:preColouring} directed this edge, the other edges touching $x$ were already coloured, and they were either undirected or directed towards $x$. Thus $x$ cannot have an additional out-neighbour besides $y$, showing that $D$ has maximum out-degree (at most) $1$.
        
        Suppose that $(x_1 x_2 \ldots x_t x_1)$ is a cycle in the underlying graph of $D$. By the maximum out-degree being $1$, this is, in fact, a directed cycle in $D$. Assume that the direction of the edges follows the direction of the cycle, namely $x_1 x_2 \in E(D)$, etc.\ (if this is not the case, we can look at $(x_t x_{t-1} \ldots x_1 x_t)$ instead).
        Without loss of generality, $x_1x_2$ was the last edge to be coloured. Then $x_2 x_3$ was directed before $x_1 x_2$ was coloured, a contradiction.
    \end{proof}

\begin{claim}\label{cl:StructureOfTrees}
    Let $T$ be a component of $D$. Then
    \begin{enumerate}[label = \rm(\roman*)]
        \item \label{itm:tree-directed}
            $T$ is a tree whose edges are directed towards a single vertex, namely, a sink (which we think of as the root of the tree),
        \item \label{itm:tree-indegree}
            every non-root vertex in $D$ has in-degree $0$ or $2$,
        \item \label{itm:tree-leaves}
            all non-root leaves of $D$ lie in the interior of $P$,
        \item \label{itm:tree-root}
            if the root is in $P$ then it is an end of $P$ and incident to both red and blue edges.
    \end{enumerate}
\end{claim}

\begin{proof}
    Note that if the minimum out-degree of $T$ is 1, then $T$ contains a cycle, contradicting \Cref{cl:MaxDegInTree}. Thus $T$ has a vertex with out-degree 0, namely, a sink $s$. Let $S$ be the set of vertices for which there is a directed path from them to $s$. Assume that $S\cup \{s\}\neq V(T)$ and let $U := V(T) \setminus (S \cup \{s\})$. Since $T$ is connected, there is an edge connecting a vertex $u$ in $U$ with a vertex $v$ in $S \cup \{s\}$. 
    We claim that this edge must be directed from $u$ to $v$. Indeed, this is true if $v = s$, because $s$ is a sink. Suppose instead that $v \in S$. By definition of $S$, all vertices in $S$ have out-degree at least 1 in $T[S\cup \{s\}]$, and since the maximum out-degree of $T$ is 1, the edges $uv$ must again be directed towards $v$, as claimed. This means that $u \in S$, a contradiction. We conclude that $V(T)=S\cup \{s\}$, proving \ref{itm:tree-directed}.
    
    Now, let $v\in S$ and let $u$ be the only out-neighbour of $v$ in $T$. 
    Then $vu$ was directed in one of the steps \ref{itm:step-12} to \ref{itm:step-15} of \Cref{alg:preColouring}.
    If $vu$ was directed in \ref{itm:step-12} or \ref{itm:step-13} then $v$ has in-degree $0$, and otherwise it has in-degree $2$, proving \ref{itm:tree-indegree}.
    In particular, $v$ is a non-root leaf if and only if it was directed in \ref{itm:step-12} or \ref{itm:step-13}, in which case $v$ is in the interior of $P$, as required for \ref{itm:tree-leaves}.
    
    In fact, all vertices in the interior of $P$ are non-root leaves, because if $v$ is in the interior of $P$, it is incident to an edge $vz$ which is not in $P$. Because $P$ is a geodesic, $z \notin V(P)$. It follows that $vz$ is directed in \ref{itm:step-12} or \ref{itm:step-13}. 
    
    Suppose that the sink $s$ is in $P$. Then by the previous paragraph, $s$ is an end of $P$, proving the first part of \ref{itm:tree-root}.
    Let $u$ be the neighbour of $s$ in $P$.
    We now show that $s$ is incident to both red and blue edges. Indeed, if $s$ has two in-neighbours $w_1$ and $w_2$ then either at least one of them is red, so \ref{itm:step-16} does nothing and $s$ ends up incident with both red and blue edges, or $s$ has two blue in-neighbours and then after \ref{itm:step-16} it is incident to both blue and red edges. Otherwise, $s$ has exactly one in-neighbour $w$. If $ws$ is red, then \ref{itm:step-16} does nothing and $s$ has both blue and red neighbours, and if $ws$ is blue then a quick case analysis shows that after \ref{itm:step-16}, $s$ has both red and blue edges.
\end{proof}

Given a rooted tree (directed or not), let $L(T)$ be the set of non-root leaves of $T$ and let $\ell(T) := |L(T)|$. We say that a rooted tree $T$ is \emph{good} if $\diam(T) \ge \ell(T) - 1$.

\begin{claim}\label{cl:DiameterOfTree}
    Every component $T$ of $D$ is good.
\end{claim}

\begin{proof}
    By \Cref{cl:StructureOfTrees}, we have that $L(T)\subseteq V(P)$. Let $x$ be the left-most non-root leaf of $T$ in $P$ and $y$ the right-most vertex. Then $\dist_P(x,y)\geq \ell(T)-1$. As $P$ is a geodesic, $\dist_P(x,y)\leq \dist_T(x,y)\leq \diam(T)$. 
\end{proof}

We say that a rooted (undirected) tree $T$ is \emph{$x$-binary} if its root has degree $x$ and its other vertices have degree either $1$ or $3$ (i.e.\ non-leaf non-roots have degree $3$, and thus have two children). We say that a rooted tree is \emph{almost binary} if it is $x$-binary for some $x \in [3]$, and that it is \emph{binary} if it is $2$-binary. 
A big part of the remainder of this proof will analyse the structure of good almost binary trees, ignoring directions.



\begin{claim}\label{cl:goodSubTree}
    Let $T$ be an almost binary tree, and let $T'$ be an almost binary subtree of $T$.
    If $T$ is good then so is $T'$.
\end{claim}

\begin{proof}
    We prove the claim by induction on $|T| - |T'|$. If $|T| - |T'| = 0$, then $T = T'$, and the claim holds trivially.
    So suppose $|T| > |T'|$. We may also assume $|T'| \ge 2$, because a singleton is a good tree.
    Denote the roots of $T$ and $T'$ by $r$ and $r'$, respectively.
    
    We claim that one of the following holds.
    \begin{enumerate}[label = \rm(\alph*)]
        \item \label{itm:Tprime-a}
            $r' \neq r$ and $\deg(r') = 2$,
        \item \label{itm:Tprime-b}
            $r' \neq r$ and $\deg(r') = 1$,
        \item \label{itm:Tprime-c}
            $r' = r$ and $\deg(r') < \deg(r)$,
        \item \label{itm:Tprime-d}
            there is a non-root leaf $v$ in $T'$ which is not a leaf in $T$.
    \end{enumerate}
    Indeed, if $r' \neq r$ then $\deg(r') \in \{1, 2\}$ (because every neighbour of $r'$ in $T'$ is a child of $r'$ in $T$ and so $\deg(r') \le 2$, and moreover $\deg(r') \ge 1$ because $|T'| \ge 2$). If $r' = r$ and $\deg(r') = \deg(r)$, then there is a non-root vertex $v$ in $T'$ which has a smaller degree in $T'$ than in $T$. But then, by the trees being almost-binary, $v$ is a leaf in $T'$ but not in $T$.
    
    In all cases, we will modify $T'$ slightly to obtain an almost binary tree $T''$ which is a rooted subtree of $T$, apply induction to deduce that $T''$ is good, and then conclude that the same holds for $T'$.
    
    Let us first define $T''$ in each of the four cases: if \ref{itm:Tprime-a} holds, form $T''$ by adding a new vertex $r''$ to be the root of $T''$, and attaching $r''$ to $r'$; if \ref{itm:Tprime-b} or \ref{itm:Tprime-c} holds, form $T''$ by adding a new child to $r'$; and if \ref{itm:Tprime-d} holds, add two new children to $v$.
    
    Notice that in all cases $T''$ is an almost binary subtree of $T$, which is larger than $T'$. Thus, by induction, $T''$ is good, and so $\diam(T'') \ge \ell(T'') - 1$.
    In case \ref{itm:Tprime-a}, $\diam(T') = \diam(T'')$ and $\ell(T') = \ell(T'')$.
    In all other cases, $\diam(T') \ge \diam(T'') - 1$ and $\ell(T') = \ell(T'') - 1$. 
    Either way, it follows that $\diam(T') \ge \ell(T') - 1$, so $T'$ is good, as required.
\end{proof}

\begin{claim}\label{cl:CloseLeaves}
    Let $T$ be an almost binary good tree. Then every vertex in $T$ has a non-root leaf at distance at most 3 from it.
\end{claim}

\begin{proof}
    Assume that there is a  non-root non-leaf vertex $u$ with no non-root leaf vertex at distance at most 3 from it. Consider the subtree $T_u$ of $T$ rooted at $u$. It is easy to check that $T_u$ is a binary tree, and by \Cref{cl:goodSubTree} it is good. 
    By assumption on $u$, the tree $T_u$ has depth at least 3. Let $T'\subseteq T_u$ be the subtree of $T_u$ rooted in $u$ with depth 3; by \Cref{cl:goodSubTree}, $T'$ is good. But $T'$ has eight non-root leaves and diameter 6, a contradiction. 
    So every non-root non-leaf has a non-root leaf at distance at most 2. Thus the root has a non-root leaf at distance at most 3.
\end{proof}

Recall that by \Cref{cl:StructureOfTrees}~\ref{itm:tree-directed} and \ref{itm:tree-indegree} and by \Cref{cl:DiameterOfTree}, every component of $D$ is a good almost binary tree. Thus, by \Cref{cl:CloseLeaves}, every edge coloured by \Cref{alg:CreateGadgetsInPreCol} is at distance at most $3$ from $P$, proving \ref{itm:mono-vx-dist}.
It remains to prove \ref{itm:mono-vx-F}.

Observe that in every copy of $\Fexc$ which is parallel to $P$, with vertices labelled as in \Cref{def:OnlyMonoVertex}, \Cref{alg:preColouring} colours it according to $\Fexcs$. In particular, $y_2$ has red degree 3. Thus, to prove \ref{itm:mono-vx-F}, we need to show that every vertex of monochromatic degree $3$ is the vertex $y_2$ in some copy of $\Fexc$ which is parallel to $P$.

\begin{definition}\label{def:caterpillar}
    A caterpillar $T$ is a tree with $V(T)=\{y_0,\dots,y_s,z_1,\dots,z_{s-1}\}$ and $E(T)=\{y_{i-1}y_i : i\in[s]\}\cup \{y_iz_i : i\in [s-1]\}$.
    In other words, $T$ consists of a path $(y_0 \ldots y_s)$, along with a pendant edge at each vertex in the interior of the path.
    \begin{figure}[h!]
        \centering
        \includegraphics{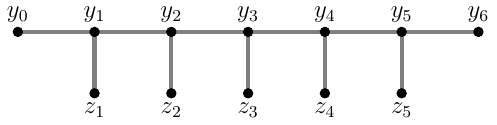}
        \caption{A caterpillar}
        \label{fig:caterpillar}
    \end{figure}
\end{definition}

\begin{claim}\label{cl:cubicTreesCaterpiller}
    All good $3$-binary trees are caterpillars.
\end{claim}

\begin{proof}
    Suppose not, and take $T$ to be a minimal good $3$-binary tree that it is not a caterpillar. Let $P=(x_0\dots x_t)$ be a longest path in $T$; so the ends $x_0$ and $x_t$ are leaves in $T$. If $t\leq 2$, then, using that $T$ is $3$-binary, $T$ is a star and thus a caterpillar, a contradiction.
    
    So $t\geq 3$. Since $x_0$ and $x_t$ are leaves, and $T$ is $3$-binary, $x_1$ and $x_{t-1}$ have degree 3. Let $x_0'$ and $x_{t}'$ be the neighbours of $x_1$ and $x_{t-1}$ 
    that are not in $P$, respectively. By possibly relabelling the vertices of $P$, we may assume that $x_1$ is not the root of $T$. 
    Let $T'=T\setminus \{x_0,x_0'\}$. Then $T'$ is a $3$-binary tree with the same root as $T$, and $T'\subseteq T$. By \Cref{cl:goodSubTree}, $T'$ is good, and by minimality of $T$ we get that $T'$ is a caterpillar. 
    
    Write $V(T')=\{y_0,\dots,y_s,z_1,\dots,z_{s-1}\}$ and $E(T')=\{y_{i-1}y_i : i\in[s]\}\cup \{y_iz_i : i\in [s-1]\}$. We now want to understand how to obtain $T$ from $T'$. To do so, we need to determine which of the vertices of $T'$ is $x_1$. 
    We consider three cases: $x_1 \in \{y_1, \ldots, y_{s-1}\}$; $x_1 \in \{y_0, y_s, z_1, z_{s-1}\}$; and $x_1 \in \{z_2, \ldots, z_{s-2}\}$.
    In the first case, $T$ has a vertex of degree $5$, contradicting $T$ being $3$-binary. In the second case, $T$ is a caterpillar, contradicting the choice of $T$. And in the third case, $T$ has diameter $s$ and $s+2$ non-root leaves, contradicting $T$ being good.
    We have thus reached a contradiction and thereby proved the claim.
\end{proof}

In the following claim we identify the only possible monochromatic vertices. Note that up to this point we did not care if the edges are coloured red or blue, only that they were coloured during \Cref{alg:preColouring}. In the sequel we look at the precise colouring. 

For the remainder of the proof, let $u$ be a vertex in $D$ with monochromatic degree $3$, and let $T$ be the component of $D$ that contains $u$. We will show that $T$ is a copy of $\Fexcs$ which is parallel to $P$, and that $u$ is the vertex $y_2$ in $\Fexcs$.

\begin{claim}\label{cl:MonoVxAreRoots}
    $T$ is a caterpillar, whose leaves are in $P$ and whose interior does not intersect $P$, and $u$ is a non-leaf root in $T$.
\end{claim}

\begin{proof}
    By \Cref{cl:StructureOfTrees}, $u$ is not in $P$ and the leaves in $T$ are in the interior of $P$. Thus all edges incident to $u$ are directed. Moreover, these edges are directed into $u$, because if there is an edge directed from $u$, it arises in \ref{itm:step-14} or \ref{itm:step-15}, which ensure that $u$ is not monochromatic.
    Recall that, by \Cref{cl:StructureOfTrees}, $T$ is an almost binary tree, directed towards its root. Since $u$ is a sink with in-degree $3$, this means that $T$ is a $3$-binary tree rooted at $u$. By \Cref{cl:DiameterOfTree}, $T$ is good, and thus by \Cref{cl:cubicTreesCaterpiller}, $T$ is a caterpillar. To complete the proof, notice that non-root non-leaves of $T$ have in-degree $1$ and out-degree $2$ and thus are not in $P$ (the steps involving $P$, namely \ref{itm:step-11}, \ref{itm:step-12} and \ref{itm:step-16} do not create such a vertex).
\end{proof}

By the previous claim, $T$ is a caterpillar with a non-leaf root $u$. We may thus write $V(T) = \{y_0, \ldots, y_s, z_1, \ldots, z_{s-1}\}$, for some $s \ge 2$. The claim implies that $u \in \{y_1, \ldots, y_{s-1}\}$, because all other vertices are leaves of $T$. Notice that the vertices of $T$ which are in $P$ form an interval in $P$ (though possibly the order of their appearance in $P$ does not correspond to their order in $T$), because the distance between any two of them in $T$ is at most $s$, and if they do not form an interval then the distance between some two of them in $P$ is at least $s+1$, a contradiction to $P$ being a geodesic. In a similar vain, we may assume that the left-most vertex is $y_0$ and the right-most vertex is $y_s$.

\begin{claim}
    The vertex $u$ has exactly one neighbour in $P$. 
\end{claim}

\begin{proof}
    Notice that by \Cref{cl:MonoVxAreRoots}, $u$ has at least one neighbour in $P$. Thus we need to rule out the possibility that $u$ has either two or three neighbours in $P$.
    Suppose first that $u$ has three neighbours in $P$. This means $s = 2$ and $u = y_1$. Since $y_0, z_1, y_2$ form an interval, it follows from \ref{itm:step-12} and \ref{itm:step-13} that $u$ has two blue in-edges and one red in-edge, contradicting its being monochromatic.
    
    Now suppose $u$ has two neighbours in $P$ and $s \ge 3$. Without loss of generality, this means $u = y_1$. Notice that its neighbours $y_0, z_1$ in $P$ are either consecutive or at distance $2$ in $P$, as $P$ is a geodesic. In the former case, $u$ gets two blue in-edges in \ref{itm:step-12} and its third edge is directed out of $u$ and coloured red in \ref{itm:step-14}, a contradiction. 
    
    So suppose that $y_0$ and $z_1$ are at distance $2$ in $P$. 
    We claim that $z_2$ must appear between $y_0$ and $z_1$. Indeed, recalling that $y_0$ is the left-most vertex in $P$ and $y_s$ is the right-most vertex, otherwise the vertex $w$ between $y_0$ and $z_1$ is at distance $s-1$ from $y_s$ in $T$ but at distance $s-2$ from $y_s$ in $T$, a contradiction to $P$ being a geodesic. It follows that $u$ receives two red in-edges in \ref{itm:step-12} and $y_1$ receives at least one red in-edge in the same step. But this means that the edge $y_1 u$ is directed in \ref{itm:step-15} and is blue, a contradiction.
\end{proof}

By the previous claim, $u = y_i$ for some $i \in \{2, \ldots, s-2\}$. In particular, $s \ge 4$. Notice that if $s = 4$ then $u = y_2$, as required for \ref{itm:mono-vx-F}. So suppose that $s \ge 5$. Without loss of generality, $u \neq y_{s-2}$. Notice that both $u = y_i$ and $y_{i+1}$ have exactly one neighbour in $P$, and thus they receive a red in-edge in \ref{itm:step-13}. But then if $y_{i+1}y_i$ is ever directed and coloured, this happens in \ref{itm:step-15} and the colour is blue, preventing $u$ from being monochromatic.
\end{proof}

\subsection{Stage II: Fixing monochromatic vertices (\Cref{alg:FixPreCol})}\label{sec:StageII}

The purpose of this section is to fix the colouring given by \Cref{alg:preColouring} to have no monochromatic vertices. Recall \Cref{def:OnlyMonoVertex} and \Cref{lem:OnlyMonoVertex}.

\begin{definition}
    Let $\Fexcss = $ \raisebox{-11pt}{\includegraphics[]{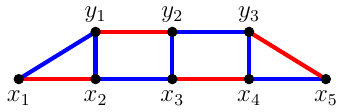}}. Namely, $\Fexcss$ is a red-blue colouring of a $\Fexc$, with vertices and edges labelled as in \Cref{def:OnlyMonoVertex}, where the red edges are $\{x_1 x_2, y_1 y_2, y_2 x_3, x_3 x_4, x_4 y_3, y_3 y_5\}$ and the blue edges are $\{x_1 y_1, y_1 x_2, x_2 x_3, y_2 y_3, x_4 x_5\}$.
\end{definition}

\begin{algorithm}[Fixing monochromatic vertices]\label{alg:FixPreCol}
Consider the partial colouring $\phi_1$ (the outcome of \Cref{alg:preColouring}). Obtain a new partial colouring $\phi_2$ from $\phi_1$, as follows. 
\begin{itemize}
    \item 
        For each copy $F$ of $\Fexcs$ (with vertices labelled as in \Cref{def:OnlyMonoVertex}), do:
        \begin{itemize}
            \item 
                Recolour $F$ according to $\Fexcss$. (Notice that, unlike $\Fexcs$, the coloured graph $\Fexcss$ is not symmetric, so if we relabelled $x_i$ to $x_{6-i}$ and $y_j$ to $y_{4-j}$ in $\Fexcs$, we would get a different colouring. We choose one of these colourings arbitrarily.)
            \item   
                Update $P' \leftarrow (P' \setminus \{x_1x_2,\, x_3x_4\} \cup \{x_1 y_1 x_2,\, x_3y_2y_3x_4\})$.
        \end{itemize}
    \item
        Delete all directions.
\end{itemize}
\end{algorithm}

\begin{lemma}\label{lem:PreColouringPropAfterFix}
 The partial colouring $\phi_2$ has the following properties.
\begin{enumerate}[label = \rm(\arabic*)]
    \item \label{itm:phi-2-deg}
        if a vertex is incident to at least two coloured edges, then at least one of them is red and at least one is blue,
    \item \label{itm:phi-2-path}
        every maximal monochromatic path, except for possibly $P'$, has an end touching two edges of the opposite colour, 
    \item \label{itm:phi-2-cycle}
        there are no monochromatic cycles,
    \item \label{itm:phi-2-dist}
        if $u$ is contained in some coloured edge, then its distance from $P$ is at most 3.
\end{enumerate}
\end{lemma}

\begin{proof}
    For \ref{itm:phi-2-deg}, notice that every vertex which is in a copy of $\Fexcs$ after \Cref{alg:preColouring} is incident to at least one edge of each colour after \Cref{alg:FixPreCol}, and by \Cref{lem:OnlyMonoVertex}, every other vertex which is incident to at least two coloured edges, is incident to at least one of each colour.
    
    Suppose that $Q$ is a maximal monochromatic path which is not $P'$. If it does not intersect the vertices of any copy of $\Fexc$ parallel to $P$ then $Q$ has an end with two neighbours of the opposite colour, by \Cref{lem:OnlyMonoVertex}~\ref{itm:mono-vx-path}.
    Now, if $Q$ is red and intersects the vertices of a copy $F$ of $\Fexc$ parallel to $P$, then it contains a red edge of $F$, which has a vertex incident to two blue edges. If $Q$ is blue and intersects such a copy $F$ then $Q = P'$. Item \ref{itm:phi-2-path} follows.
    
    Item \ref{itm:phi-2-path} shows that there is no monochromatic cycle with edges outside of $P'$. Since $P'$ is a path, there are no monochromatic cycles at all, proving \ref{itm:phi-2-cycle}.
    Item \ref{itm:phi-2-dist} follows directly from \Cref{lem:OnlyMonoVertex}.
\end{proof}

\subsection{Stage III: Creating gadgets (\Cref{alg:CreateGadgetsInPreCol})}\label{sec:StageIII}
In this section we will show that a small modification of the partial colouring $\phi_2$, defined in \Cref{sec:StageII}, contains a blue $\ell$-gadget.

Let $P'$ and $\phi_2$ be the path and partial colouring obtained by running \Cref{alg:preColouring} and then \Cref{alg:FixPreCol}. Write $P' = (p_0 \ldots p_L)$, and note that $L \ge 5\ell+200$.
\begin{claim} \label{claim:gadget-cases}
    One of the following cases holds (or the lemma holds).
    \begin{enumerate}[label = \rm\textbf{\Roman*.}, ref = \rm\Roman*]
        \item \label{caseI-gadgets}
            There is a copy of $\Fexcs$ in $\phi_1$ which is parallel to $P$ and whose distance on $P'$ from both end points of $P'$ is at least $\ell+20$.
        \item \label{caseII-gadgets}
            Case \ref{caseI-gadgets} does not hold, and there is a vertex $p_s \in V(P') \setminus V(P)$, with $s \in [2\ell+40, L - 2\ell - 40]$.  
        \item \label{caseIII-gadgets}
            Cases \ref{caseI-gadgets} and \ref{caseII-gadgets} do not hold, and there exists $s \in [2\ell+100, L - 2\ell-100]$ with such that $p_{s-1}$ and $p_s$ have a common neighbour $w$ outside of $P'$. 
    \end{enumerate}
\end{claim}

\begin{proof}
    Suppose that \ref{caseI-gadgets} and \ref{caseII-gadgets} do not hold. Then, writing $P = (q_0 \ldots q_{L'})$, we claim that the vertices $q_s$ and $q_{s+1}$, with $s \in [2\ell + 60, L' - 2\ell - 60]$ do not have a common neighbour. Indeed, suppose that $q_s$ and $q_{s+1}$ have a common neighbour $w$, and write $w = p_{s'}$. Notice that $s' \in [2\ell + 60, L - 2\ell - 60]$.
    Assuming that \ref{caseI-gadgets} does not hold, we find that $w$ is not in a copy of $\Fexcs$ in $\phi_1$, and thus its edges are not recoloured in $\phi_2$, showing that \ref{caseII-gadgets} holds for $s'$.
    
    It follows that step \ref{itm:step-12} does not occur with $uv_{\ref{itm:step-12}} = q_sq_{s+1}$ and $s \in [2\ell+60, L' - 2\ell - 60]$, and so $P'' := (p_{2\ell + 80} \ldots p_{L - 2\ell - 80})$ is a subpath of $P$. 
    This implies that $P''$ is a geodesic (of length at least $\ell + 36$). If no two vertices in $P''$ have a common neighbour outside of $P''$, then we can apply \Cref{lem:GoodCaterpillarGadgetsExist} to complete the proof of this lemma. So, we assume that there are two vertices in $P''$ that have a common neighbour $w$ outside of $P''$. These vertices are not consecutive in $P''$, by the previous paragraph, so they are at distance $2$ in $P''$. In other words, there exists $s \in [2\ell+81, L-2\ell-81]$ such that $p_{s-1}$ and $p_{s+1}$ have a common neighbour $w$ outside of $P''$. We claim that $w$ is also outside of $P'$. Indeed, suppose not. Then $w$ is at distance at most $1$ from some vertex in $P \setminus \{p_{s-1}, p_s, p_{s+1}\}$; denote this vertex by $p_t$. So $p_{s-1}$ and $p_{s+1}$ are at distance at most $2$ from $p_t$, a contradiction to $P$ being a geodesic. 
\end{proof}

\begin{algorithm}[Creating a gadget]\label{alg:CreateGadgetsInPreCol}
Throughout this description, if $v$ is an interior vertex of $P'$, then we denote by $v'$ the (unique) neighbour of $v$ such that $vv'$ is not an edge of $P'$ (so $vv'$ is a red edge in $\phi_2$). 

We explain how to proceed in each of the three cases given in \Cref{claim:gadget-cases}.
\begin{enumerate}[label = \rm\textbf{\Roman*.}, ref = \rm\Roman*]
    \item 
        There is a copy of $\Fexcs$ in $\phi_1$ which is parallel to $P$ and whose distance on $P'$ from both end points of $P'$ is at least $\ell+20$.
        
        Label the vertices and edges of such a copy of $\Fexcs$ as in \Cref{def:OnlyMonoVertex}. Thus $(x_1 \ldots x_5)$ is a subpath of $P$, and $(x_1 y_1 x_2 x_3 y_2 y_3 x_4 x_5)$ is a subpath of $P'$. Let $s$ be such that $x_4 = p_s$ and suppose that $p_{s+1} = x_5$; notice that $s + \ell + 2 \le L-1$.
        Recolour the edges $p_s p_{s-1}$, $p_{s-2}p_{s-3}$ and $p_{s+\ell} p_{s+\ell+1}$ red. If $p_{s+\ell}' p_{s+\ell+1}'$ is an edge, colour it blue (see \Cref{fig:gadget-case1}).
        
        In this way we get an $\ell$-gadget of Type II as follows: $Q_0=(p_{s-2}\dots p_{s+\ell})$, $Q_1 = (p_{s-1} p_{s+1})$, $e_1 = p_{s-5} p_{s-2}$, $e_2 = p_{s-3} p_{s-2}$, $e_3 = p_{s-3} p_s$, $e_4 = p_{s+\ell} p_{s+\ell+1}$, and $e_5 = p_{s+\ell} p_{s+\ell}'$.
        \begin{figure}[h!]
            \centering
            \includegraphics[scale = .8]{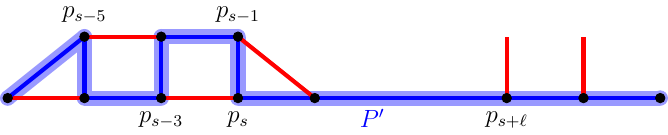}
            \vspace{.3cm}
            \includegraphics[scale = .8]{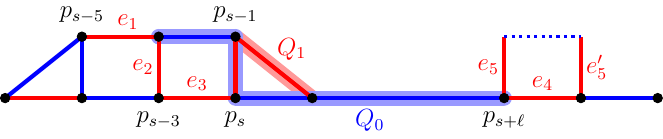}
            \caption{Case I: before and after \Cref{alg:CreateGadgetsInPreCol}}
            \label{fig:gadget-case1}
        \end{figure}
       
    \item \label{caseII-gadgets}
        Case \ref{caseI-gadgets} does not hold, and there is a vertex $p_s \in V(P') \setminus V(P)$, with $s \in [2\ell+40, L - 2\ell - 40]$.  
        
        Recolour the edges $p_{s-3} p_{s-2}$, $p_{s-1} p_s$, and $p_{s+\ell} p_{s + \ell + 1}$ red. If $p_{s-2}' p_{s-3}'$ or $p_{s+\ell}' p_{s+\ell+1}'$ are edges, colour them blue (see \Cref{fig:gadget-case2}).
        
        \begin{figure}[h!]
            \centering
            \includegraphics[scale = .8]{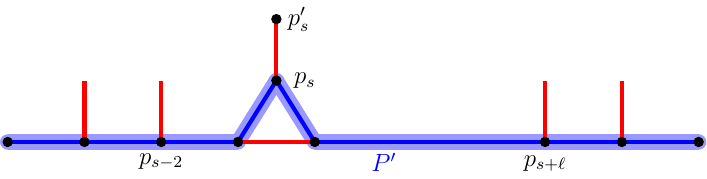} 
            \vspace{.3cm}
            \includegraphics[scale = .8]{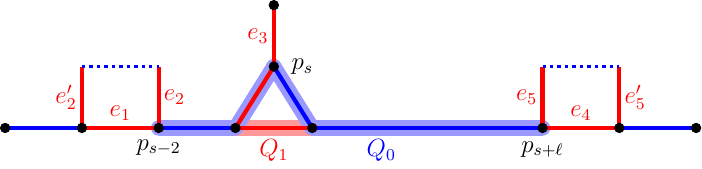} 
            \caption{Case II: before and after \Cref{alg:CreateGadgetsInPreCol}}
            \label{fig:gadget-case2}
        \end{figure}
        
        This gives a gadget of Type II as follows: $Q_0=(p_{s-2}\dots p_{s+\ell})$, $Q_1=(p_{s-1}p_{s+1})$, $e_1=p_{s-3}p_{s-2}$, $e_2=p_{s-2}p_{s-2}'$, $e_3=p_sp_s'$, $e_4=p_{s+\ell}p_{s+\ell+1}$,  $e_5=p_{s+\ell}p_{s+\ell}'$. Note that the edges $e_3$ and $e_2$ may intersect or be equal.

    \item \label{caseIII-gadgets}
    
        Cases \ref{caseI-gadgets} and \ref{caseII-gadgets} do not hold, and there exists $s \in [2\ell+100, L - 2\ell-100]$ with such that $p_{s-1}$ and $p_s$ have a common neighbour $w$ outside of $P'$. 
        
        Recolour the edges $p_{s-3}p_{s-2}$, $p_{s-1}p_s$ and $p_{s+\ell} p_{s+\ell+1}$ red. If $p_{s-2}' p_{s-3}'$ or $p_{s+\ell}' p_{s+\ell+1}'$ are edges, colour them blue (see \Cref{fig:gadget-case3}).
        
        \begin{figure}[h!]
            \centering
            \includegraphics[scale = .8]{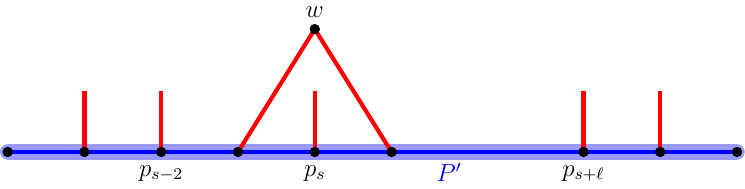} 
            \vspace{.3cm}
            \includegraphics[scale = .8]{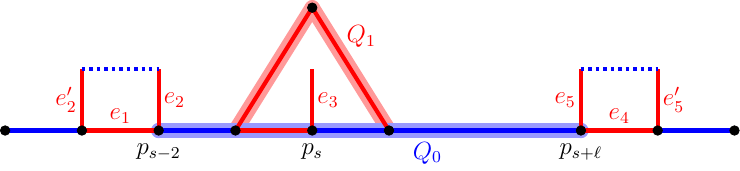} 
            \caption{Case III: before and after \Cref{alg:CreateGadgetsInPreCol}}
            \label{fig:gadget-case3}
        \end{figure}

        This gives a gadget of Type II as follows: $Q_0=(p_{s-2}\dots p_{s+\ell})$, $Q_1=p_{s-1}wp_{s+1}$, $e_1=p_{s-3}p_{s-2}$, $e_2=p_{s-2}p_{s-2}'$, $e_3=p_sp_s'$, $e_4=p_{s+j}p_{s+j+1}$, and $e_5=p_{s+j+1}p_{s+j+1}'$.
        
\end{enumerate}
Denote the partial colouring at the end of this process by $\phi_3$ and write $e_5' := p_{s+\ell+1} p_{s+\ell+1}'$ and $e_2' := p_{s-3}p_{s-3}'$.
\end{algorithm}

\begin{lemma}\label{lem:PartialColProp}
    The partial colouring $\phi_3$ satisfies the following properties, and thus satisfies the requirements of \Cref{lem:GoodGadgetsExist}.
\begin{enumerate}[label = \rm(\arabic*)]
    \item \label{itm:phi-3-dist}
        all coloured edges are at distance at most 4 from the geodesic $P$ and form a connected subgraph,
    \item \label{itm:phi-3-deg}
        if a vertex $u$ is incident to at least two coloured edges, then at least one of them is red and at least one is blue,
    \item \label{itm:phi-3-cycle}
        if $C$ is a cycle in $G$ containing no red edges or no blue edges, then $C$ has two consecutive edges that are not coloured by $\phi_3$.
\end{enumerate}
\end{lemma}

\begin{proof}
    By \Cref{lem:PreColouringPropAfterFix}, before applying \Cref{alg:CreateGadgetsInPreCol} all coloured edges are at distance at most $3$ from $P$. Note that in \Cref{alg:CreateGadgetsInPreCol} all newly coloured edges (if any) intersect an already coloured edge, and thus are at distance at most $4$ from $P$; therefore, \ref{itm:phi-3-dist} holds. 
    
    By \Cref{lem:PreColouringPropAfterFix}, before applying \Cref{alg:CreateGadgetsInPreCol}, Item \ref{itm:phi-3-deg} above holds, and thus it still holds after applying \Cref{alg:CreateGadgetsInPreCol}, from the way the algorithm is defined. It remains to prove \ref{itm:phi-3-cycle}. 
    
    Consider the blue edges $f_1 := p_{s-3}' p_{s-2}'$ and $f_2' = p_{s+\ell}' p_{s+\ell+1}'$ (if exist). We claim that $f_1$ and $f_2$ were not red in $\phi_2$. Indeed, notice that every red edge in both $\phi_1$ and $\phi_2$ has a vertex with two blue neighbours. It follows that $\phi_2$ has no red path of length $3$, and thus $f_1$ and $f_2$ cannot be red in $\phi_2$. Thus, either $f_1$ is in a blue path in $\phi_3$ which has an end with two red neighbours, or both edges of $f_1$ have an uncoloured neighbour.
    
    We now show that there are no monochromatic cycles. Suppose $C$ is a monochromatic cycle. Then by \Cref{lem:PreColouringPropAfterFix}~\ref{itm:phi-2-cycle}, $C$ intersects an edge that was coloured (or recoloured) in \Cref{alg:CreateGadgetsInPreCol}. 
    By the previous paragraph, $C$ does not contain $f_1$ or $f_2$.
    Notice that all maximal monochromatic paths in $\phi_3$, that intersect an edge coloured in this algorithm, have an end with two edges of the opposite colour, except for possibly the maximal red paths containing $e_2 e_1 e_2'$ or $e_5 e_4 e_5'$ and for $f_1$ or $f_2$ if they do not intersect other blue paths. So we need to rule out the case that $C$ contains one of $e_5$ and $e_2$.
    
    We explain why $C$ does not contain $e_5$. Indeed, suppose otherwise, and recall that $e_5 = p_{s+\ell} p_{s+\ell}'$. Then $p_{s+\ell} \notin P'$ (if $p_{s+\ell}'$ is in $P'$ then it has two blue neighbours). In particular, $e_5$ is not in a copy of $\Fexcss$. 
    Thus $e_5$ was coloured in \ref{itm:step-13} and was not recoloured afterwards.
    Since $e_5$ is at distance at least $10$, say, from the ends of $P'$, the edge $e_5$ does not touch any edges coloured in \ref{itm:step-16} of \Cref{alg:preColouring}. It follows that, in $\phi_1$, the vertex $p_{s+\ell}'$ either has no red neighbours other than $p_{s+\ell}$, or it has both a red and a blue one. In the former case, in $\phi_3$ the vertex $p_{s+\ell}'$ still has red degree $1$. In the latter case, denoting by $u$ the red neighbour of $p_{s+\ell}'$ which is not $p_{s+\ell}$, we find that $u$ has two blue neighbours, and this remains the case in $\phi_3$ (observe that $u \neq p_{s+\ell+1}$). Either way, $e_5$ is not contained in a red cycle, and similarly for $e_2$. 
    
    Suppose that $C$ is a cycle containing edges of at most one colour. Because there are no monochromatic cycles, $C$ contains at least one uncoloured edge, and we may assume that it contains at least one coloured edge (otherwise \ref{itm:phi-3-cycle} holds trivially). Let $R_1, \ldots, R_k$ be the maximal monochromatic subpaths of $C$. We claim that $R_i$ is a maximal monochromatic path also in $\phi_3$, for $i \in [k]$. Indeed, write $R_i = (r_0 \ldots r_{\rho})$ and suppose that $r_0$ is incident to an edge $e \neq r_0 r_1$ coloured with the colour of $R_i$. Then the third edge at $r_0$ has the opposite colour, by \ref{itm:phi-3-deg}, so by assumption on $C$ it cannot contain this third edge and thus it contains $e$, contradicting the maximality of $R_i$.
    Recalling our analysis of maximal monochromatic paths with an end incident with two edges of the same colour, the only candidates for $R_1, \ldots, R_k$ are $f_1, f_2$ (if their vertices are not incident with other blue edges), and the maximal red paths containing $e_5e_4e_5'$ or $e_2 e_1 e_2'$. In other words, either $C$ contains one or both of $f_1$ and $f_2$ and no other coloured edges, or (and this follows similarly to the argument about there being no red cycles) $C$ contains one of $e_5 e_4 e_5'$ and $e_2 e_1 e_2'$ and no other coloured edges.
    
    \begin{claim} \label{claim:uv}
        In Cases \ref{caseII-gadgets} and \ref{caseIII-gadgets}, the vertices $u := p_{s-3}$ and $v := p_{s+\ell+1}$ are at distance at least $5$ from each other. 
    \end{claim}
    
    \begin{proof}
        Notice that in both cases $u$ and $v$ are at distance at least $\ell + 4 \ge 7$ in $P'$.  
        In Case \ref{caseIII-gadgets} the vertices $u$ and $v$ are at distance at least $7$ also in $P$ and thus in $G$. In Case \ref{caseII-gadgets} both $e_1$ and $e_4$ are edges in $P$. Indeed, if say $e_1$ is not in $P$, then, using that Case \ref{caseI-gadgets} does not hold and recalling \ref{itm:step-12} from \Cref{alg:preColouring}, one of the edges $e_2$ and $e_2'$ has a vertex touching two blue edges, a contradiction. It follows that $u$ and $v$ are both in $P$ and are at distance at most $5$ in $P$ and thus in $G$. Either way, $u$ and $v$ are at distance at least $5$ from each other in $G$, so $|C| \ge 10$ and $C$ has two consecutive uncoloured edges. 
    \end{proof}    
    
    We obtain a lower bound on $|C|$ in each of the possible cases, using \Cref{claim:uv} above.
    \begin{itemize}
        \item 
            $C$ contains exactly one of $f_1, f_2$. Then $|C| \ge 3$.
        \item
            $C$ contains both $f_1$ and $f_2$. Then $|C| \ge 6$ as $p_{s-3}'$ and $p_{s+\ell+1}'$ are a distance at least $3$ apart.
        \item
            $C$ contains exactly one of $e_2e_1e_2'$ and $e_5e_4e_5'$. Then $|C| \ge 5$, because $f_1$ and $f_2$, if exist, are blue.
        \item
            $C$ contains both $e_2e_1e_2'$ and $e_5e_4e_5'$. Then $|C| \ge 10$, because $C$ contains $u$ and $v$ which are a distance at least $5$ apart.
    \end{itemize}
    In all cases, it easily follows that $C$ has two consecutive uncoloured edges, proving \ref{itm:phi-3-cycle}.
\end{proof}

\section{Concluding Remarks} \label{sec:conc}
In this paper we essentially prove Wormald's conjecture, showing that large connected cubic graphs, whose number of vertices is divisible by $4$, can be decomposed into two isomorphic linear forests (\Cref{thm:main}). 

In fact, our proof can be tweaked to prove the following stronger statement.

\begin{theorem} \label{thm:main-strong}
    There exists $k_0$ such that the following holds. Suppose that $G$ is a cubic graph on $n$ vertices, where $n$ is divisible by $4$, and $G$ has a component on at least $k_0$ vertices. Then there is a red-blue edge-colouring of $G$ whose colour classes span isomorphic linear forests.
\end{theorem}

\begin{proof}[Proof sketch]
    We now sketch how this strengthening can be proved. 
    First, we note that our only use of the connectivity assumption on $G$ was in \Cref{cl:geodesics}, where we find $\Omega(n^{1-\eps})$ geodesics of length $m := 10^{10}\sqrt{\log n}$ which are at distance at least $50$ from each other. For this, it actually suffices to have at least say $n/2$ vertices in $G$ be in components of order at least $3\cdot 2^m$. Indeed, running the same proof on the union of components of at least this order yields the same result.
    
    For the proof of \Cref{thm:main-strong}, let $G$ be a cubic graph as in the theorem.
    For each $k$ let $\cC_k$ be the collection of components of $G$ on exactly $k$ vertices. Apply \Cref{thm:thomassen} (Thomassen's theorem) to each component in $\cC_k$ to get a red-blue colouring where every monochromatic component is a path of length at most $5$. To each such colouring, associate a $10$-tuple $(r_1, \ldots, r_5, b_1, \ldots, b_5)$, where $r_i$ and $b_i$ denote the numbers of red and blue components which are paths of length $i$. Notice that the number of possible $10$-tuples is at most the number of non-negative integer solutions to the equation $x_1 + \ldots + x_5 + y_1 + \ldots + y_5 = 3k/2$, where $x_i$ and $y_i$ count the number of edges in red and blue components which are paths of length $i$.
    This number is $\binom{3k/2 + 9}{9} = O(k^9)$. This implies that we can pair up the components in $\cC_k$, leaving at most $O(k^9)$ unpaired components, such that the $10$-tuples corresponding to each pair $(C_1, C_2)$ are identical. Now reverse the colouring in $C_2$ to get a red-blue colouring of $C_1 \cup C_2$ where the red and blue graphs are isomorphic linear forests (with paths of length at most $5$). It thus suffices to consider the graph $G'$ which is the union of unpaired graphs, and this is a graph where the number of components on $k$ vertices is $O(k^9)$, for every $k$. 
    
    We now show that $G'$ can be dealt with via the above paragraph. Indeed, denote by $n$ the number of vertices in $G'$ and by $K$ the number of vertices in the largest component of $G'$. Then $K \ge k_0$ (technically we need to make sure we leave at least one component on at least $k_0$ unpaired; this is clearly possible) and $n \le \sum_{i \le K} O(i^9) = O(K^{10})$. Writing $m := 10^{10}\sqrt{\log n}$, as above, and noting that $m \le n^{1/20}$, say (using that $n$ is large, which follows from $k_0$ being large), we get that there are at most $\sum_{i \le m} O(i^{9}) = O(m^{10}) = O(\sqrt{n})$ components on fewer than $m$ vertices. Thus the first paragraph above can be applied to find a red-blue colouring of $G'$ where the red and blue graphs are isomorphic linear forests, completing the proof.
\end{proof}

Notice that the assumption that the number of vertices is divisible by $4$ in Wormald's conjecture is necessary, as otherwise the number of edges is not divisible by $2$. Nevertheless, we can address the case when the number of vertices is not divisible by $2$, as follows.

\begin{theorem} \label{thm:main-not-divisible-by-4}
    Let $G$ be a connected cubic graph on $n$ vertices, where $n$ is large and $n \equiv 2 \!\pmod 4$. Then there is a red-blue edge-colouring of $G$ and a linear forest $F$, such that the blue graph is isomorphic to the disjoint union of $F$ with a path of length $3$, and the red graph is isomorphic to the disjoint union of $F$ with a matching of size $2$. 
    In particular, the two colour classes are linear forests, and the red graph is isomorphic to the blue one minus one edge.
\end{theorem}

\begin{proof}[Proof sketch]
    This can be proved by tweaking our proof of the main result, \Cref{thm:main} very slightly, as follows. In \Cref{claim:chi-3}, instead of requiring the number of red and blue edges to be same, make sure that the number of blue edges exceeds the number of red edges by exactly 1 (this can be done analogously). Next, follow the proof \Cref{claim:chi-4} to make sure that the number of red and blue components which are paths of length $t$, with $t \ge 4$, is the same, and that the number of blue components which are paths of length $3$ exceeds the number of such red components by exactly $1$. A counting argument like the one that appears after the proof of \Cref{claim:chi-4} shows that the number of red and blue components which are paths of length $2$ is the same. It thus follows that there are two more red than blue components which are edges, proving \Cref{thm:main-not-divisible-by-4}.
\end{proof}

As our paper is pretty long, we would like to point out that weaker versions of our main result can be proved with considerably less effort. Indeed, first, \Cref{sec:approx} on its own proves an approximate version of our main result (which applies for any cubic graph). 
Additionally, proving \Cref{WormaldConj} for (large connected) cubic graphs with high girth is also a mush easier task. Indeed, if the girth is at least $5$, then there are no geodesics with two vertices having a common neighbour outside the geodesic, making \Cref{sec:GeodesicWithCommon} irrelevant, and if the girth is at least $6$, then Cases \ref{case:alg-comb-b} and \ref{case:alg-comb-c} in \Cref{sec:GeodesicNoCommonNghbs} also become irrelevant; either way, imposing a high girth condition would shorten our paper by about 15 pages.

\section{Open Problems}

We now mention some directions for future research.
As one such direction, it would be interesting to see if our result could be combined with Thomassen's result (that every cubic graph can be $2$-edge-coloured so that monochromatic components are paths of length at most $5$). That is, is it true that every cubic graph can be $2$-edge-coloured such that the two colour classes are isomorphic linear forests with bounded component sizes?

\begin{problem}\label{prob:WormaldShortPaths}
    Is there a constant $c$ such that every cubic graph, whose number of vertices is divisible by $4$, can be 2-edge-coloured such that the two colour classes are isomorphic, and every monochromatic component is a path of length at most $c$?
\end{problem}

We note that the proof of Thomassen (as well as those of Jackson and Wormald \cite{jackson1996linear} and Aldred and Wormald \cite{aldred1998more} with larger constants) is very technical and requires a lot of case analysis. However, we can prove a version of Thomassen's result, where we bound components' length by a constant larger than $5$, in a much simpler way, following the argument in \Cref{claim:chi-2} (this is also similar to an argument from \cite{alon1988linear}). 
Therefore, it is possible that some of the ideas from this paper could be used in order to answer \Cref{prob:WormaldShortPaths}. 

Let us briefly sketch a simpler proof of a weakening of Thomassen's result  here. Given a cubic graph $G$, pick a red-blue colouring of $G$ whose colour classes are linear forests (see \cite{akiyama1981short} for a very short proof of the existence of such a colouring).
For an appropriate constant $c$, let $\cP_r$ be the set of red components that are paths of length at least $c$. Now decompose each path in $\cP_r$ into segments of length in $[c/2, c]$, letting $\cI_r$ be the family of such segments.
For each segment $I \in \cI_r$, let $\cand(I)$ be the set of edges in $I$'s interior that do not close a blue cycle if their colour is swapped; recall that $|\cand(I)| \ge (|I|-3)/2 \ge (c/2-2)/2 = c/4 - 1$.
Now form a graph $H_r$ by joining an edge $e \in \cand(I)$ with an edge $f \in \cand(J)$, for distinct $I, J \in \cI_r$, if they touch the same blue component. Then $H_r$ is a $|\cI_r|$-partite graph, with parts of size at least $c/4 - 1$, whose maximum degree is at most $4$. An application of the Lov\'asz local lemma (or the result in \cite{loh2007independent}) proves the existence of an independent $J_r$ transversal (namely, an independent set consisting of one vertex from each part) in $H_r$, provided that $c$ is large enough. Defining $\cP_b$ etc.\ analogously, and flipping the colours of edges in $J_r \cup J_b$, yields a red-blue colouring where the monochromatic components are paths of length at most $2c+1$.

Another interesting problem in this direction relates to the tightness of Thomassen's result. While it is known that the result is tight, namely that there is a cubic graph that cannot be $2$-edge-coloured with monochromatic components being paths of length at most $4$, the only known minimal examples are graphs on six vertices. It is thus plausible that for larger connected cubic graphs a bound of $4$ on the length of monochromatic components is enough.

\begin{problem}\label{prob:ThomassenWith4}
    Can every large enough connected cubic graph be $2$-edge-coloured so that all monochromatic components are paths of length at most $4$?
\end{problem}

One can also consider vertex versions of the above problems. This has been considered before. In 1990 Ando conjectured that the vertices of any cubic graph can be two-coloured such that the two colour classes induce isomorphic subgraphs. Ando's conjecture was first mentioned in the paper of Abreu, Goedgebeur, Labbate and Mazzuoccolo \cite{abreu2019colourings} where they made an even stronger conjecture, adding the requirement that the two colour classes induce linear forests. Recently, Das, Pokrovskiy and Sudakov \cite{das2021isomorphic} proved Ando's conjecture for large connected cubic graphs. In fact, their proof also verifies a stronger conjecture by Abreu, Goedgebeur, Labbate and Mazzuoccolo \cite{abreu2019colourings} who conjectured a vertex version of Wormald's conjecture, asserting that every cubic graph can be $2$-vertex-coloured so that the two colour classes induce isomorphic linear forest. This was proved for large connected cubic graph with large girth \cite{das2021isomorphic}, but it is open in general. Note that in \cite{das2021isomorphic} it is also proved that for large connected cubic graphs, with unrestricted girth, there is a $2$-vertex-colouring where the two colours induce isomorphic graphs, whose components are paths and cycles.

\begin{conjecture}[Conjecture 2.20 in \cite{abreu2019colourings}] \label{StrongAndo}
    Every cubic graph admits a 2-colouring of its vertex set such that the two colour classes induce isomorphic linear forests.
\end{conjecture}

Ban and Linial \cite{ban2016internal} stated an even stronger conjecture (but under some further restrictions on $G$).

\begin{conjecture}[Conjecture 1 in \cite{ban2016internal}] \label{conj:BanLinial}
Every bridgeless cubic graph, with the exception of the Petersen graph, can
be 2-vertex-coloured so that the two colour classes induce isomorphic graphs that consist of a matching and isolated vertices.
\end{conjecture}

This was proved for some specific cubic graphs~\cite{abreu2018note,ban2016internal} but it is still widely open in general, and seems hard.

We already stated a natural generalization of Wormald's conjecture to higher degrees, Conjecture~\ref{conj:wormaldgeneralized} which generalizes the linear arboricity conjecture.  One can weaken it by asking the following question.  Let $\la(d)$ be the minimum $k$ such that every $d$-regular graph can be decomposed into at most $k$ linear forests. 

\begin{question}
    Let $d \ge 3$ be an integer.
    Is it true that every large connected $d$-regular graph, whose number of edges is divisible by $\la(d)$, can be decomposed into $\la(d)$ isomorphic linear forests?
\end{question}

A similar and perhaps easier question about a balanced decomposition into graphs with maximum degree at most $2$. A simple consequence of Vizing's theorem is that every $d$-regular graph can be decomposed into $\ceil{(d+1)/2}$ graphs of maximum degree $2$. Here we ask if they can also be isomorphic. 

\begin{question}
    Let $d \ge 3$ be an integer.
    Is it true that every large connected $d$-regular graph, whose number of edges is divisible by $\ceil{(d+1)/2}$, can be decomposed into $\ceil{(d+1)/2}$ isomorphic graph of maximum degree $2$?
\end{question}
In this paper we showed that the answer is yes (to both questions) for cubic graphs. We believe this direction might be a potential avenue of further progress towards a solution of the linear arboricity conjecture.  Some of our methods may be applied towards attacking this conjecture, however, there are certain obstacles which would need to be resolved. One of them is to find good ``switchings'' which would allow us to break monochromatic  cycles or monochromatic ``long'' paths. In a cubic graph, it is easy to flip a colour of an edge and still maintain the property of monochromatic components being linear forests, but that is no longer the case when $d$ grows.

\bibliography{cubic}
\bibliographystyle{amsplain}
\end{document}